\documentclass[12pt,reqno]{amsart}

\makeatletter
\def\@tocline#1#2#3#4#5#6#7{\relax
  \ifnum #1>\c@tocdepth 
  \else
    \par \addpenalty\@secpenalty\addvspace{#2}%
    \begingroup \hyphenpenalty\@M
    \@ifempty{#4}{%
      \@tempdima\csname r@tocindent\number#1\endcsname\relax
    }{%
      \@tempdima#4\relax
    }%
    \parindent\z@ \leftskip#3\relax \advance\leftskip\@tempdima\relax
    \rightskip\@pnumwidth plus4em \parfillskip-\@pnumwidth
    #5\leavevmode\hskip-\@tempdima
      \ifcase #1
       \or\or \hskip 1em \or \hskip 2em \else \hskip 3em \fi%
      #6\nobreak\relax
    \hfill\hbox to\@pnumwidth{\@tocpagenum{#7}}\par
    \nobreak
    \endgroup
  \fi}
\makeatother

\usepackage{multirow}
\usepackage{hhline}

\usepackage{mathtools}
\usepackage{times}
\usepackage[T1]{fontenc}
\usepackage{mathrsfs}
\usepackage{latexsym}
\usepackage[dvips]{graphics}
\usepackage[titletoc, title]{appendix}
\usepackage{amsmath,amsfonts,amsthm,amssymb,amscd}
\usepackage{color}
\usepackage{hyperref}
\usepackage{amsmath}

\usepackage{color}
\usepackage{breakurl}

\usepackage{comment}
\newcommand{\bburl}[1]{\textcolor{blue}{\url{#1}}}

\newtheorem{thm}{Theorem}[section]

\newtheorem{cor}[thm]{Corollary}
\newtheorem{claim}[thm]{Claim}
\newtheorem{lem}[thm]{Lemma}
\newtheorem{prop}[thm]{Proposition}
\newtheorem{exa}[thm]{Example}

\newtheorem{defi}[thm]{Definition}
\newtheorem{rek}[thm]{Remark}

\DeclareMathOperator{\supp}{supp}
\DeclareMathOperator{\spann}{span}
\DeclareMathOperator{\sgn}{sgn}
\numberwithin{equation}{section}

\DeclareFontFamily{U}{mathx}{}
\DeclareFontShape{U}{mathx}{m}{n}{<-> mathx10}{}
\DeclareSymbolFont{mathx}{U}{mathx}{m}{n}
\DeclareMathAccent{\widehat}{0}{mathx}{"70}
\DeclareMathAccent{\widecheck}{0}{mathx}{"71}

\begin{document}

\title{Strong partially greedy bases with respect to an arbitrary sequence}

\author{H\`ung Vi\d{\^e}t Chu}

\email{\textcolor{blue}{\href{mailto:hungchu1@tamu.edu}{hungchu1@tamu.edu}}}
\address{Department of Mathematics, Texas A\&M University, College Station, TX 77843, USA}

\begin{abstract} 
For Schauder bases, Dilworth et al. introduced and characterized the partially greedy property, which is strictly weaker than the (almost) greedy property. Later,  Berasategui et al. defined and studied the strong partially greedy property for general bases. Let $\mathbf n$ be any strictly increasing sequence of positive integers. 
In this paper, we define the strong partially greedy property with respect to $\mathbf n$, called the ($\mathbf n$, strong partially greedy) property. We give characterizations of this new property, study relations among ($\mathbf n$, strong partially greedy) properties for different sequences $\mathbf n$, establish Lebesgue-type inequalities for the ($\mathbf n$, strong partially greedy) parameter, investigate ($\mathbf n$, strong partially greedy) bases with gaps, and weighted ($\mathbf n$, strong partially greedy) bases, to name a few. Furthermore, we introduce the ($\mathbf n$, almost greedy) property and equate the property to a strengthening of the ($\mathbf n$, strong partially greedy) property.  This paper can be viewed both as a survey of recent results regarding strong partially greedy bases and as an extension of these results to an arbitrary sequence instead of $\mathbb{N}$. 
\end{abstract}



\thanks{The author acknowledges the summer funding from the Department of Mathematics at the University of Illinois Urbana-Champaign. The author is thankful to Timur Oikhberg for helpful feedback on earlier drafts of this paper.}

\maketitle

\section{Introduction}

\subsection{Settings and classical results}
Let $X$ be a separable, infinite dimensional Banach space over the field $\mathbb{F} = \mathbb{R}$ or $\mathbb{C}$ with the dual space $X^*$. A countable set $\mathcal{B} = (e_n)_{n=1}^\infty\subset X$ is said to be a semi-normalized basis (or basis, for short) of $X$ if 
\begin{enumerate}
    \item[a)] $\spann\{e_n: n\in\mathbb{N}\}$ is norm-dense in $X$;
    \item[b)] there is a semi-normalized sequence $(e_n^*)_{n=1}^\infty\subset X^*$ such that $e_j^*(e_k) = \delta_{j, k}$ for all $j, k\in\mathbb{N}$;
    \item[c)] there exist $c_1, c_2 > 0$ such that 
    $$0 \ <\ c_1 := \inf_n\{\|e_n\|, \|e_n^*\|\}\ \le\ \sup_n\{\|e_n\|, \|e_n^*\|\} \ =:\ c_2 \ <\ \infty.$$
\end{enumerate}
For a basis $\mathcal{B}$, we represent each vector $x\in X$ with the formal series $\sum_{n\in\mathbb{N}} e_n^*(x)e_n$. It is possible that two different vectors have the same representation unless $\mathcal{B}$ is a Markushevich basis, which, in addition to items a), b), and c) above, also satisfies
\begin{enumerate}
    \item [d)] $\overline{\spann\{e_n^*: n\in\mathbb{N}\}}^{w^*} = X^*$.
\end{enumerate}
A Schauder basis satisfies items from a) to d), and there exists the least constant $\mathbf K_b$, called the basis constant, such that $\|S_m(x)\|\le \mathbf K_b\|x\|$, for all $x\in X$ and $m\in\mathbb{N}$, where $S_m(x) = \sum_{n=1}^m e_n^*(x)e_n$. When $\mathbf K_b = 1$, we say that $\mathcal{B}$ is monotone. 

By a density argument, it is easy to see that if $\mathcal{B}$ is a basis, then \begin{equation}\label{e11}\lim_{n\rightarrow\infty}|e_n^*(x)| \ =\ 0,\forall x\in X.\end{equation}
We now introduce notation, some of which have been frequently used in the literature: fixing $x\in X$, finite subsets $A, B$ of positive integers, and a strictly increasing sequence of positive integers $\mathbf n = n_1, n_2, \ldots$, we define 
\begin{enumerate}
    \item a sign $\varepsilon$ to be a sequence $(\varepsilon_n)_{n=1}^\infty$ of scalars of modulus $1$.
    \item $1_A = \sum_{n\in A}e_n$ and $1_{\varepsilon A} = \sum_{n\in A}\varepsilon_n e_n$, for a given sign $\varepsilon$. 
    \item $P_A(x)= \sum_{n\in A}e_n^*(x)e_n$ and $P_{A^c}(x) = x- P_A(x)$.
    \item $\|x\|_\infty = \max_{n}|e_n^*(x)|$ and $\supp(x) = \{n: e_n^*(x) \neq 0\}$.
    \item $A < B$ to mean that $a < b$ for all $a\in A, b\in B$. It holds vacuously that $\emptyset < A$ and $\emptyset > A$. Also, $n < A$ for a number $n$ means $\{n\} < A$.
    \item $A\sqcup B$ to mean $A\cap B = \emptyset$. Hence, $D < (A\sqcup B)\cap E$ for two sets $D, E$ means that $D < (A\cup B)\cap E$ and $A\cap B = \emptyset$.
    \item $\mathbb{T}(\mathbf n)$ to be the collection of all ordered pairs of finite sets $(A, B)$ such that $A\subset\mathbf n$, $|A|\le |B|$, and $A < B\cap \mathbf n$. 
    \item $\mathbb{S}(\mathbf n)$ to be the collection of all ordered pairs of finite sets $(A, B)$ such that $A\subset\mathbf n$ and $|A|\le |B|$. 
    \item for a number $a$, $A|_a = \{n\in A: n\ge a\}$. Furthermore, for sets $A, B, C$, $A < (B\sqcup C)|_{\min A}$ means that $B\sqcup C$ and either $A = \emptyset$ or $A < (B\cup C)|_{\min A}$.
    \item for $m\ge 0$, $\mathcal{I}^{(m)}: = \{A\subset\mathbb{N}: |A| = m\mbox{ and }A \mbox{ is an interval}\}$, $\mathcal{I}^{\le m}:= \cup_{1\le k\le m}\mathcal{I}^{(k)}$, and $\mathcal{I} := \cup_{m\ge 0}\mathcal{I}^{(m)}$.
\end{enumerate}

The limit \eqref{e11} enables the definition of 
the thresholding greedy algorithm (TGA), introduced by Konyagin and Temlyakov in 1999 \cite{KT1}. 
For each $x\in X$, the algorithm chooses the largest coefficients (in modulus) with respect to $\mathcal{B}$: a set $A\subset \mathbb{N}$ is a greedy set of $x$ of order $m$ if $|A| = m$ and 
$$\min_{n\in A}|e_n^*(x)|\ \ge\ \max_{n\notin A}|e_n^*(x)|.$$
The corresponding greedy sum is 
$G_m(x)\ :=\ P_A(x)$.
Let $G(x, m)$ denote the set of all greedy sets of $x$ of order $m$. The TGA thus produces a (possibly nonunique) sequence $(G_m(x))_{m=1}^\infty$ for each $x\in X$.

Konyagin and Temlyakov \cite{KT1} then defined and characterized greedy bases as being unconditional and democratic. In particular,

\begin{defi}\normalfont
A basis $\mathcal{B}$ is greedy if there exists $\mathbf C\ge 1$ such that
\begin{equation}\label{e6}\left\|x-G_m(x)\right\| \ \le\ \mathbf C\sigma_m(x),\forall x\in X, \forall m\in \mathbb{N}, \forall G_m(x),\end{equation}
where 
$$\sigma_m(x) \ :=\ \inf_{a_n\in\mathbb{F}}\left\{\left\|x-\sum_{n\in A}a_ne_n\right\| \,:\, |A| = m\right\}.$$
If $\mathbf C$ satisfies \eqref{e6}, we say that $\mathcal{B}$ is $\mathbf C$-greedy.
\end{defi}

\begin{defi}\label{d1}\normalfont A basis $\mathcal{B}$ is unconditional if there exists $\mathbf K \ge 1$ such that for all $N\in\mathbb{N}$, $$\left\|\sum_{n=1}^Na_ne_n\right\|\ \le\ \mathbf K\left\|\sum_{n=1}^N b_n e_n\right\|,$$
whenever $|a_n|\le |b_n|$ for all $1\le n\le N$. The least constant $\mathbf K$ is denoted by $\mathbf K_u$, called the unconditional constant of $\mathcal{B}$. Equivalently, there exists the so-called suppression-unconditional constant $\mathbf K_{su}$, which is the smallest constant such that $\|P_A(x)\|\le \mathbf K_{su}\|x\|$ for all $A\subset\mathbb{N}$.
\end{defi}

\begin{defi}\label{d2}\normalfont A basis $\mathcal B$ is  superdemocratic if there is $\mathbf C \ge 1$ such that
\begin{equation}\label{e12}\|1_{\varepsilon A}\|\ \le\ \mathbf C\|1_{\delta B}\|,\end{equation}
for all finite sets $A, B\subset \mathbb{N}$ with $|A| \le |B|$ and for all signs $\varepsilon, \delta$. Let $\mathbf C_{sd}$ be the smallest constant for \eqref{e12} to hold. If \eqref{e12} holds for $\varepsilon \equiv \delta \equiv 1$, then we say that $\mathcal{B}$ is $\mathbf C$-democratic; the smallest constant in this case is denoted by $\mathbf C_d$.
\end{defi}

\begin{thm}[Konyagin and Temlyakov \cite{KT1}]\label{KT1}
A basis $\mathcal{B}$ in a Banach space is greedy if and only if it is unconditional and democratic.
\end{thm}

Continuing the work, Dilworth et al.\ \cite{DKKT} defined and characterized the so-called almost greedy bases as being quasi-greedy and democratic. 

\begin{defi}\normalfont \label{defAG} A basis $\mathcal B$ is almost greedy if there exists $\mathbf C\ge 1$ such that
\begin{equation}\label{e7}\|x-G_m (x)\|\ \le\ \mathbf C\widetilde{\sigma}_m(x), \forall x\in X, \forall m\in \mathbb{N}, \forall G_m(x),\end{equation}
where $\widetilde{\sigma}_m(x) := \inf\{\|x-P_A(x)\|: |A| = m\}$. If $\mathbf C$ verifies \eqref{e7}, then $\mathcal{B}$ is said to be $\mathbf C$-almost greedy. 
\end{defi}

\begin{defi}\label{d3}\normalfont
A basis $\mathcal{B}$ is quasi-greedy if there exists $\mathbf C > 0$ such that
$$\|G_m(x)\|\ \le\ \mathbf C\|x\|,\forall x\in X, m\in\mathbb{N}, \forall G_m(x).$$
The least such $\mathbf C$ is denoted by $\mathbf C_q$, called the quasi-greedy constant. Also when $\mathcal{B}$ is quasi-greedy, let $\mathbf C_\ell$ be the least constant such that $$\|x-G_m(x)\|\ \le\ \mathbf \mathbf \mathbf C_\ell\|x\|,\forall x\in X, m\in\mathbb{N}, \forall G_m(x).$$
We call $\mathbf C_\ell$ the suppression quasi-greedy constant.
\end{defi}

\begin{thm}[Dilworth et al. \cite{DKKT}]\label{DKKTcha}
A basis $\mathcal{B}$ in a Banach space is almost greedy if and only if it is quasi-greedy and democratic. 
\end{thm}

\begin{rek}\normalfont
By definition, any greedy basis is almost greedy. However, \cite[Example 10.2.9]{AK} is an almost greedy basis that is not greedy. 
\end{rek}

Another popular greedy-type basis is the so-called partially greedy, also introduced by Dilworth et al. for Schauder bases \cite{DKKT}. 
Later, Berasategui et al.\ defined strong partially greedy bases for general bases \cite{BBL}. Both partially greedy Schauder bases and strong partially greedy bases can be characterized as being quasi-greedy and conservative. Since we work with general bases, we shall use the definition of strong partially greedy bases.

\begin{defi}\normalfont
A basis $\mathcal B$ is strong partially greedy if there exists $\mathbf C\ge 1$ such that
\begin{equation}\label{e13}\|x-G_m (x)\|\ \le\ \mathbf C\widehat{\sigma}_m(x), \forall x\in X, \forall m\in \mathbb{N}, \forall G_m(x).\end{equation}
where $\widehat{\sigma}_m(x) := \min_{0\le k\le m}\|x-P_{\{1, \ldots, k\}}(x)\|$. If $\mathbf C$ verifies \eqref{e7}, then $\mathcal{B}$ is said to be $\mathbf C$-strong partially greedy. 
\end{defi}

\begin{defi}\label{d5}\normalfont A basis $\mathcal B$ is  superconservative if there is $\mathbf C > 0$ such that
\begin{equation}\label{e15}\|1_{\varepsilon A}\|\ \le\ \mathbf C\|1_{\delta B}\|,\end{equation}
for all finite sets $A, B\subset \mathbb{N}$ with $|A| \le |B|$, $A < B$ and for all signs $\varepsilon, \delta$. Let $\mathbf C_{sc}$ be the smallest constant for \eqref{e15} to hold. If \eqref{e15} holds for $\varepsilon \equiv \delta \equiv 1$, then we say that $\mathcal{B}$ is $\mathbf C$-conservative; the smallest constant in this case is denoted by $\mathbf C_c$. 
\end{defi}

\begin{thm}\cite[Theorem 4.2]{B}\label{B}
A basis $\mathcal{B}$ is strong partially greedy if and only if it is quasi-greedy and conservative. 
\end{thm}

\subsection{Motivation and outline}
Throughout the paper, let $\mathbf n = n_1, n_2, \ldots$ be a strictly increasing sequence of positive integers. The first motivation comes from a natural extension of strong partially greedy bases.
While $\widehat{\sigma}_m(x)$ measures the distance between $x$ and the projection of $x$ onto the first $m$ vectors in $\mathcal{B}$, this paper investigates the resulting bases that satisfy \eqref{e13} with $\widehat{\sigma}_m(x)$ replaced by $\widehat{\sigma}^{\mathbf n}_m(x)$, where $\widehat{\sigma}^{\mathbf n}_m(x)$ measures the distance between $x$ and the projection of $x$ onto the first $m$ vectors in $(e_n)_{n\in\mathbf n}$. Formally, for each $x\in X$, let $P^{\mathbf n}_m(x) := \sum_{i=1}^m e_{n_i}^*(x)e_{n_i}$. Define
$$\widehat{\sigma}^{\mathbf n}_m(x)\ :=\ \min_{0\le k\le m}\|x-P^{\mathbf n}_k(x)\|.$$

\begin{defi}\label{d10}\normalfont
A basis $\mathcal{B}$ in a Banach space is said to be ($\mathbf n$, strong partially greedy) if there exists $\mathbf C\ge 1$ such that 
\begin{equation}\label{e1}\|x-G_m(x)\|\ \le\ \mathbf C\widehat{\sigma}_m^{\mathbf n}(x), \forall x\in X,\forall m\in\mathbb{N},\forall G_m(x).\end{equation}
The smallest constant $\mathbf {C}$ for \eqref{e1} to hold is denoted by $\mathbf C_{\mathbf n, sp}$. Note that 
when $\mathbf n = \mathbb{N}$, we have strong partially greedy bases. 
\end{defi}

The second motivation comes from recent work by Berasategui et al.\ \cite{BBC}, where the authors let
$$\widecheck{\sigma}_m (x)\ :=\ \inf \{\|x-P_I(x)\|: I\in\mathcal{I}, |I| = m\}$$
and define consecutive almost greedy bases of type I. 
\begin{defi}\normalfont\label{defcon}
A basis $\mathcal{B}$ is said to be consecutive almost greedy of type I (CAG(I)) if there exists $\mathbf C\ge 1$ such that 
\begin{equation}\label{ec1}\|x-G_m(x)\|\ \le\ \mathbf C\widecheck{\sigma}_m(x), \forall x\in X, \forall m\in\mathbb{N}, \forall G_m(x).\end{equation}
\end{defi}

\begin{thm}\cite[Theorem 1.7]{BBC}\label{rt1}
A basis $\mathcal{B}$ is CAG(I) if and only if 
$$\|x-G_m(x)\|\ \le\ \mathbf C\widecheck{\sigma}_k(x), \forall x\in X, \forall m\in\mathbb{N}, \forall k\le m.$$
\end{thm}

It can be seen from Definition \ref{defcon} that the notion of CAG(I) bases is a strengthening of the notion of strong partially greedy bases. Specifically, instead of projecting $x$ onto the first $m$ vectors in $\mathcal{B}$, we project $x$ onto any set of $m$ consecutive vectors arbitrarily to the right. We have the implications
$$\mbox{greedy} \Longrightarrow \mbox{almost greedy} \Longrightarrow \mbox{CAG(I)} \Longrightarrow \mbox{strong partially greedy}.$$
Surprisingly, Theorem \ref{rt1} states that a basis is CAG(I) if and only if it is almost greedy. By introducing ($\mathbf n$, strong partially greedy) bases, we would like to explore whether we still obtain almost greedy bases if we project $x$ onto vectors with consecutive indices in $\mathbf n$. We predict that the answer is negative at least in the case $\mathbb N\backslash \mathbf n$ is infinite, since then, there are certain parts of $\mathcal{B}$ that are never projected on. Then we would have a new type of greedy bases that lie strictly between being almost greedy and being ($\mathbf n$, strong partially greedy). Our prediction is confirmed by Proposition \ref{pe1} and Theorem \ref{mm5}.

The third motivation comes from the following property of unconditional bases of the classical space $\ell_p\oplus \ell_q$ for $1\le p < q <\infty$. \`{E}del'\v{s}te\v{i}n and Wojtaszczyk \cite{EW} showed that any unconditional basis $\mathcal{B}$ of the direct sum $\ell_p\oplus \ell_q$ consists of two subsequences, one of which, denoted by $\mathcal{B}_1$, is a basis for $\ell_p$, while the
other, denoted by $\mathcal{B}_2$, is a basis for $\ell_q$. By \cite[Proposition 2.1.3]{AK}, we can find subsequences of $\mathcal{B}_1$ and $\mathcal{B}_2$ that are equivalent to the canonical bases of $\ell_p$ and $\ell_q$, respectively. We now use the same reasoning as in \cite[Example 10.4.4]{AK} to conclude that $\mathcal{B}$ is not conservative. By Theorem \ref{B}, $\mathcal{B}$ is not strong partially greedy. In short, if $\mathcal{B}$ is an unconditional basis of $\ell_p\oplus \ell_q$, then $\mathcal{B}$ is not strong partially greedy. However, there exists an unconditional basis of $\ell_p\oplus \ell_q$ that is ($\mathbf n$, strong partially greedy) for some $\mathbf n$. For example, let $\mathcal{B} = (e_n)_n$ be the direct sum of the standard unit vector bases of the two spaces to have
$$\left\|\sum_{n=1}^{\infty}a_n e_n\right\|\ =\ \left(\sum_{k=0}^{\infty}|a_{2k+1}|^p\right)^{1/p} + \left(\sum_{k=1}^\infty|a_{2k}|^q\right)^{1/q}.$$
Let $\mathbf n = 2, 4, 6, 8, \ldots$ and $(A, B)\in \mathbb{T}(\mathbf n)$. Setting $B_1 = B\cap \mathbf n$ and $B_2 = B\backslash B_1$, we have
$$\|1_A\|\ \le\ |A|^{1/q} \mbox{ and }\|1_B\| \ =\ |B_1|^{1/q} + |B_2|^{1/p}.$$
If $|B_1|\ge |A|$, then $\|1_B\|\ge \|1_A\|$. If $|B_1| < |A|$, then 
$$\|1_A\|\ \le\ |A|^{1/q} \ \le\ |B_1|^{1/q} + (|A|-|B_1|)^{1/q} \ \le\ |B_1|^{1/q} + |B_2|^{1/q}\ \le\ \|1_B\|.$$
We have shown that $\|1_A\|\le \|1_B\|$ for all $(A, B)\in\mathbb{T}(\mathbf n)$. As we shall see later, it follows that $\mathcal{B}$ is ($\mathbf n$, strong partially greedy). 

We study various aspects of ($\mathbf n$, strong partially greedy) bases in this paper, which can be viewed both as a survey of recent results regarding strong partially greedy bases and as an extension of these results to an arbitrary sequence instead of $\mathbb{N}$. The paper structure is as follows: 
\begin{itemize}
    \item Section \ref{characterization} defines ($\mathbf n$, (super)conservative bases) and ($\mathbf n$, partial symmetry for largest coefficients) (or ($\mathbf n$, PSLC), for short) and characterizes ($\mathbf n$, strong partially greedy) bases in the same manner as Theorems \ref{KT1},  \ref{DKKTcha}, and \ref{B}. 
    \item Section \ref{differentSeq} utilizes results in Section \ref{characterization} to examine when an ($\mathbf m$, strong partially greedy) basis is ($\mathbf n$, strong partially greedy) for two different sequences $\mathbf m$ and $\mathbf n$. In particular, let $\Delta_{\mathbf m, \mathbf n}$ be the difference set of the two sequences $\mathbf m$ and $\mathbf n$. We prove that the ($\mathbf m$, strong partially greedy) property is equivalent to the ($\mathbf n$, strong partially greedy) property if and only if $\Delta_{\mathbf m, \mathbf n}$ is finite.
    \item Section \ref{Lebesgue} establishes some Lebesgue-type estimates (previously studied in \cite{BBL, BBG, BBGHO, C0}) for the ($\mathbf n$, strong partially greedy) parameter, denoted by $\widehat{\mathbf L}^{\mathbf n}_m$. Our results extend several inequalities proved in \cite{BBL}. We also give examples, which are modifications of \cite[Examples 3.4 and 3.5]{BBL}, to show the optimality of the estimates.
    \item Section \ref{1PSLC} gives two characterizations of $1$-($\mathbf n$, strong partially greedy) bases. Furthermore, for any sequences $\mathbf m, \mathbf n$ with infinite $\Delta_{\mathbf m, \mathbf n}$, we construct a $1$-($\mathbf n$, strong partially greedy) basis that is not ($\mathbf m$, conservative). This strengthens \cite[Example 4.3]{BBL}.
    \item Section \ref{consecutive} introduces ($\mathbf n$, almost greedy) bases, which can be shown to be equivalent to a strengthening of the ($\mathbf n$, strong partially greedy) property. For two sequences $\mathbf m$ and $\mathbf n$ such that $\{i: i\in \mathbf m, i\notin \mathbf n\}$, we prove that there exists an ($\mathbf n$, almost greedy) basis that is not ($\mathbf m$, almost greedy). However, if $\{i: i\in\mathbf m, i\notin \mathbf n\}$ is finite, then any ($\mathbf n$, almost greedy) basis is ($\mathbf m$, almost greedy). 
    \item Section \ref{gaps} studies ($\mathbf n$, strong partially greedy) bases with gaps including their implications and characterizations. We generalize several results in \cite{BB1} to the case of an arbitrary sequence $\mathbf n$.
    \item Section \ref{largersum} investigates bases that satisfy \eqref{e1} with $G_m(x)$ replaced by $G_{\lceil\lambda m\rceil}(x)$. The replacement by a larger greedy sum was started by Dilworth et al. \cite{DKKT}, followed by the author of the present paper in \cite{C1}.
    \item Section \ref{weight} discusses weighted ($\mathbf n$, strong partially greedy) bases. Let $\zeta = (s_n)_{n}\in (0,\infty)^{\mathbb N}$. For $\mathbf n\neq \mathbb{N}$, we prove the following distinction between $\zeta$-strong partially greedy bases and $\zeta$-($\mathbf n$, strong partially greedy) bases: while there is a $\zeta$ such that all quasi-greedy bases are $\zeta$-strong partially greedy, there is no $\zeta$ such that all quasi-greedy bases are $\zeta$-($\mathbf n$, strong partially greedy). 
\end{itemize}

It is worth mentioning that while all of our results are concerned with strong partially greedy bases, we expect their analogs to hold for reverse partially greedy bases (introduced by Dilworth and Khurana \cite{DK}) due to various evidence that the two types of bases are companions \cite{B, C3, DK, K}.

\section{Characterizations of ($\mathbf n$, strong partially greedy) bases}\label{characterization}

\subsection{Truncation operator, (super)conservative bases, and PSLC}
First, we recall the uniform boundedness of the truncation operator for quasi-greedy bases. Fixing $\alpha > 0$, we define the truncation function $T_\alpha$ as follows: for  $b\in\mathbb{F}$,
$$T_{\alpha}(b)\ =\ \begin{cases}\sgn(b)\alpha, &\mbox{ if }|b| > \alpha,\\ b, &\mbox{ if }|b|\le \alpha.\end{cases}$$
The truncation operator $T_\alpha: X\rightarrow X$ is defined as 
$$T_{\alpha}(x)\ =\ \sum_{n=1}^\infty T_\alpha(e_n^*(x))e_n \ =\ \alpha 1_{\varepsilon \Gamma_{\alpha}(x)}+ P_{\Gamma_\alpha^c(x)}(x),$$
where $\Gamma_\alpha(x) = \{n: |e_n^*(x)| > \alpha\}$ and $\varepsilon = (\sgn(e_n^*(x)))$.

\begin{thm}\label{bto}\cite[Lemma 2.5]{BBG} Let $\mathcal{B}$ be a $\mathbf{C}_\ell$-suppression quasi-greedy of a Banach space. Then for any $\alpha > 0$, $\|T_\alpha\|\le \mathbf{C}_\ell$.
\end{thm}

\begin{defi}\label{d11}\normalfont
A basis $\mathcal{B}$ in a Banach space is said to be ($\mathbf n$, superconservative) if there exists $\mathbf C > 0$ such that
\begin{equation}\label{e2}\|1_{\varepsilon A}\|\ \le\ \mathbf C\|1_{\delta B}\|,\end{equation}
for all $(A, B)\in \mathbb{T}(\mathbf n)$ and for all signs $\varepsilon, \delta$.
The smallest constant $\mathbf C$ for \eqref{e2} to hold is denoted by $\Delta_{\mathbf n, sc}$. If \eqref{e2} holds for $\varepsilon \equiv \delta \equiv 1$, then we say that $\mathcal{B}$ is ($\mathbf n$, conservative), and the smallest $\mathbf C$ in this case is denoted by $\Delta_{\mathbf n, c}$. Note that ``($\mathbb{N}$, superconservative)" is the same as ``superconservative".
\end{defi}

Next, we introduce the notion of ($\mathbf n$, partial symmetry for largest coefficients) (or ($\mathbf n$, PSLC), for short). 

\begin{defi}\label{pslc}\normalfont
A basis $\mathcal{B}$ in a Banach space is said to be ($\mathbf n$, PSLC) if there exists $\mathbf C \ge 1$ such that
\begin{equation}\label{e3}\|x+ 1_{\varepsilon A}\|\ \le\ \mathbf C\|x + 1_{\delta B}\|,\end{equation}
for all $x\in X$ with $\|x\|_\infty \le 1$, for all $(A, B)\in \mathbb{T}(\mathbf n)$ with $A < (B\sqcup \supp(x))\cap \mathbf n$, and for all signs $\varepsilon, \delta$. The smallest $\mathbf C$ for which \eqref{e3} holds is denoted by $\Delta_{\mathbf n, pl}$. When $\mathbf n = \mathbb{N}$, we say that $\mathcal{B}$ is PSLC as in \cite[Definition 1.9]{BBL}.
\end{defi}

\begin{rek}\label{r1}\normalfont
\begin{enumerate}
\item[i)] Clearly, ($\mathbf n$, PSLC) $\Longrightarrow$ ($\mathbf n$, superconservative).
\item[ii)] It is well-known that a quasi-greedy basis has the UL property: there exist $\mathbf C_1, \mathbf C_2 > 0$ such that 
$$\frac{1}{\mathbf C_1}\min |a_n|\|1_A\|\ \le\ \left\|\sum_{n\in A}a_ne_n\right\|\ \le\ \mathbf C_2\max|a_n|\|1_A\|,$$
which was first proved in \cite{DKKT}. It follows that a quasi-greedy and ($\mathbf n$, conservative) basis is ($\mathbf n$, superconservative).
\end{enumerate}
\end{rek}

\begin{lem}\label{l10}
If $\mathcal{B}$ is $\mathbf C_q$-quasi-greedy and $\Delta_{\mathbf n, sc}$-($\mathbf n$, superconservative), then $\mathcal{B}$ is $\Delta_{\mathbf n, pl}$-($\mathbf n$, PSLC) with $\Delta_{\mathbf n, pl}\le 1+\mathbf C_q + \Delta_{\mathbf n, sc}\mathbf C_q$. 
\end{lem}
\begin{proof}
Let $x, A, B, \varepsilon,\delta$ be chosen as in Definition \ref{pslc}. We have
$$\|x\|\ \le\ \|x+1_{\delta B}\| + \|1_{\delta B}\|\ \le\ (1+\mathbf C_q)\|x+1_{\delta B}\|.$$
Furthermore, 
$$\|1_{\varepsilon A}\|\ \le\ \Delta_{\mathbf n, sc}\|1_{\delta B}\|\ \le\ \Delta_{\mathbf n, sc}\mathbf C_q\|x+1_{\delta B}\|.$$
By the triangle inequality, we get
$$\|x+1_{\varepsilon A}\|\ \le\ \|x\| + \|1_{\varepsilon A}\|\ \le\ (1+\mathbf C_q + \Delta_{\mathbf n, sc}\mathbf C_q)\|x+1_{\delta B}\|.$$
This completes our proof. 
\end{proof}

\subsection{Characterizations of ($\mathbf n$, strong partially greedy) bases}

We state the main result of this section. 

\begin{thm}\label{m1}
Let $\mathcal{B}$ be a basis in a Banach space. The following are equivalent:
\begin{enumerate}
    \item [i)] $\mathcal{B}$ is ($\mathbf n$, strong partially greedy).
    \item [ii)] $\mathcal{B}$ is quasi-greedy and ($\mathbf n$, PSLC).
    \item [iii)] $\mathcal{B}$ is quasi-greedy and ($\mathbf n$, superconservative).
    \item [iv)] $\mathcal{B}$ is quasi-greedy and ($\mathbf n$, conservative).
\end{enumerate}
\end{thm}

\begin{prop}\label{p1}
If a basis $\mathcal{B}$ is $\mathbf C_{\mathbf n, sp}$-($\mathbf n$, strong partially greedy), then $\mathcal{B}$ is 
\begin{enumerate}
\item[i)] $\mathbf C_\ell$-suppression quasi-greedy with $\mathbf C_\ell\le \mathbf C_{\mathbf n, sp}$.
\item[ii)] $\Delta_{\mathbf n, pl}$-($\mathbf n$, PSLC) with $\Delta_{\mathbf n, pl}\le \mathbf C_{\mathbf n, sp}$.
\end{enumerate}
\end{prop}

\begin{proof}
i) From \eqref{e1}, we obtain
$$\|x-G_m(x)\|\ \le\ \mathbf C_{\mathbf n, sp}\|x\|,\forall x\in X, \forall m\in\mathbb{N}, \forall G_m(x).$$
This yields the desired conclusion. 

ii) Choose $x, A, B, \varepsilon, \delta$ as in Definition \ref{pslc}. Let $n_m = \max A$ and $D =  \{n_1, \ldots, n_m\}\backslash A$. Then 
$$|D\cup B|\ =\ |B| + |D|\ \ge\ |A| + |D|\ =\ m.$$
Set
$$y\ :=\ 1_{\varepsilon A} + 1_D + x + 1_{\delta B}.$$ Since $D\cup B$ is a greedy sum of $y$ of order at least $m$, we have
$$\|x+1_{\varepsilon A}\|\ =\ \|y-P_{D\cup B}(y)\|\ \le\ \mathbf C_{\mathbf n, sp}\|y-P^{\mathbf n}_{m}(y)\|\ =\ \mathbf C_{\mathbf n, sp}\|x+1_{\delta B}\|.$$
This completes our proof. 
\end{proof}

\begin{prop}\label{p2}
If a basis $\mathcal{B}$ is $\mathbf C_\ell$-suppression quasi-greedy and $\Delta_{\mathbf n, pl}$-($\mathbf n$, PSLC), then it is $\mathbf C_{\mathbf n, sp}$-($\mathbf n$, strong partially greedy) with $\mathbf C_{\mathbf n, sp}\le \mathbf C_\ell\Delta_{\mathbf n, pl}$.
\end{prop}

Before proving Proposition \ref{p2}, we need the following lemma.
\begin{lem}\label{l2}
A basis $\mathcal{B}$ is $\Delta_{\mathbf n, pl}$-($\mathbf n$, PSLC) if and only if 
\begin{equation}\label{e5}\|x\|\ \le\ \Delta_{\mathbf n, pl}\|x-P_A(x) + 1_{\varepsilon B}\|,\end{equation}
for all $x\in X$ with $\|x\|_\infty\le 1$, for all signs $\varepsilon$, and for all $(A, B)\in \mathbb{T}(\mathbf n)$ with $A < (\supp(x-P_A(x))\sqcup B)\cap \mathbf n$.
\end{lem}

\begin{proof}
Suppose that $\mathcal{B}$ satisfies \eqref{e5}. Choose $x, A, B, \varepsilon, \delta$ as in Definition \ref{pslc}. Let $y = x+ 1_{\varepsilon A}$. We have
$$\|x+1_{\varepsilon A}\|\ =\ \|y\| \ \stackrel{\eqref{e5}}{\le}\ \Delta_{\mathbf n, pl}\|y - P_A(y) + 1_{\delta B}\|\ =\ \Delta_{\mathbf n, pl}\|x+1_{\delta B}\|.$$

Conversely, suppose that $\mathcal{B}$ is $\Delta_{\mathbf n, pl}$-($\mathbf n$, PSLC). Choose $x, A, B, \varepsilon$ as in \eqref{e5}. We have
\begin{align*}
    \|x\|\ =\ \left\|x-P_A(x) + \sum_{n\in A}e_n^*(x)e_n\right\|&\ \le\ \sup_{(\delta)}\|x-P_A(x)+1_{\delta A}\|\\
    &\ \le\ \Delta_{\mathbf n, pl}\|x-P_A(x) + 1_{\varepsilon B}\|.
\end{align*}
This completes our proof. 
\end{proof}

\begin{proof}[Proof of Proposition \ref{p2}]
Let $x\in X$, $m\in \mathbb{N}$ and $A\in G(x, m)$. Fix $k\le m$. We need to show that 
$$\|x-P_A(x)\|\ \le\ \mathbf C_\ell \Delta_{\mathbf n, pl}\|x-P^{\mathbf n}_k(x)\|.$$
Let $E = \{n_1, n_2, \ldots, n_k\}\backslash A$, $F = A\backslash \{n_1, n_2, \ldots, n_k\}$, and $\alpha = \min_{n\in A}|e_n^*(x)|$. We verify that 
$x-P_A(x)$, $E$, and $F$ satisfy the condition of Lemma \ref{l2}:
\begin{enumerate}
    \item[i)] by definition, $E\subset \mathbf n$;
    \item[ii)] that $k\le m$ implies that $|E|\le |F|$;
    \item[iii)] by definition, $\supp(x-P_A(x)-P_E(x))\sqcup F$; $E < F\cap \mathbf n$; $E <\supp(x-P_A(x)-P_E(x))\cap \mathbf n$.
\end{enumerate}
By Lemma \ref{l2} and Theorem \ref{bto}, we get
\begin{align*}
    \|x-P_A(x)\|&\ \le\ \Delta_{\mathbf n, pl}\left\|x-P_A(x)-P_E(x)+\alpha\sum_{n\in F}\sgn(e_n^*(x))e_n\right\|\\
    &\ \le\ \Delta_{\mathbf n, pl}\left\|T_\alpha(x-P_A(x)-P_E(x)+P_F(x))\right\|\\
    &\ \le\ \Delta_{\mathbf n, pl}\mathbf C_\ell\|x-P^{\mathbf n}_k(x)\|,
\end{align*}
as desired. 
\end{proof}

\begin{proof}[Proof of Theorem \ref{m1}]
By Propositions \ref{p1} and \ref{p2}, we have i) $\Longleftrightarrow$ ii). By Remark \ref{r1}, ii) $\Longrightarrow$ iii) and iii) $\Longleftrightarrow$ iv). Finally, Lemma \ref{l10} gives iii) $\Longrightarrow$ ii).
\end{proof}

\section{Strong partially greedy bases with respect to different sequences}\label{differentSeq}

In this section, we answer the question of when an ($\mathbf m$, strong partially greedy) basis is necessarily ($\mathbf n$, strong partially greedy) for two arbitrary sequences $\mathbf m$ and $\mathbf n$. First, we define ($\mathbf n$, (super)democratic), which is stronger than ($\mathbf n$, conservative). 

\begin{defi}\normalfont A basis $\mathcal{B}$ in a Banach space is said to be ($\mathbf n$, superdemocratic) if there exists $\mathbf C \ge 1$ such that
\begin{equation}\label{e20}\|1_{\varepsilon A}\|\ \le\ \mathbf C\|1_{\delta B}\|,\end{equation}
for all $(A, B)\in \mathbb{S}(\mathbf n)$ and for all signs $\varepsilon, \delta$.
The smallest constant $\mathbf C$ for \eqref{e20} to hold is denoted by $\Delta_{\mathbf n, sd}$. If \eqref{e20} holds for $\varepsilon \equiv \delta \equiv 1$, then we say that $\mathcal{B}$ is ($\mathbf n$, democratic), and the smallest $\mathbf C$ in this case is denoted by $\Delta_{\mathbf n, d}$. 
\end{defi}

\subsection{($\mathbf n$, strong partially greedy) but not strong partially greedy bases}

\begin{thm}\label{m2}
Let $\mathbf n$ be any strictly increasing sequence such that $\mathbb{N}\backslash \mathbf n$ is infinite. There exists a $1$-unconditional basis $\mathcal{B}$ that is ($\mathbf n$, strong partially greedy) but is not conservative and thus, not strong partially greedy. 
\end{thm}

\begin{proof}
Define the sequence of weights $w_n = \frac{1}{\sqrt{n}}$ for $n\ge 1$. 
Let $X$ be the completion of $c_{00}$ with respect to the following norm: for $x = (x_1, x_2, \ldots)$,
$$\|x\|\ =\ \sup_{\pi}\sum_{i}w_{\pi(i)}|x_{n_i}| + \sum_{n\notin\mathbf n}|x_n|,$$
where $\pi$ is a bijection on $\mathbb{N}$.
Let $\mathcal{B}$ be the canonical basis. Then $\mathcal{B}$ is normalized and $1$-unconditional. We prove the following:
\begin{enumerate}
    \item [i)] $\mathcal{B}$ is $1$-($\mathbf n$, superdemocratic) and thus, is ($\mathbf n$, strong partially greedy).
    \item [ii)] $\mathcal{B}$ is not conservative and thus, is not strong partially greedy. 
\end{enumerate}

i) For all finite sets $(A, B)\in \mathbb{S}(\mathbf n)$ and for all signs $\varepsilon, \delta$, we have $$\|1_{\varepsilon A}\| \ =\ \sum_{i=1}^{|A|}\frac{1}{\sqrt{i}}\mbox{ and }\|1_{\delta B}\|\ =\ \sum_{i=1}^{|B_1|}\frac{1}{\sqrt{i}} + |B_2|,$$
where $B_1 = B\cap \mathbf n$ and $B_2 = B\backslash B_1$. If $|B_1| \ge |A|$, then $\|1_{\varepsilon A}\|\le\|1_{\delta B}\|$. If $|B_1| < |A|$, then
$$\|1_{\varepsilon A}\|\ =\ \sum_{i=1}^{|B_1|}\frac{1}{\sqrt{i}} + \sum_{|B_1|+1}^{|A|}\frac{1}{\sqrt{i}}\ \le\ \sum_{i=1}^{|B_1|}\frac{1}{\sqrt{i}} + |A|-|B_1|\ \le\ \sum_{i=1}^{|B_1|}\frac{1}{\sqrt{i}} + |B_2|\ =\ \|1_{\delta B}\|.$$

ii) Fix $N\in\mathbb{N}$. Choose $N$ numbers in $\mathbb{N}\backslash \mathbf n$ to form a set $D$. Let $E$ consists of $N$ numbers in $\mathbf n$ such that $E > D$. While $\|1_D\| = N$, $\|1_E\| \le 2\sqrt{N}$. Since $\|1_D\|/\|1_E\|\rightarrow \infty$ as $N\rightarrow\infty$, we know that $\mathcal{B}$ is not conservative. 
\end{proof}

\subsection{($\mathbf n$, strong partially greedy) but not ($\mathbf m$, strong partially greedy) bases}

Let $\mathbf m = m_1, m_2, \ldots$ be another strictly increasing sequence. The example in Theorem \ref{m2} is ($\mathbf n$, democratic). In the next example, we give a basis that is ($\mathbf n$, conservative) but is not ($\mathbf n$, democratic). Also, we can extend Theorem \ref{m2} to two sequences $\mathbf m$ and $\mathbf n$:

\begin{thm}\label{m3}
Let $\mathbf m,\mathbf n$ be two strictly increasing sequences such that the set $\{i: i\in\mathbf m, i\notin \mathbf n\}$ is infinite. There exists a $1$-unconditional basis $\mathcal{B}$ that is 
\begin{enumerate}
\item[i)] ($\mathbf n$, strong partially greedy).
\item[ii)] not ($\mathbf m$, strong partially greedy).
\item[iii)] not ($\mathbf n$, democratic). 
\end{enumerate}
\end{thm}

\begin{rek}\normalfont
The basis we construct for Theorem \ref{m3} is in fact $1$-($\mathbf n$, strong partially greedy), which will be proved in Section \ref{1PSLC} when we have characterizations of $1$-($\mathbf n$, strong partially greedy) bases.   
\end{rek}

\begin{proof}[Proof of Theorem \ref{m3}]
Let $\mathcal{F}:= \{F\subset\mathbb{N}: \sqrt{\min F}\ge |F|\}$ and $w_n = \frac{1}{\sqrt{n}}$ for $n\ge 1$. Write $\mathbf n = n_1, n_2, \ldots$. Define $f: \mathbf n\rightarrow\mathbb{N}$ by $f(n_k) = k$. Let $X$ be the completion of $c_{00}$ under the following norm: for $x = (x_1, x_2, \ldots)$,
$$\|x\|\ =\ \sup_{\substack{F\in\mathcal{F}\\ \pi}}\sum_{i\in F}w_{\pi(i)}|x_{n_i}| + \sum_{n\notin \mathbf n}|x_n|,$$
where $\pi$ is a bijection on $\mathbb{N}$.
Let $\mathcal{B}$ be the canonical basis, which is $1$-unconditional. 

i) ($\mathbf n$, strong partially greedy): we need only to show that $\mathcal{B}$ is $1$-($\mathbf n$, conservative). Choose $(A, B)\in\mathbb{T}(\mathbf n)$. Pick $F\in \mathcal{F}$ with $F\subset f(A)$ and a bijection $\pi$ on $\mathbb{N}$. We have
$$\sum_{i\in F}w_{\pi(i)}|e^*_{n_i}(1_A)|\ \le\ \sum_{i=1}^{|F|}\frac{1}{\sqrt{i}}.$$
Set $B_1 = B\cap \mathbf n$ and $B_2 = B\backslash B_1$. 

Case 1: $|B_1|\ge |F|$. Let $F' \subset f(B_1)$ such that $|F'| = |F|$. Since $B\cap \mathbf n > A$, we know that $ F' >  F$ and so, $F'\in \mathcal{F}$. Choose a bijection $\sigma$ on $\mathbb{N}$ such that $\sigma(F') = \{1, \ldots, |F'|\}$. We obtain
$$\sum_{i=1}^{|F|}\frac{1}{\sqrt{i}} \ =\ \sum_{i=1}^{|F'|}\frac{1}{\sqrt{i}} \ =\ \sum_{i\in F'}w_{\sigma(i)}|e^*_{n_i}(1_{B_1})|\ \le\ \|1_B\|.$$

Case 2: $|B_1| < |F|$. Since $|B_1| + |B_2|\ge |A|$, we know that
$|B_2| > |A| - |F|$. Let $F' = f(B_1)$, so $F' > F$ and $|F'|< |F|$. Hence, $F'\in\mathcal{F}$. Letting $\sigma$ be a bijection on $\mathbb{N}$ such that $\sigma(F') = \{1, \ldots, |F'|\}$, we obtain
\begin{align*}\|1_B\|&\ \ge\ \sum_{i\in F'}w_{\sigma(i)}|e^*_{n_i}(1_{B_1})| + |B_2|\\
& \ =\ \sum_{i=1}^{|B_1|}\frac{1}{\sqrt{i}} + |B_2|\ \ge\ \sum_{i=1}^{|B|}\frac{1}{\sqrt{i}}\ \ge\ \sum_{i=1}^{|A|}\frac{1}{\sqrt{i}}\ \ge\ \sum_{i=1}^{|F|}\frac{1}{\sqrt{i}}.\end{align*}

In both cases, we get $\|1_B\|\ge \sum_{i=1}^{|F|}\frac{1}{\sqrt{i}} = \sum_{i\in F}w_{\pi(i)}|e_{n_i}^*(1_A)|$. Let $F\in\mathcal{F}$ and $\pi$ vary to conclude that $\|1_B\|\ge \|1_A\|$.

ii) not ($\mathbf m$, strong partially greedy): we show that $\mathcal{B}$ is not ($\mathbf m$, conservative).
Let $D := \{i: i\in \mathbf m, i\notin \mathbf n\}$ and $N\in\mathbb{N}$. Since $D$ is infinite, we can choose $E\subset D$ with $|E| = N$. Also choose $F\subset \mathbf n$ such that $F > E$ and $|F| = N$. Then 
$\|1_E\|  = N$, while $\|1_F\| \le \sum_{i=1}^{|F|}\frac{1}{\sqrt{i}}\le 2\sqrt{N}$. Hence, $\|1_E\|/\|1_F\|\rightarrow\infty$ as $N\rightarrow\infty$. We conclude that $\mathcal{B}$ is not ($\mathbf m$, conservative).

iii) not ($\mathbf n$, democratic): Let $N\in\mathbb{N}$. Write $\mathbf n = n_1, n_2, \ldots$. Choose $E = \{n_1, \ldots, n_N\}$ and $F = \{n_{N^2+1}, n_{N^2+2}, \ldots, n_{N^2+N}\}$. Clearly, $\|1_F\| = \sum_{i=1}^N \frac{1}{\sqrt{i}}\ge 2(\sqrt{N+1}-1)$. We find an upper bound for $\|1_E\|$. Let $G\in\mathcal{F}$ with $G \subset \{1, \ldots, N\}$ and a bijection $\pi$ on $\mathbb{N}$. Set $m = \min G$.

Case 1: $\sqrt{m}\ge N-m+1$. Then $m\ge N-\frac{\sqrt{4N+5}}{2}+\frac{3}{2}$ and $|G|\le N-m+1$. Hence, 
$$\sum_{i\in G}w_{\pi(i)}|e^*_{n_i}(1_E)|\ \le\ \sum_{i=1}^{N-m+1}\frac{1}{\sqrt{i}}\ \le\ 2\sqrt{N-m+1}\ \le\ \sqrt{2}\sqrt{\sqrt{4N+5}-1}.$$

Case 2: $\sqrt{m} < N-m+1$. Then $m < N - \frac{\sqrt{4N+5}}{2} + \frac{3}{2}$ and $|G|\le \sqrt{m}$. Hence,  
$$\sum_{i\in G}w_{\pi(i)}|e^*_{n_i}(1_E)|\ \le\ \sum_{i=1}^{\lfloor\sqrt{m}\rfloor}\frac{1}{\sqrt{i}}\ \le\ 2\sqrt[4]{m}\ \le\ 2\sqrt[4]{N - \frac{\sqrt{4N+5}}{2} + \frac{3}{2}}.$$

We conclude that $\|1_E\| \lesssim \sqrt[4]{N}$ and so, $\|1_F\|/\|1_E\|\rightarrow\infty$ as $N\rightarrow\infty$. Therefore, $\mathcal{B}$ is not ($\mathbf n$, democratic).
\end{proof}

The norm in the proof of Theorem \ref{m3} is independent of $\mathbf m$. We can, therefore, have a stronger version of Theorem \ref{m3}.
\begin{thm}
Let $\mathbf n$ be a strictly increasing sequence. Then there exists a $1$-unconditional basis $\mathcal{B}$ that is 
\begin{enumerate}
\item[i)] ($\mathbf n$, strong partially greedy).
\item[ii)] not ($\mathbf m$, strong partially greedy) for all strictly increasing sequence $\mathbf m$ with $\{i:i\in\mathbf m, i\notin\mathbf n\}$ is infinite.
\item[iii)] not ($\mathbf n$, democratic). 
\end{enumerate}
\end{thm}

We now investigate whether the conclusion in Theorem \ref{m3} holds if we drop the condition ``$\{i: i\in \mathbf m, i\notin \mathbf n\}$ is infinite". The answer depends on the cardinality of the set ``$D = \{i: i\in\mathbf n, i\notin \mathbf m\}$". If $D$ is infinite, then we have the same conclusion as in Theorem \ref{m3}.

\begin{thm}\label{m4}
Let $\mathbf m,\mathbf n$ be two strictly increasing sequences such that the set $\{i: i\in\mathbf n, i\notin \mathbf m\}$ is infinite. There exists a $1$-unconditional basis $\mathcal{B}$ that is 
\begin{enumerate}
\item[i)] ($\mathbf n$, strong partially greedy).
\item[ii)] not ($\mathbf m$, strong partially greedy).
\item[iii)] not ($\mathbf n$, democratic). 
\end{enumerate}
\end{thm}

\begin{proof}
If $\{i: i\in \mathbf m, i\notin \mathbf n\}$ is infinite, then by Theorem \ref{m3}, we are done. Assume, for the rest of the proof, that $\{i: i\in \mathbf m, i\notin \mathbf n\}$ is finite. 
Let $D = \{i: i\in \mathbf n, i\notin \mathbf m\} = \{d_1, d_2, \ldots\}$, for $d_1 < d_2 < \cdots$, $\mathcal{F} = \{F\subset\mathbb{N}: \sqrt{\min F}\ge |F|\}$, and $w_n = \frac{1}{\sqrt{n}}$ for $n\ge 1$. Write $\mathbf n = n_1, n_2, \ldots$ and $\mathbf m = m_1, m_2, \ldots$. Let $X$ be the completion of $c_{00}$ with respect to the following norm: for $x = (x_1, x_2, x_3, \ldots)$, 
$$\|x\|\ =\ \sup_{\substack{\pi, \sigma\\ F, F'\in \mathcal{F}}}\left(\sum_{i\in F}w_{\pi(i)}|x_{m_i}| + \sum_{i\in F'}w_{\sigma(i)}|x_{d_i}| + \sum_{i\notin \mathbf n\cup \mathbf m}|x_i|\right),$$
where $\pi, \sigma$ are bijections on $\mathbb{N}$.
Define $f: \mathbf m\rightarrow\mathbb{N}$ such that $f(m_k) = k$ and $g: D\rightarrow \mathbb{N}$ such that $f(d_k) = k$. Let $\mathcal{B}$ be the canonical basis. 

i) ($\mathbf n$, strong partially greedy): we need only to show that $\mathcal{B}$ is ($\mathbf n$, conservative). Let $(A, B)\in \mathbb{T}(\mathbf n)$ and $M:= |\{i: i\in\mathbf m, i\notin \mathbf n\}|$. If $|A|\le |B|\le 12M$, then 
$$\|1_A\|\ \le\ 12M\sup_{n}\|e_n\|\ \le\ 12M\sup_{n}\|e_n\|\sup_{m}\|e^*_m\|\|1_B\|\ \le\ 12Mc_2^2\|1_B\|.$$ Assume that $|B| > 12M$.
Set $A_1 = A\cap \mathbf m$ and $A_2 = A\backslash A_1$. Pick bijections $\pi, \sigma: \mathbb{N}\rightarrow\mathbb{N}$ and $F, F'\in\mathcal{F}$ such that $F\subset f(A_1)$ and $F'\subset g(A_2)$. Then
$$\sum_{i\in F}w_{\pi(i)}|e^*_{m_i}(1_A)| + \sum_{i\in F'}w_{\sigma(i)}|e^*_{d_i}(1_A)|\ \le\ \sum_{i=1}^{|F|}\frac{1}{\sqrt{i}} + \sum_{i=1}^{|F'|}\frac{1}{\sqrt{i}}\ \le\ 2\sum_{i=1}^{\ell}\frac{1}{\sqrt{i}},$$
where $\ell =  \max\{|F|, |F'|\}$.
We estimate $\|1_B\|$. Set $B_1 = B\cap \mathbf m$, $B_2 = B\cap D$,  and $B_3  = B\backslash (B_1\cup B_2)$. We proceed by case analysis.

Case 1: $|B_3| \ge |B|/3$. Then 
$$\|1_B\|\ \ge\ |B_3|\ \ge\ \frac{|B|}{3} \ \ge\ \frac{|A|}{3}\ \ge\ \frac{\ell}{3}\ \ge\ \frac{1}{3}\sum_{i=1}^{\ell}\frac{1}{\sqrt{i}}.$$

Case 2: $|B_2| \ge |B|/3$. Since $|B|\ge |A|$, we have $|B_2| \ge |A|/3\ge \ell/3$. Choose $F''\subset g(B_2)$ such that $|F''|= \lceil\ell/3\rceil\le \ell$. Since $A < B\cap \mathbf n$, we know that $F'' >  F$ and $F'' > F'$. Hence, $F''\in\mathcal{F}$. We have
$$\|1_B\|\ \ge\ \sum_{i=1}^{|F''|}\frac{1}{\sqrt{i}}\ =\ \sum_{i=1}^{\lceil\ell/3\rceil}\frac{1}{\sqrt{i}}.$$

Case 3: $|B_1| \ge |B|/3$. 
Write $B_1 = B_{1,1}\cup B_{1,2}$, where $B_{1,1} = B\cap\{i:i\in\mathbf m, i\notin\mathbf n\}$ and $B_{1,2} = B\cap \mathbf m\cap \mathbf n$. We have
$$|B_{1,2}|\ =\ |B_1| - |B_{1,1}|\ \ge\ |B_1| - M \ \ge\ \frac{|B|}{3} - M \ \ge\ \frac{|B|}{4} + \left(\frac{|B|}{12}-M\right)\ >\ \frac{|B|}{4}.$$

The same argument as in Case 2 (applied to $B_{1,2}$) gives that 
$$\|1_B\|\ \ge\ \sum_{i=1}^{\lceil \ell/4\rceil}\frac{1}{\sqrt{i}}.$$

In all cases, we obtain
\begin{align*}\|1_B\|&\ \ge\ \frac{1}{3}\sum_{i=1}^{\lceil \ell/4\rceil}\frac{1}{\sqrt{i}}\ \ge\ \frac{1}{12}\sum_{i=1}^{\ell}\frac{1}{\sqrt{i}}\ =\ \frac{1}{24}\left(2\sum_{i=1}^{\ell}\frac{1}{\sqrt{i}}\right)\\
&\ \ge\ \frac{1}{24}\left(\sum_{i\in F}w_{\pi(i)}|e^*_{m_i}(1_A)| + \sum_{i\in F'}w_{\sigma(i)}|e^*_{d_i}(1_A)|\right).\end{align*}
Letting  $\pi, \sigma, F, F'$ vary, we conclude that 
$$24\|1_B\|\ \ge\ \|1_A\|.$$

ii) not ($\mathbf m$, conservative): Let $N\in\mathbb{N}$. Choose 
$$A = \{d_1, \ldots, d_N\}\mbox{ and }B = \{m_{N^2+1}, \ldots, m_{N^2+N}\}.$$
Clearly, $(B, A) \in \mathbb{T}(\mathbf m)$.
Using the same argument as in the proof of item iii) in Theorem \ref{m3}, we have that 
$\|1_B\|/\|1_A\|\rightarrow\infty$ as $N\rightarrow\infty$.

iii) not ($\mathbf n$, democratic): Since $\{i: i\in\mathbf m, i\notin \mathbf n\}$ is finite, we know that $B = \{m_{N^2+1}, \ldots, m_{N^2+N}\}\subset \mathbf n$ for sufficiently large $N$. Therefore, item ii) implies that $\mathcal{B}$ is not ($\mathbf n$, democratic). 
\end{proof}

\begin{rek}\normalfont
The example in the proof of Theorem \ref{m4} is not $1$-($\mathbf n$, strong partially greedy) (see Remark \ref{r2}). However, there does exist a $1$-($\mathbf n$, strong partially greedy) basis that satisfies the conclusion of Theorem \ref{m4}. For an example of such a basis, see Case 2 in the proof of Theorem \ref{m8}. 
\end{rek}

The following theorem is immediate from Theorems \ref{m3} and \ref{m4}.
\begin{thm}\label{m5}
Let $\mathbf m, \mathbf n$ be two strictly increasing sequences such that the difference set $\Delta_{\mathbf m, \mathbf n}$ is infinite. There exists a $1$-unconditional basis $\mathcal{B}$ that is 
\begin{enumerate}
\item[i)] ($\mathbf n$, strong partially greedy).
\item[ii)] not ($\mathbf m$, strong partially greedy).
\item[iii)] not ($\mathbf n$, democratic). 
\end{enumerate}
\end{thm}

\subsection{When ($\mathbf n$, strong partially greedy) is the same as ($\mathbf m$, strong partially greedy)}

We consider the remaining case when the difference set $\Delta_{\mathbf m, \mathbf n}$ is finite. 

\begin{thm}\label{equiv}Let $\mathbf m, \mathbf n$ be two strictly increasing sequences such that the difference set $\Delta_{\mathbf m, \mathbf n}$ is finite. Then a basis $\mathcal{B}$ is ($\mathbf n$, strong partially greedy) if and only if it is ($\mathbf m$, strong partially greedy). 
\end{thm}

\begin{proof}
Assume that $\mathcal{B}$ is $\mathbf C_{\mathbf n, sp}$-($\mathbf n$, strong partially greedy). Let $D_1 = \{i: i\in\mathbf m, i\notin \mathbf n\}$ and $D_2 = \{i: i\in \mathbf n, i\notin\mathbf m\}$. Since $D_1$ and $D_2$ are finite, there exists $N$ such that $n_i, m_i > N$ implies that $m_i\in \mathbf n$, $n_i\in \mathbf m$. According to Theorem \ref{m1}, we need only to show that $\mathcal{B}$ is ($\mathbf m$, conservative). Pick finite sets $(A,B)\in \mathbb{T}(\mathbf m)$ with $2(N+1)\le |A|\le |B|$. Define 
\begin{align*}
    A_1 &\ =\ \{i\in A: i\le N\}\mbox{ and }A_2 \ =\ A\backslash A_1,\\
    B_1 &\ =\ \{i\in B: i\le N\}\mbox{ and }B_2 \ =\ B\backslash B_1.
\end{align*}
For $j\in B$, we have
$$\|1_{A_1}\|\ \le\ N\sup_n\|e_n\|\mbox{ and } \|1_B\|\sup_{n}\|e_n^*\|\ \ge\ |e_j^*(1_B)|\ =\ 1.$$
Hence, 
\begin{equation}\label{e10}\|1_{A_1}\|\ \le\ N\sup_{n}\|e_n\|\sup_{n}\|e_n^*\|\|1_B\|\ \le\ Nc_2^2\|1_B\|.\end{equation}
We claim that $A_2 < B_2\cap \mathbf n$. Suppose otherwise; that is, there exist $k\in B_2\cap \mathbf n$ and $j\in A_2$ such that $j \ge k$. Since $k\in B_2\cap \mathbf n$, $k\in B\cap \mathbf m$. Since $j\in A_2$, $j\in A$. However, the fact that $j \ge k$ contradicts our assumption that $A < B\cap \mathbf m$.

Case 1: $|A_2|\le |B_2|$. Since $\mathcal{B}$ is $\mathbf C_{\mathbf n, sp}$-($\mathbf n$, conservative) and $\mathbf C_{\mathbf n, sp}$-suppression quasi-greedy, we get 
$$\|1_{A_2}\|\ \le\ \mathbf C_{\mathbf n, sp}\|1_{B_2}\|\ \le\  \mathbf C^2_{\mathbf n, sp}\|1_{B}\|,$$
which, together with \eqref{e10}, gives
$$\|1_{A}\|\ \le\ \|1_{A_1}\| + \|1_{A_2}\|\ \le\ (Nc_2^2 + \mathbf C^2_{\mathbf n, sp})\|1_B\|.$$

Case 2: $|A_2| > |B_2|$. We have $|A_2|\le |A|$ and $|B_2|\ge |B| - N\ge |A|-N$. Since $|A|\ge 2(N+1)$, we get 
$$\frac{|A_2|}{2}+1\ \le \ \frac{|A|}{2}+1\ \le\ |A|-N\ \le\ |B_2|.$$
Therefore, we can partition $A_2$ into two disjoint sets $A_{2,1}$ and $A_{2,2}$ whose sizes $|A_{2,1}|, |A_{2,2}|\le |B_2|$.
Using the same argument as in Case 1, we obtain
$$\|1_{A_{2,i}}\|\ \le\ \mathbf C^2_{\mathbf n, sp}\|1_{B}\|, \mbox{ for }i = 1,2.$$
Hence,
$$\|1_A\|\ \le\ \|1_{A_1}\| + \|1_{A_2}\|\ \le\ (Nc_2^2 + 2\mathbf C^2_{\mathbf n, sp})\|1_B\|.$$

We have shown that for $(A, B)\in \mathbb{T}(\mathbf m)$ with $2(N+1)\le |A| \le |B|$, it holds that 
$$\|1_A\|\ \le\ (Nc_2^2 + 2\mathbf C^2_{\mathbf n, sp})\|1_B\|.$$
If $|A| < 2(N+1)$, then 
$$\|1_A\| < 2(N+1)\sup_{n}\|e_n\|\le 2(N+1)\sup_{n}\|e_n\|\sup_{n}\|e_n^*\|\|1_B\|\ \le\ 2(N+1)c_2^2\|1_B\|.$$
This completes our proof that $\mathcal{B}$ is ($\mathbf m$, conservative) and thus, ($\mathbf m$, strong partially greedy).
\end{proof}

\section{Lebesgue inequalities for the ($\mathbf n$, strong partially greedy) parameter $\widehat{\mathbf L}_m^{\mathbf n}$}\label{Lebesgue}

\subsection{($\mathbf n$, strong partially greedy)-type parameters}

For $m\in\mathbb{N}_0: = \mathbb{N}\cup \{0\}$, we capture the error term of the TGA by 
$$\gamma_m(x)\ : =\ \sup_{G_m(x)}\|x-G_m(x)\|.$$
Let $\widehat{\mathbf L}^{\mathbf n}_{m}$ denote the smallest constant such that 
$$\gamma_m(x) \ \le\ \widehat{\mathbf L}^{\mathbf n}_{m}\widehat{\sigma}^{\mathbf n}_m(x), \forall x\in X.$$
A greedy operator, $\mathcal{G}_m:X\rightarrow X$, maps each $x$ to a greedy sum of $x$ of order $m$ in $G(x,m)$; that is, $\mathcal{G}_m(x) = P_{A}(x)$ for some $A\in G(x, m)$.
Given two greedy operators $\mathcal G_m$ and $\mathcal G_n$, we write $\mathcal G_m < \mathcal G_n$ when $m < n$ and $\supp(\mathcal G_m(x))\subset \supp(\mathcal G_n(x))$ for all $x\in X$.
We list ($\mathbf n$, strong partially greedy)-type parameters that are of interests:
\begin{enumerate}
    \item Quasi-greedy constants:
    \begin{align*}g_m&\ :=\ \sup\{\|\mathcal{G}_k\|: k\le m\}\mbox{ and }g_m^c\ :=\ \sup\{\|I-\mathcal{G}_k\|: k\le m\},\\
    \tilde{g}_m&\ :=\ \sup\{\|\mathcal{G}_k - \mathcal{G}_j\|: 0\le k < j\le m, \mathcal{G}_k < \mathcal{G}_j\}.
    \end{align*}
    \item ($\mathbf n$, superconservative) constant:
        $$sc^{\mathbf n}_m\ :=\ \sup\left\{\frac{\|1_{\varepsilon A}\|}{\|1_{\delta B}\|}: (A,B)\in \mathbb{T}(\mathbf n), |A|\le |B| \le m, A\le n_m, \mbox{ signs }\varepsilon, \delta\right\}.$$
    \item ($\mathbf n$, PSLC) constant: 
    \begin{align*}
    &\omega^{\mathbf n}_m \ :=\ \\
    &\sup_{\substack{\varepsilon, \delta\\ \|x\|_\infty\le 1}}\left\{\frac{\|x+1_{\varepsilon A}\|}{\|x+ 1_{\delta B}\|}: |A|\le |B|\le m, A < (\supp(x)\sqcup B)\cap \mathbf n, A\le n_m, A\subset \mathbf n\right\}.\end{align*}
\end{enumerate}

\begin{lem}\cite[Lemmas 2.3 and 2.5]{BBG}\label{b2g}
If $x\in X$ and $\varepsilon = (\sgn(e_n^*(x)))$, then
\begin{enumerate}
    \item[i)] $\min_{n\in\Lambda} |e_n^*(x)|\|1_{\varepsilon \Lambda}\| \le \tilde{g}_m\|x\|, \forall \Lambda\in G(x, m)$.
    \item[ii)] $\|T_\alpha(x)\|\le g^c_{|\Lambda_\alpha|}\|x\|$,
\end{enumerate}
where $T_\alpha$ is the truncation operator at $\alpha$ and $\Lambda_\alpha = \{n: |e_n^*(x)|> \alpha\}$.
\end{lem}

\subsection{Bounds for $\widehat{\mathbf L}^{\mathbf n}_m$}

\begin{thm}[Upper bounds]\label{m6}
For $m\in\mathbb{N}$,
\begin{enumerate}
    \item[i)] $\widehat{\mathbf L}^{\mathbf n}_m \le 1+ 2\kappa m$, where $\kappa = \sup_{j,k}\|e_j\|\|e_k^*\|$.
    \item[ii)] $\widehat{\mathbf L}^{\mathbf n}_m \le g_m^c + \tilde{g}_m sc^{\mathbf n}_m$.
    \item[iii)] $\widehat{\mathbf L}^{\mathbf n}_m\le g^c_{m-1}\omega^{\mathbf n}_m$.
\end{enumerate}
\end{thm}

\begin{thm}[Lower bounds]\label{m7}
For $m\in\mathbb{N}$, $$\max\{g_m^c, \omega^{\mathbf n}_m\}\ \le\ \max_{1\le k\le m}\widehat{\mathbf L}^{\mathbf n}_k.$$
\end{thm}

\begin{cor}\label{c1}
It holds that $\widehat{\mathbf L}^{\mathbf n}_1 = \omega^{\mathbf n}_1$.
\end{cor}

\begin{proof}
By Theorem \ref{m6} item iii), $\widehat{\mathbf L}^{\mathbf n}_1 \le g^c_0\omega^{\mathbf n}_1 = \omega^{\mathbf n}_1$. By Theorem \ref{m7}, $\omega^{\mathbf n}_1\le \widehat{\mathbf L}^{\mathbf n}_1$. Hence, $\omega^{\mathbf n}_1 = \widehat{\mathbf L}^{\mathbf n}_1$.
\end{proof}

Before proving Theorem \ref{m6}, we need the following lemma. 
\begin{lem}\label{fl}
For $m\in\mathbb{N}$, let
 \begin{align*}
    &\widehat{\omega}^{\mathbf n}_m \ :=\ \\
    &\sup_{\substack{\varepsilon\\\|x\|_\infty\le 1\\A\subset\mathbf n, A\le n_m}}\left\{\frac{\|x\|}{\|x-P_A(x)+1_{\varepsilon B}\|}: |A|\le|B|\le m, A < (\supp(x-P_A(x))\sqcup B)\cap \mathbf n\right\}.\end{align*}
Then $\omega^{\mathbf n}_m = \widehat{\omega}^{\mathbf n}_m$.
\end{lem}

\begin{proof}
Let $x, A, B, \varepsilon$ be as in the definition of $\widehat{\omega}^{\mathbf n}_m$. We have
\begin{align*}
    \|x\|\ =\ \left\|x-P_A(x) + \sum_{n\in A}e_n^*(x)e_n\right\|&\ \le\ \sup_{\delta}\left\|x-P_A(x) + 1_{\delta A}\right\|\\
    &\ \le\ \omega^{\mathbf n}_m\left\|x-P_A(x) + 1_{\varepsilon B}\right\|.
\end{align*}
Hence, $\widehat{\omega}^{\mathbf n}_m\le \omega^{\mathbf n}_m$. 

Now, let $x, A, B, \varepsilon, \delta$ be as in the definition of $\omega^{\mathbf n}_m$. Let $y = x + 1_{\varepsilon A}$. We have
\begin{align*}
    \|x + 1_{\varepsilon A}\| \ =\ \|y\|\ \le\ \widehat{\omega}^{\mathbf n}_m\|y - P_A(y) + 1_{\delta B}\| \ =\ \widehat{\omega}^{\mathbf n}_m\|x + 1_{\delta B}\|.
\end{align*}
Hence, $\widehat{\omega}^{\mathbf n}_m\ge \omega^{\mathbf n}_m$, and we conclude that $\widehat{\omega}^{\mathbf n}_m= \omega^{\mathbf n}_m$.
\end{proof}

\begin{proof}[Proof of Theorem \ref{m6}]
i) For $m\ge 1$, let $\widetilde{\mathbf L}_m$ be the smallest constant such that 
$$\gamma_m(x)\ \le\ \widetilde{\mathbf L}_m\widetilde{\sigma}_m(x).$$
By \cite[Theorem 1.8]{BBG}, $\widetilde{\mathbf L}_m\le 1+ 2\kappa m$. Since $\widehat{\mathbf L}^{\mathbf n}_m\le \widetilde{\mathbf L}_m$, we obtain $\widehat{\mathbf L}^{\mathbf n}_m\le 1 + 2\kappa m$.

ii) Pick $x\in X$, $m\in\mathbb{N}$, and $A\in G(x, m)$. Choose $D = \{n_1, \ldots, n_k\}$ for some $k\le m$. Write 
\begin{equation}\label{e31}x-P_A(x)\ =\ P_{(A\cup D)^c}(x) + P_{D\backslash A}(x).\end{equation}
Since $A\backslash D$ is a greedy set of $x - P_{k}^{\mathbf n}(x)$ and 
\begin{align*}
    P_{(A\cup D)^c}(x)\ =\ (I-P_{A\backslash D})(x-P^{\mathbf n}_k(x)),
\end{align*}
we have \begin{equation}\label{e32}\|P_{(A\cup D)^c}(x)\|\ \le\ g_m^c\|x-P_k^\mathbf n(x)\|.\end{equation}
Let $\varepsilon = (\sgn(e_n^*(x)))$. Observe the following:
\begin{enumerate}
    \item[i)] $D\backslash A < (A\backslash D)\cap \mathbf n$;
    \item[ii)] $D\backslash A\subset\mathbf n$, $D\backslash A\le n_m$;
    \item[iii)] $|D\backslash A|\le |A\backslash D|\le m$;
\end{enumerate}
therefore, $\|1_{\delta (D\backslash A)}\|\le sc_m^{\mathbf n}\|1_{\varepsilon (A\backslash D)}\|$ for all signs $\delta$. We have
\begin{align}
    \|P_{D\backslash A}(x)\|&\ \le\ \max_{n\in D\backslash A}|e_n^*(x)|\sup_{\delta}\|1_{\delta (D\backslash A)}\|\mbox{ by convexity}\nonumber\\
    &\ \le\ sc_m^{\mathbf n}\min_{n\in A\backslash D}|e_n^*(x-P_k^{\mathbf n}(x))|\|1_{\varepsilon (A\backslash D)}\|\nonumber\\
    &\ \le\ \tilde{g}_{m}sc_m^{\mathbf n}\|x-P_k^{\mathbf n}(x)\|\mbox{ by Lemma \ref{b2g}}\label{e33}
\end{align}
By \eqref{e31}, \eqref{e32}, and \eqref{e33}, we obtain
$$\|x-P_A(x)\|\ \le\ (g_m^c+\tilde{g}_{m}sc_m^{\mathbf n})\|x-P^{\mathbf n}_k(x)\|,$$
and so, $\widehat{\mathbf L}^{\mathbf n}_m\le g_m^c+\tilde{g}_{m}sc_m^{\mathbf n}$.

iii) Let $x\in X$, $m\in\mathbb{N}$, and $A\in G(x, m)$. Fix $k\le m$. Set $E = \{1, \ldots, n_k\}\backslash A$, $F = A\backslash \{1, \ldots, n_k\}$, and $\alpha = \min_{n\in A}|e_n^*(x)|$. We check that $x-P_A(x), E, F$ satisfy the definition of $\widehat{\omega}^{\mathbf n}_m$:
\begin{enumerate}
    \item[a)] $|E|\le |F|\le m$, $E\subset\mathbf n$, and $E \le n_m$,
    \item[b)] $\|x-P_A(x)\|_\infty\le \alpha$,
    \item[c)] $E < (\supp(x-P_A(x)-P_E(x))\sqcup F)\cap \mathbf n$.
\end{enumerate}
Let $\varepsilon = (\sgn(e_n^*(x))$
By Lemma \ref{fl} and Lemma \ref{b2g}, we get
\begin{align*}
    \|x-P_A(x)\|&\ \le\ \omega^{\mathbf n}_m\|x-P_A(x) - P_E(x) + \alpha 1_{\varepsilon F}\|\\
                &\ =\ \omega^{\mathbf n}_m\|T_\alpha(x-P_{A\cup E}(x) + P_F(x))\|\\
                &\ \le \ \omega^{\mathbf n}_mg^c_{|\Lambda_\alpha|}\|x-P_k^{\mathbf n}(x)\|,
\end{align*}
where $\Lambda_\alpha = \{n: |e_n^*(x-P_k^{\mathbf n}(x))| > \alpha\}$. Since $|\Lambda_\alpha|\le m-1$, we obtain
$$\|x-P_A(x)\|\ \le\ \omega^{\mathbf n}_mg^c_{m-1}\|x-P_k^{\mathbf n}(x)\|.$$
This completes our proof that $\widehat{\mathbf L}^{\mathbf n}_m\le  \omega^{\mathbf n}_mg^c_{m-1}$.
\end{proof}

\begin{proof}[Proof of Theorem \ref{m7}]
First, we show that $\max_{1\le k\le m}\widehat{\mathbf L}^{\mathbf n}_k\ge g_m^c$.
Pick $x\in X$, $m\in\mathbb{N}$, and $A\in G(x, j)$ for some $j\le m$. We have
$$\|(I-P_A)x\|\ \le\ \widehat{\mathbf L}_j^{\mathbf n}\min_{k\le j}\|x-P^{\mathbf n}_{k}(x)\|\ \le\ \max_{1\le k\le m}\widehat{\mathbf L}_k^{\mathbf n}\|x\|.$$
Therefore, $g_m^c\le \max_{1\le k\le m}\widehat{\mathbf L}_k^{\mathbf n}$.

Next, we show that $\max_{1\le k\le m}\widehat{\mathbf L}^{\mathbf n}_k\ge \omega^{\mathbf n}_m$. Let $x, A, B, \varepsilon, \delta$ be chosen as in the definition of $\omega^{\mathbf n}_m$. Let $n_s = \max A$. Choose $D\subset\mathbf{n}$ such that $\max D < n_s$, $D\sqcup A$, and $s\le t := |D\cup B|\le m$. We show that $D$ actually exists, consider $D':= \{n_1, \ldots, n_s\}\backslash A$ and have
$$|D'\cup B| \ =\ |D'| + |B| \ =\ s + (|B| - |A|) \ \ge\ s.$$
If $|D'\cup B| \le m$, then choose $D = D'$. If $|D'\cup B| = m+p$ for some $p\ge 1$, then $|D'|\ge p$. Discard $p$ numbers from $D'$ to get $D$. Set 
$$y\ : =\ 1_{\varepsilon A} + 1_D + x+ 1_{\delta B}.$$
We have
$$\|x + 1_{\varepsilon A}\|\ =\ \|y - P_{D\cup B}(y)\|\ \le\ \widehat{\mathbf L}^{\mathbf n}_{|D\cup B|}\|y-P^{\mathbf n}_{s}(y)\|\ \le\ \max_{1\le k\le m}\widehat{\mathbf L}^{\mathbf n}_k\|x+1_{\delta B}\|.$$
We have shown that $\omega_m^{\mathbf n}\le \max_{1\le k\le m}\widehat{\mathbf L}^{\mathbf n}_k$.
\end{proof}

\subsection{Examples of optimality}

Our first example is a slight modification of \cite[Example 3.4]{BBL} (the summing basis), which shows the optimality of Theorem \ref{m6} item i) and Theorem \ref{m7}.

\begin{exa}\normalfont
Let $X$ be the completion of $c_{00}$ under the following norm: for $x = (x_1, x_2, \ldots)$, 
$$\|x\|\: = \ \sup_{m\ge 1}\left|\sum_{i=1}^{m}x_{n_i}\right| + \max_{n\notin \mathbf n}|x_n|.$$
The canonical basis $\mathcal{B} = (e_n)_{n=1}^\infty$ is monotone and normalized. The following hold
\begin{enumerate}
    \item[i)] $\|e^*_{n}\| = 1$ for all $n\notin \mathbf n$; $\|e^*_{n_1}\| = 1$; $\|e^*_{n_s}\| = 2$ for all $s> 1$.
    \item[ii)] $\omega^{\mathbf n}_m = \widehat{\mathbf L}^{\mathbf n}_m = \max_{1\le k\le m} \widehat{\mathbf L}^{\mathbf n}_k = 1 + 4m$.
\end{enumerate}

\begin{proof}[Proof of i)]
Let $n\notin \mathbf n$. Fix $x = (x_1, x_2, \ldots)\in X$. We have $|e^*_n(x)| = |x_n|\le \|x\|$. Hence, $\|e^*_n\|\le 1$. On the other hand, $e^*_n(e_n) = 1$, so $\|e^*_n\|\ge 1$. We conclude that $\|e^*_n\| = 1$. 

Assume $n = n_s\in \mathbf n$ and $s>1$.  We have 
$$|e^*_n(x)| \ = \ |x_n|\ \le\ \left|\sum_{i=1}^{s-1}x_{n_i}\right| + \left|\sum_{i=1}^sx_{n_i}\right|\ \le\ 2\|x\|;$$
hence, $\|e_n^*\|\le 2$. For a lower bound, we consider $y = -e_{n_{s-1}} + 2e_{n_s}$. Then $\|y\| = 1$, while 
$|e_n^*(y)|\ =\ 2$. Therefore, $\|e_n^*\|\ge 2$; we conclude that $\|e_n^*\| = 2$.

Finally, it is easy to check that $\|e^*_{n_1}\| = 1$.
\end{proof}

\begin{proof}[Proof of ii)]
By Theorem \ref{m6} item i), we have 
$\widehat{\mathbf L}^{\mathbf n}_k\le 1 + 4k$, for $\kappa = 2$. Hence, $\max_{1\le k\le m}\widehat{\mathbf L}^{\mathbf n}_k\le 1+4m$. We show that $\omega^{\mathbf n}_m\ge 1+4m$. Let 
\begin{align*}
    A &\ =\ \{n_1, n_2, n_3, \ldots, n_m\},\\
    B &\ =\ \{n_{m+2}, n_{m+5}, n_{m+8},\ldots, n_{4m-1}\},\\
    C &\ =\ \{n_{m+1}, n_{m+4}, n_{m+7}, \ldots, n_{4m-2}\}\cup \{n_{m+3}, n_{m+6}, \ldots, n_{4m}\}\cup\{n_{4m+1}\}.
\end{align*}
Let $x = \frac{1}{2}1_C$. Observe that 
$|A| = |B| = m$, $A< (B\sqcup \supp(x))\cap \mathbf n$, $A\subset\mathbf n$, $A\le n_m$.
Clearly, $\|x+1_A\| = 2m+\frac{1}{2}$, while $\|x-1_B\| = \frac{1}{2}$. Hence, 
$$\omega_m^\mathbf n\ \ge\ \frac{\|x+1_A\|}{\|x-1_B\|}\ =\ 4m + 1,$$
as desired. Therefore, 
$$4m+1\ \le\ \omega^{\mathbf n}_m\ \le\ \max_{1\le k\le m}\widehat{\mathbf L}^{\mathbf n}_k\ \le\ 1+4m.$$
Using $\widehat{\mathbf L}^{\mathbf n}_k\le 1 + 4k$ for all $k\in\mathbb{N}$, we obtain item ii). 
\end{proof}
\end{exa}

Next, we modify \cite[Example 3.5]{BBL} to show the optimality of Theorem \ref{m6} items ii) and iii). 

\begin{exa}\normalfont
Let $X$ be the completion of $c_{00}$ with respect to the following norm: for $x = (x_1, x_2, \ldots)$, 
$$\|x\| \ =\ \max\left\{\sum_{i=1}^{\infty}|x_{n_{2i-1}}|, \max_{n\notin\mathbf n}|x_n|, \max_{i}|x_{n_{2i}}|\right\}.$$
Let $\mathcal{B} = (e_n)_n$ be the canonical basis of $X$. Then $\mathcal{B}$ is $1$-unconditional and normalized. The following hold
\begin{enumerate}
    \item[i)] $g_m = \tilde{g}_m = g_m^c = 1$.
    \item[ii)] $sc^{\mathbf n}_{2m} = sc^{\mathbf n}_{2m-1} = m$.
    \item[iii)] $\omega^{\mathbf n}_{2m} = \omega^{\mathbf n}_{2m-1} = \widehat{\mathbf L}^{\mathbf n}_{2m-1} = \widehat{\mathbf L}^{\mathbf n}_{2m} = m+1$.
\end{enumerate}

\begin{proof}[Proof of i)]
Item i) follows from the $1$-unconditionality of $\mathcal{B}$.
\end{proof}

\begin{proof}[Proof of ii)] Let $m\in\mathbb{N}$.
Taking any set $A\subset\mathbf n$ with $A\le n_{2m}$ and sign $\varepsilon$, we have $\|1_{\varepsilon A}\|\le m$. Also, for any nonempty $B\subset\mathbb{N}$ and sign $\delta$, we have $\|1_{\delta B}\|\ge 1$. Hence, $sc^{\mathbf n}_{2m}\le m$. Since $(sc^{\mathbf n}_m)_m$ is increasing, it suffices to show that $sc^{\mathbf n}_{2m-1}\ge m$. Let
$$A_m = \{n_1, n_3,\ldots, n_{2m-1}\}\mbox{ and }B_m = \{n_{2m+2}, n_{2m+4}, \ldots, n_{4m}\}.$$
Then $|A_m| = |B_m| = m$, $A_m\subset\mathbf n$, $A_m \le n_{2m-1}$, and $A_m < B_m\cap \mathbf n$. We have
$$\|1_{A_m}\| \ = \ m, \mbox{ while }\|1_{B_m}\| \ = \ 1.$$
Hence, $sc_{2m-1}^{\mathbf n}\ge m$. We conclude that $sc_{2m-1}^{\mathbf n} = sc_{2m}^{\mathbf n} = m$.
\end{proof}

\begin{proof}[Proof of iii)]
By items i), ii), and Theorem \ref{m6} item ii), we obtain $\widehat{\mathbf L}_{2m-1}^{\mathbf n}\le m+1$ and $\widehat{\mathbf L}_{2m}^{\mathbf n}\le m+1$. Observe that 
$$\|1_{A_m}+e_{n_{2m+1}}\|\ =\ m+1, \mbox{ while }\|1_{B_m} + e_{n_{2m+1}}\|\ =\ 1.$$
Hence, $\omega_{2m-1}^{\mathbf n}\ge m+1$. According to Theorem \ref{m7}, we get
\begin{equation}\label{e40}m+1\ \le\ \omega^{\mathbf n}_{2m-1}\ \le\ \omega^{\mathbf n}_{2m}\ \le\ \max_{1\le k\le 2m}\widehat{\mathbf L}_k^{\mathbf n}\ \le\ m+1.\end{equation}
Let $\eta > 0$ and set 
$$x\ :=\ 1_{A_m} + e_{n_{2m+1}} + (1+\eta)1_{B_{2m}}.$$
Then
$$\|x-G_{2m-1}(x)\|\ \ge\ \|x-G_{2m}(x)\|\ =\ \|1_{A_m} + e_{n_{2m+1}}\|\ =\ m+1.$$
On the other hand,
$$\|x - P^{\mathbf n}_{2m}(x)\|\ \le\ \|x-P^{\mathbf n}_{2m-1}(x)\|\ =\ \|e_{n_{2m+1}} + (1+\eta)1_{B_{2m}}\|\ =\ 1+\eta.$$
Letting $\eta\rightarrow 0$, we get $\widehat{\mathbf L}^{\mathbf n}_{2m-1}\ge m+1$ and $\widehat{\mathbf L}^{\mathbf n}_{2m}\ge m+1$, which, combined with \eqref{e40}, gives iii). 
\end{proof}
\end{exa}

\section{Characterizations of $1$-($\mathbf n$, strong partially greedy) bases}\label{1PSLC}
\subsection{Characterizations of $1$-($\mathbf n$, strong partially greedy) bases}
Researchers have studied different greedy-type bases with constant $1$. Albiac and Wojtaszczyk \cite{AW} characterized $1$-greedy bases using $1$-suppression unconditionality and the so-called Property (A). Later, Albiac and Anserona \cite{AA0} showed that a basis is $1$-quasi-greedy if and only if it is $1$-suppression unconditional. They also showed that a basis is $1$-almost greedy if and only if it has Property (A) \cite{AA}. It is an open problem whether there exists a conditional basis which has Property (A). Recent advance in this direction gives an example of a basis that has Property (A) but is not $1$-suppression unconditional \cite{AABCO}.

Recently, Berasategui et al.\ characterized $1$-strong partially greedy bases as being $1$-partially symmetric for largest coefficients \cite{BBL}. We have the corresponding characterization of $1$-($\mathbf n$, strong partially greedy) bases.

\begin{thm}\label{m9}
A basis $\mathcal{B}$ is $1$-($\mathbf n$, strong partially greedy) if and only if it is $1$-($\mathbf n$, PSLC).
\end{thm}
\begin{proof}If $\mathcal{B}$ is $1$-($\mathbf n$, strong partially greedy), then Proposition \ref{p1} states that it is $1$-($\mathbf n$, PSLC). Assume that $\mathcal{B}$ is $1$-($\mathbf n$, PSLC), we know that $\omega^\mathbf n_m = 1$ for all $m\in\mathbb{N}$. Let $j\in\mathbb{N}_0$. By Theorem \ref{m6} item iii) and Theorem  \ref{m7}, we get
$$g_{j}^c\ \ge\ \max_{1\le k\le j+1}\widehat{\mathbf L}^{\mathbf n}_k\ \ge\ g_{j+1}^c;$$
hence, $g_j^c\ge g_{j+1}^c$ for all $j \in\mathbb{N}_0$. However, by definition, $g_{j}^c\le g_{j+1}^c$. Therefore, $g_j^c = g_{j+1}^c$ for all $j\in \mathbb{N}_0$. Since $g_0^c = 1$, we obtain that $g_j^c = 1$ for all $j\in\mathbb{N}_0$. This gives
$\widehat{\mathbf L}^{\mathbf n}_j = 1$ for all $j\in\mathbb{N}_0$ and so, $\mathcal{B}$ is $1$-($\mathbf n$, strong partially greedy).
\end{proof}

\begin{prop}\label{p10}
Let $\mathcal{B}$ be a basis of a Banach space $X$. The following are equivalent:
\begin{enumerate}
    \item [i)] $\mathcal{B}$ is $1$-($\mathbf n$, PSLC).
    \item [ii)] $\mathcal{B}$ satisfies the two following conditions simultaneously:
    \begin{enumerate}
        \item[a)] For $x\in X$ with $\|x\|_\infty\le 1$ and for all $k\notin \supp(x)$, 
        \begin{equation}\label{e42}\|x\|\ \le\ \|x + e_k\|.\end{equation}
        \item[b)] For $x\in X$ with $\|x\|_\infty \le 1$, for $s, t\in\mathbb{F}$ with $|s| = |t| = 1$, and for $j\in\mathbf n$, $j < (\{k\}\sqcup \supp(x))\cap \mathbf n$, 
        \begin{equation}\label{e43}\|x+se_j\|\ \le\ \|x+te_k\|.\end{equation}
    \end{enumerate}
\end{enumerate}
\end{prop}

\begin{proof}
Clearly, if $\mathcal{B}$ is $1$-($\mathbf n$, PSLC), then $\mathcal{B}$ satisfies \eqref{e42} and \eqref{e43}. Suppose that $\mathcal{B}$ satisfies both \eqref{e42} and \eqref{e43}. Choose $x, A, B, \varepsilon, \delta$ as in the definition of ($\mathbf n$, PSLC). We show that 
\begin{equation}\label{e44}\|x+1_{\varepsilon A}\|\ \le\ \|x + 1_{\delta B}\|\end{equation}
inductively on $|B|$. Base case: $|B| = 1$. If $A = \emptyset$, then \eqref{e42} implies \eqref{e44}. If $|A| = 1$, then \eqref{e43} implies \eqref{e44}. Inductive hypothesis (I.H.): assume that for some $k\in\mathbb{N}$, \eqref{e44} holds for $|B|\le k$. We show that \eqref{e44} holds for $|B| = k+1$. If $A = \emptyset$, then we use \eqref{e42} inductively to obtain \eqref{e44}. Assume that $|A| \ge 1$. Let $p =  \max A$, $q\in B$, $A' = A\backslash \{p\}$, and $B' = B\backslash \{q\}$. We have
\begin{align*}\|x+1_{\varepsilon A}\|\ =\ \|(x + \varepsilon_p e_p) + 1_{\varepsilon A'}\|&\ \le\ \|(x+\varepsilon_p e_p) + 1_{\delta B'}\|\mbox{ by I.H.}\\
&\ \le\ \|(x+\delta_k e_k) + 1_{\delta B'}\|\mbox{ by \eqref{e43}}\\
&\ =\ \|x+1_{\delta B}\|.\end{align*}
This shows that $\mathcal{B}$ is $1$-($\mathbf n$, PSLC). 
\end{proof}

\subsection{An 1-($\mathbf n$, PSLC) basis that is not ($\mathbf m$, conservative)}

Berasategui et al. gave \cite[Example 4.3]{BBL} of a basis that is $1$-PSLC but is not democratic, thus answering negatively the question of whether a $1$-strong partially greedy basis is necessarily almost greedy. We now offer a stronger example than \cite[Example 4.3]{BBL}.

\begin{thm}\label{m8}
Let $\mathbf m$, $\mathbf n$ be sequences such that the difference set $\Delta_{\mathbf m, \mathbf n}$ is infinite. There exists a $1$-unconditional basis $\mathcal{B}$ that is $1$-($\mathbf n$, PSLC) but is not ($\mathbf m$, conservative) (and thus, not democratic).
\end{thm}

\begin{proof}
Case 1: $\{i: i\in \mathbf m, i\notin \mathbf n\}$ is infinite. We reuse the example in Theorem \ref{m3}: let $\mathcal{F}:= \{F\subset\mathbb{N}: \sqrt{\min F}\ge |F|\}$ and $w_n = \frac{1}{\sqrt{n}}$ for $n\ge 1$. Write $\mathbf n = n_1, n_2, \ldots$. Define $f: \mathbf n\rightarrow\mathbb{N}$ by $f(n_k) = k$. Let $X$ be the completion of $c_{00}$ under the following norm: for $x = (x_1, x_2, \ldots)$,
$$\|x\|\ =\ \sup_{\substack{F\in\mathcal{F}\\ \pi}}\sum_{i\in F}w_{\pi(i)}|x_{n_i}| + \sum_{n\notin \mathbf n}|x_n|,$$
where $\pi$ is a bijection on $\mathbb{N}$.
Let $\mathcal{B}$ be the canonical basis, which is $1$-unconditional and thus, satisfies \eqref{e42}. Let us check that $\mathcal{B}$ satisfies \eqref{e43}. Choose $x, j, k, s, t$ as in \eqref{e43}. Let $F\in\mathcal{F}$ and $\pi$ be a bijection. Consider 
\begin{align*}h_{F, \pi}(x+se_j)&\ :=\ \sum_{i\in F}w_{\pi(i)}|e^*_{n_i}(x+se_j)| + \sum_{n\notin \mathbf n}|e_n^*(x+se_j)|\\
&\ =\ \sum_{i\in F}w_{\pi(i)}|e^*_{n_i}(x+se_j)| + \sum_{n\notin \mathbf n}|e_n^*(x)|.
\end{align*}
\begin{enumerate}
    \item[i)] Case 1.1: $f(j)\notin F$. Then 
    $$h_{F, \pi}(x+se_j) \ \le\ \|x\|\ \le\ \|x+te_k\|.$$
    \item[ii)] Case 1.2: $f(j)\in F$ and $k\in\mathbf n$. Without loss of generality, assume that $f(k)\notin F$ because $k\notin \supp(x+se_j)$. Define $F' = (F\backslash f(j))\cup f(k)$. Since $j < k$ and $|F| = |F'|$, we know that $F'\in \mathcal{F}$. Define a bijection $\pi'$ on $\mathbb{N}$ such that
    $\pi'(n)\ = \begin{cases}
    \pi(n) &\mbox{ if } n\neq f(j), f(k)\\
    \pi(f(j)) &\mbox{ if } n = f(k)\\
    \pi(f(k)) &\mbox{ if } n = f(j)
    \end{cases}$.
    We have
    \begin{align*}
        \|x+te_k\|&\ \ge\ h_{F', \pi'}(x+te_k)\\
        &\ =\ \sum_{i\in F'}w_{\pi'(i)}|e^*_{n_i}(x+te_k)| + \sum_{n\notin \mathbf n}|e_n^*(x+te_k)|\\
        &\ =\ \sum_{i\in F\backslash f(j)}w_{\pi'(i)}|e^*_{n_i}(x)| + w_{\pi'(f(k))} + \sum_{n\notin \mathbf n}|e_n^*(x)|\\
        &\ =\  \sum_{i\in F\backslash f(j)}w_{\pi(i)}|e^*_{n_i}(x)| + w_{\pi(f(j))} + \sum_{n\notin \mathbf n}|e_n^*(x)|\\
        &\ =\ h_{F, \pi}(x+se_j).
    \end{align*}
    \item[iii)] Case 1.3: $f(j)\in F$ and $k\notin\mathbf n$. By definition of $\|\cdot\|$,
    $$\|x+te_k\|\ =\ \|x\| + 1\ =\ \|x\| + \|se_j\|\ \ge\ \|x+se_j\|.$$
\end{enumerate}
    Since $F$ and $\pi$ are arbitrary, we obtain from the three cases that $\|x+se_j\|\ \le\ \|x+te_k\|$. By Proposition \ref{p10}, $\mathcal{B}$ is $1$-($\mathbf n$, PSLC). The proof of Theorem \ref{m3} shows that $\mathcal{B}$ is not ($\mathbf m$, conservative).

Case 2: $\{i: i\in \mathbf n, i\notin \mathbf m\}$ is infinite. By Case 1, we can assume that $H:= \{i: i\in \mathbf m, i\notin \mathbf n\}$ is finite. Let $K = \{i: i\in\mathbf n, i\notin \mathbf m\} = k_1 < k_2 < \cdots$.
Define a function $\phi: \mathbf n\rightarrow \mathbb{N}$ as 
$$\phi(n) \ =\ \begin{cases}1&\mbox{ if } n \le k_1,\\ j &\mbox{ if } k_{j}+1\le  n\le k_{j+1} \mbox{ for some }j\ge 1.\end{cases}$$
Let $X$ be the completion of $c_{00}$ under the following norm: for $x = (x_1, x_2, \ldots)$,
$$\|x\|\ :=\ \sup_{F}\left(\sum_{n\notin \mathbf n}|x_n| + \sum_{m\in F}|x_m|\right),$$
where $F\subset \mathbf n$ and $\sqrt{\phi(\min F)}\ge |F|$. Let $\mathcal{B}$ be the canonical basis, which is $1$-unconditional and normalized. Clearly, $\mathcal{B}$ satisfies \eqref{e42}. We now verify that $\mathcal{B}$ satisfies \eqref{e43}. Choose $x, k, j, s, t$ as in \eqref{e43}. Let $F\subset \mathbf n$ and $\sqrt{\phi(\min F)}\ge |F|$. Consider 
\begin{align*}h_{F}(x + se_j) &\ := \ \sum_{n\notin \mathbf n}|e_n^*(x+se_j)| + \sum_{m\in F}|e_m^*(x+se_j)|\\
&\ =\ \sum_{n\notin \mathbf n}|e_n^*(x)| + \sum_{m\in F}|e_m^*(x+se_j)|.
\end{align*}
We shall prove that $h_F(x+se_j)\le \|x+te_k\|$.
\begin{enumerate}
    \item [i)] Case 2.1: $j\notin F$. Then $se_j$ does not contribute to $h_F(x+se_j)$. Hence, 
    $$h_F(x+se_j) \ =\ h_F(x) \ \le\ \|x\|\ \le\ \|x+te_k\|.$$
    \item [ii)] Case 2.2: $j\in F$ and $k\notin \mathbf n$. We have
    \begin{align*}
    \|x+te_k\|\ \ge\ h_F(x+te_k) &\ =\ \sum_{n\notin \mathbf n}|e_n^*(x+te_k)| + \sum_{m\in F}|e_m^*(x+te_k)|\\
    &\ =\ \sum_{n\notin \mathbf n}|e_n^*(x)| + 1 + \sum_{m\in F}|e_m^*(x)|\\
    &\ =\  \sum_{n\notin \mathbf n}|e_n^*(x)| + \sum_{m\in F}|e_m^*(x+se_j)|\\
    &\ =\ h_F(x+se_j).
    \end{align*}
    \item [iii)] Case 2.3: $j\in F$ and $k\in \mathbf n$. Then $j < k$ because $j < (\supp(x)\sqcup \{k\})\cap \mathbf n$. Without loss of generality, assume that $k\notin F$, since $k\notin \supp(x+se_j)$.
    Form $F' = (F\backslash \{j\})\cup \{k\}$. Then $|F'| = |F|$, while $\phi(\min F') \ge \phi(\min F)$. It follows that $F'\in\mathcal{F}$. It is easy to check that $h_F(x+se_j) = h_{F'}(x+te_k)\le \|x+te_k\|$.
\end{enumerate}
In all cases, we obtain $h_{F}(x+se_j)\le \|x+te_k\|$. Since $F$ is arbitrary, $\|x+se_j\|\le \|x+te_k\|$ and so, $\mathcal{B}$ is $1$-($\mathbf n$, PSLC).

We show that $\mathcal{B}$ is not ($\mathbf m$, conservative). Let $M\in\mathbb{N}$,  $B = \{k_1, k_2, \ldots, k_M\}$, 
$A\subset \mathbf m\cap \mathbf n$, $|A| = M$, and $\min A$ is sufficiently large such that $\sqrt{\phi(\min A)} \ge M$.
Observe that $(A, B)\in \mathbb T(\mathbf m)$. While $\|1_A\| = M$, the same argument as in Theorem \ref{m3} item iii) gives $\|1_B\|\lesssim \sqrt{M}$. Therefore, $\|1_A\|/\|1_B\|\rightarrow\infty$ as $M\rightarrow\infty$, which shows that $\mathcal{B}$ is not ($\mathbf m$, conservative). 
\end{proof}

\begin{rek}\label{r2}\normalfont
The example in the proof of Theorem \ref{m4} is not $1$-($\mathbf n$, PSLC). Indeed, since $|\{i: i\in \mathbf m, i\notin \mathbf n\}|<\infty$, pick $i_0\in \mathbf m\cap \mathbf n$. Choose $n_1, n_2\in\{i: i\in \mathbf n, i\notin \mathbf m\}$, which is infinite, such that $i_0 < n_1 < n_2$. We obtain $$\|e_{n_1}+e_{i_0}\| \ =\ 2,\mbox{ while }\|e_{n_1} + e_{n_2}\|\ =\ 1 + \frac{1}{\sqrt{2}},$$
which violates \eqref{e43}. 
\end{rek}

\section{Consecutive projections onto a sequence $\mathbf n$}\label{consecutive}
\subsection{Characterization of ($\mathbf n$, almost greedy) bases}
We define ($\mathbf n$, almost greedy) bases that shall be shown to lie strictly between the realm of almost greedy bases and the realm of ($\mathbf n$, strong partially greedy) bases. 
\begin{defi}\normalfont
A basis $\mathcal{B}$ is said to be ($\mathbf n$, almost greedy) if there exists a constant $\mathbf C\ge 1$ such that
$$\|x-G_m(x)\|\ \le\ \mathbf C\widetilde{\sigma}^{\mathbf n}_m(x), \forall x\in X, \forall m\in\mathbb{N}, \forall G_m(x),$$
where 
$$\widetilde{\sigma}^{\mathbf n}_m(x)\ =\ \inf\left\{\|x-P_A(x)\|: A\subset\mathbf n, |A| = m\right\}.$$
The least such $\mathbf C$ is denoted by $\mathbf C_{\mathbf n, a}$.
\end{defi}

\begin{prop}\label{pp1}
Let $\mathcal{B}$ be a $\mathbf C_{\mathbf n, a}$-($\mathbf n$, almost greedy) basis. Then
\begin{equation}\label{ee8}\|x-G_m(x)\|\ \le\ \mathbf C_{\mathbf n, a}\min_{0\le k\le m}\widetilde{\sigma}^{\mathbf n}_k(x), \forall x\in X, \forall m\in\mathbb{N}, \forall G_m(x).\end{equation}
\end{prop}

\begin{proof}
Fix $x\in X$, $m\in\mathbb{N}$, $G_m(x)$, and $0\le k\le m$. Let $A\subset\mathbf n$ with $|A| = k$. We shall show that $\|x-G_m(x)\|\le \mathbf C_{\mathbf n, a}\|x-P_A(x)\|$. Choose a set $B_N\subset\mathbf n$ such that $B_N > \max A + N$, and $|B_N| = m-k$. Since $\mathcal{B}$ is $\mathbf C_{\mathbf n, a}$-($\mathbf n$, almost greedy), we have
$$\|x-G_m(x)\|\ \le\ \mathbf C_{\mathbf n, a}\|x-P_A(x)-P_{B_N}(x)\|\ \le\ \mathbf C_{\mathbf n, a}\|x-P_A(x)\| + \mathbf C_{\mathbf n, a}\|P_{B_N}(x)\|.$$
Let $\alpha_N = \max_{n > \max A + N}|e_n^*(x)|$. We get
$$\|P_{B_N}(x)\|\ \le\ m\alpha_N\sup_{n}\|e_n\|\ \le\ m\alpha_N c_2\rightarrow 0\mbox{ as }N\rightarrow\infty.$$
Therefore, $\|x-G_m(x)\|\le \mathbf C_{\mathbf n, a}\|x-P_A(x)\|$, as desired. 
\end{proof}

\begin{thm}\label{mm2}
A basis $\mathcal{B}$ is ($\mathbf n$, almost greedy) if and only if $\mathcal{B}$ is quasi-greedy and ($\mathbf n$, democratic).
\end{thm}

\begin{proof}
Assume that $\mathcal{B}$ is $\mathbf C_{\mathbf n, a}$-($\mathbf n$, almost greedy). Substituting $k=0$ into \eqref{ee8}, we see that $\mathcal{B}$ is $\mathbf C_{\mathbf n, a}$-suppression quasi-greedy. Now take $(A, B)\in \mathbb{S}(\mathbf n)$. Let $x = 1_{A\cup B}$.
We have 
$$\|1_{A}\|\ \le\ \|x-P_{B\backslash A}(x)\|\ \le\ \mathbf C_{\mathbf n, a}\widetilde{\sigma}^{\mathbf n}_{|A\backslash B|}(x)\ \le\ \mathbf C_{\mathbf n, a}\|x-P_{A\backslash B}(x)\|\ =\ \mathbf C_{\mathbf n, a}\|1_B\|.$$
Therefore, $\mathcal{B}$ is $\mathbf C_{\mathbf n, a}$-($\mathbf n$, democratic). 

Next, assume that $\mathcal{B}$ is $\mathbf C_q$-quasi-greedy and $\Delta_{\mathbf n, d}$-($\mathbf n$, democratic). By Remark \ref{r1}, $\mathcal{B}$ is $\Delta_{\mathbf n, sd}$-($\mathbf n$, superdemocratic). Pick $x\in X$, $m\in \mathbb{N}$, $A\in G(x, m)$, and $B\subset\mathbf n$ with $|B| = m$. Write 
$$x - P_A(x)\ =\ P_{B^c}(x) + P_{B\backslash A}(x) - P_{A\backslash B}(x).$$
Let $\varepsilon = (\sgn(e_n^*(x)))_n$.
We have $\|P_{A\backslash B}(x)\|\le \mathbf C_q\|x-P_B(x)\|$ and 
\begin{align*}
\|P_{B\backslash A}(x)\|&\ \le\ \max_{n\in B\backslash A}|e_n^*(x)|\sup_{\delta}\|1_{\delta (B\backslash A)}\|\\
&\ \le\ \min_{n\in A\backslash B}|e_n^*(x)|\Delta_{\mathbf n, sd}\|1_{\varepsilon (A\backslash B)}\|\\
&\ \le\ 2\mathbf C_q\Delta_{\mathbf n, sd}\|x-P_B(x)\|\mbox{ by Lemma \ref{b2g}}.
\end{align*}
We conclude that 
$$\|x-P_A(x)\|\ \le\ (1 + (1+ 2\Delta_{\mathbf n, sd})\mathbf C_q)\|x-P_B(x)\|.$$
This completes our proof. 
\end{proof}

\subsection{The realm of ($\mathbf n$, almost greedy) bases}
From definitions, we know that for any sequence $\mathbf n$, almost greedy $\Longrightarrow$ ($\mathbf n$, almost greedy) $\Longrightarrow$ ($\mathbf n$, strong partially greedy). We now show that neither of the reverse implications hold. In fact, we show more general results. 

\begin{prop}\label{pe1}
\begin{enumerate}
\item[i)] If $\{i: i\in \mathbf m, i\notin \mathbf n\}$ is infinite, then there exists an ($\mathbf n$, almost greedy) basis that is not ($\mathbf m$, almost greedy). 
\item[ii)] If $\{i: i\in\mathbf m, i\notin \mathbf n\}$ is finite, then an ($\mathbf n$, almost greedy) basis is necessarily ($\mathbf m$, almost greedy). 
\item[iii)] For every sequence $\mathbf n$, there exists a $1$-($\mathbf n$, strong partially greedy) basis that is not ($\mathbf n$, almost greedy).
\end{enumerate}
\end{prop}

\begin{proof}
i) The set $\{i: i\in \mathbf m, i\notin \mathbf n\}$ is infinite. Let $w_n = \frac{1}{\sqrt{n}}$ for $n\ge 1$ and $X$ be the completion of $c_{00}$ under the following norm:
$$\|(x_i)_i\| \ =\ \max\left\{\|(x_i)_i\|_\infty, \sup_{\pi}\sum_{i\in \mathbf n}w_{\pi(i)}|x_i|, \sum_{i\notin \mathbf n}|x_i|\right\},$$
where $\pi: \mathbf n\rightarrow\mathbb{N}$ is a bijection. Let $\mathcal{B}$ be the canonical basis. Since $\mathcal{B}$ is $1$-unconditional, by Theorem \ref{mm2}, we need only to show that $\mathcal{B}$ is ($\mathbf n$, democratic) but is not ($\mathbf m$, democratic).

Pick $N\in\mathbb{N}$, $A_N \subset \{i: i\in \mathbf m, i\notin \mathbf n\}$ with $|A_N| = N$, and $B_N \subset \mathbf n$ with $|B_N| = N$. It is easy to check that $\|1_{A_N}\| = N$ and $\|1_{B_N}\|\le 2\sqrt{N}$. Since $\|1_{A_N}\|/\|1_{B_N}\|\rightarrow\infty$ as $N\rightarrow\infty$, $\mathcal{B}$ is not ($\mathbf m$, democratic). 

Pick $(A, B)\in \mathbb{S}(\mathbf n)$. Assume that $|B|\ge 2$, since when $|B| \le 1$, we always get $\|1_A\|\le \|1_B\|$. 
Let us estimate $\|1_B\|$. Write $B_1 = B\cap \mathbf n$ and $B_2 = B\backslash B_1$.  If $|B_1|\ge |B|/2$, then 
$$\|1_B\|\ \ge\ \sum_{n=1}^{\lfloor |B|/2\rfloor}\frac{1}{\sqrt{n}}\ \ge\ \frac{1}{2}\sum_{n=1}^{|B|}\frac{1}{\sqrt{n}}\ \ge\ \frac{1}{2}\sum_{n=1}^{|A|}\frac{1}{\sqrt{n}}\ =\ \frac{1}{2}\|1_A\|.$$
If $|B_2|\ge |B|/2$, then 
$$\|1_B\|\ \ge\ \frac{|B|}{2}\ \ge\ \frac{1}{2}|A|\ \ge\ \frac{1}{2}\sum_{n=1}^{|A|}\frac{1}{\sqrt{n}}\ =\ \frac{1}{2}\|1_A\|.$$
Therefore, $\mathcal{B}$ is ($\mathbf n$, democratic).

ii) By Theorem \ref{mm2}, it suffices to show that a $\Delta_{\mathbf n, d}$-($\mathbf n$, democratic) basis is ($\mathbf m$, democratic).  Since $\{i: i\in\mathbf m, i\notin \mathbf n\}$ is finite, there exists $N\in\mathbb{N}$ such that if $\mathbf m\cap [N,\infty)\subset \mathbf n$. Let $(A, B)\in \mathbb{S}(\mathbf m)$. 
Write $A = A_1\cup A_2$, where $A_1 = A\cap [1, N-1]$ and $A_2 = A\cap [N,\infty)\subset \mathbf n$. We have
\begin{align*}
    \|1_{A_1}\|&\ \le\ (N-1)\sup_{n}\|e_n\|\ \le\ (N-1)c_2\sup_{n}\|e_n^*\|\|1_B\|\ \le\ (N-1)c_2^2\|1_B\|,\\
    \|1_{A_2}\|&\ \le\ \Delta_{\mathbf n, d}\|1_B\|.
\end{align*}
Hence, $$\|1_A\|\ \le\ \|1_{A_1}\| + \|1_{A_2}\|\ \le\ ((N-1)c_2^2 + \Delta_{\mathbf n, d})\|1_B\|.$$
This completes our proof. 

iii) Use the example in Case 2 in the proof of Theorem \ref{m8}.
\end{proof}

\subsection{($\mathbf n$, consecutive almost greedy) bases and characterizations}
We now strengthen the notion of ($\mathbf n$, strong partially greedy) bases and show that the new notion is equivalent to the ($\mathbf n$, almost greedy) property. For $m\in\mathbb{N}_0$, let
\begin{align*}
\mathcal{I}^{\mathbf n, (m)}&\ =\ \{A\subset \mathbf n: A = \{n_{k+1}, n_{k+2}, \ldots, n_{k+m}\}\mbox{ for some }k\ge 0\},\\
\mathcal{I}^{\mathbf n, \le m}&\ =\ \bigcup_{0\le k\le m} \mathcal{I}^{\mathbf n, (k)}\mbox{ and } \mathcal{I}^{\mathbf n} \ =\ \bigcup_{m\in\mathbb{N}_0} \mathcal{I}^{\mathbf n, (m)}. 
\end{align*}

\begin{defi}\normalfont
A basis $\mathcal{B}$ is said to be ($\mathbf n$, consecutive almost greedy) of type I (or CAG($\mathbf n$, I), for short) if there exists a constant $\mathbf C\ge 1$ such that 
$$\|x-G_m(x)\|\ \le\ \mathbf C\widecheck{\sigma}^{\mathbf n}_m(x), \forall x\in X, \forall m\in\mathbb{N}, \forall G_m(x),$$
where 
$$\widecheck{\sigma}^{\mathbf n}_m(x)\ =\ \inf\left\{\|x-P_{I}(x)\|: I\in \mathcal{I}^{\mathbf n, (m)}\right\}.$$
The least constant $\mathbf C$ is denoted by $\mathbf C_{\mathbf n, ca}$.
\end{defi}

\begin{rek}\normalfont
For the definition of CAG($\mathbf n$, I), we still project on consecutive vectors with indices in the sequence $\mathbf n$, but we do not restrict the projection to the first $m$ vectors in $\mathbf n$. 
\end{rek}

Consider the condition: there exists a constant $\mathbf C\ge 1$ such that
\begin{equation}\label{ee5}\|x-G_m(x)\|\ \le\ \mathbf C\min_{0\le k\le m}\widecheck{\sigma}^{\mathbf n}_k(x), \forall x\in X, \forall m\in\mathbb{N}, \forall G_m(x).\end{equation}

\begin{thm}[Generalization of Theorem 1.7 in \cite{BBC}]\label{mm5} Let $\mathcal{B}$ be a basis of a Banach space $X$.
The following statements are equivalent:
\begin{enumerate}
    \item[i)] $\mathcal{B}$ is ($\mathbf n$, almost greedy).
    \item[ii)] $\mathcal{B}$ satisfies \eqref{ee5}.
    \item[iii)] $\mathcal{B}$ is CAG($\mathbf n$, I).
\end{enumerate}
\end{thm}

\begin{proof}
By Proposition \ref{pp1}, i) $\Longrightarrow$ ii). That ii) $\Longrightarrow$ iii) is obvious. We shall show that iii) $\Longrightarrow$ i). Assume that $\mathcal{B}$ is $\mathbf C_{\mathbf n, ca}$-CAG($\mathbf n$, I). Using Theorem \ref{mm2}, we need to show that $\mathcal{B}$ is quasi-greedy and ($\mathbf n$, democratic).

i) $\mathcal{B}$ is quasi-greedy: Pick $x\in X$ and $m\in\mathbb{N}$. Let $I_N = \{n_{N+1}, \ldots, n_{N+m}\}$. Since $\mathcal{B}$ is CAG($\mathbf n$, I), 
$$\|x-G_m(x)\|\ \le\ \mathbf C_{\mathbf n, ca}\|x-P_{I_N}(x)\|\ \le\ \mathbf C_{\mathbf n, ca}(\|x\| + \|P_{I_N}(x)\|).$$
Clearly, 
$$\|P_{I_N}(x)\|\ \le\ m\max_{n > N}|e^*_n(x)|\sup_{\ell}\|e_\ell\|\ \le\ mc_2\max_{n>N}|e_n^*(x)|\rightarrow 0\mbox{ as }N\rightarrow 0.$$
Hence, 
$$\|x-G_m(x)\|\ \le\ \mathbf C_{\mathbf n, ca}\|x\|,$$
which shows that $\mathcal{B}$ is $\mathbf C_{\mathbf n, ca}$-suppression quasi-greedy. 

ii) $\mathcal{B}$ is ($\mathbf n$, democratic): Let $(A, B)\in\mathbb{S}(\mathbf n)$. First, assume that $|A|$ is even. Pick $I_1, I_2\in\mathcal{I}^{\mathbf n}$ such that $I_1\sqcup I_2$, $A\subset I_1\cup I_2$,  and $$|A\cap I_1|\ = \ |A\cap I_2|\ =\ |A|/2.$$
Then either $|B\backslash I_1|\ge |A|/2$ or $|B\backslash I_2|\ge |A|/2$. Otherwise, we have the contradiction
$$|A|\ \le\ |B|\ \le\ |B\backslash I_1| + |B\backslash I_2|\ <\ |A|.$$
Without loss of generality, assume that 
$|B\backslash I_1|\ge |A|/2$. Choose $B_1\subset B\backslash I_1$ such that $|B_1| = |A|/2$. Let 
$$x \ =\ 1_{A\cap I_1} + 1_{I_1\backslash A} + 1_{B_1}.$$
Since $\mathcal{B}$ is $\mathbf C_{\mathbf n, ca}$-CAG($\mathbf n$, I) and $\mathbf {C}_{\mathbf n, ca}$-suppression quasi-greedy, we have
$$\|1_{A\cap I_1}\|\ =\ \|x - P_{(I_1\backslash A)\cup B_1}(x)\|\ \le\ \mathbf C_{\mathbf n, ca}\|x-P_{I_1}(x)\|\ =\ \mathbf C_{\mathbf n, ca}\|1_{B_1}\|\ \le\ \mathbf C^2_{\mathbf n, ca}\|1_{B}\|.$$
Let 
$$y\ =\ 1_{A\cap I_2} + 1_{I_2\backslash A} + 1_{A\cap I_1}.$$
Observe that 
    $$|I_2| \ =\ |A\cap I_2| + |I_2\backslash A| \ =\ |A\cap I_1| + |I_2\backslash A|.$$
    We have
\begin{align*}\|1_{A\cap I_2}\|\ =\ \|y-P_{(I_2\backslash A)\cup (A\cap I_1)}(y)\|&\ \le\ \mathbf C_{\mathbf n, ca}\|y - P_{I_2}(y)\|\\
&\ \le\ \mathbf C_{\mathbf n, ca}\|1_{A\cap I_1}\|\ \le\ \mathbf C_{\mathbf n, ca}^3\|1_{B}\|.\end{align*}
    We obtain
    $$\|1_A\|\ \le\ \|1_{A\cap I_1}\| + \|1_{A\cap I_2}\|\ \le\ (\mathbf C_{\mathbf n, ca}^2 + \mathbf C_{\mathbf n, ca}^3)\|1_B\|.$$
    
We consider the case when $|A|$ is odd. Let $A' = A\backslash \{\max A\}$. By above, we have
\begin{align*}
    \|1_{A'}\|\ \le\ \|1_A\| + c_2&\ \le\ (\mathbf C_{\mathbf n, ca}^2 + \mathbf C_{\mathbf n, ca}^3)\|1_B\| + c_2\sup_{k}\|e_k^*\|\|1_B\|\\
    &\ \le\ (\mathbf C_{\mathbf n, ca}^2 + \mathbf C_{\mathbf n, ca}^3 + c_2^2)\|1_B\|.
\end{align*}

Therefore, $\mathcal{B}$ is ($\mathbf n$, democratic). 
\end{proof}

In order to obtain tight estimates as in \cite[Theorem 3.3]{AA}, \cite[Theorem 2]{DKOSZ}, \cite[Theorems 4.1 and 4.3]{BDKOW}, and \cite[Theorems 1.14 and 1.15]{C2}, we define another type of ($\mathbf n$, consecutive almost greedy) bases and what we call ($\mathbf n$, restricted symmetry for largest coefficients) (or ($\mathbf n$, RSLC)).

\begin{defi}\normalfont
A basis $\mathcal{B}$ is said to be ($\mathbf n$, consecutive almost greedy) of type II (or CAG($\mathbf n$, II), for short) if there exists a constant $\mathbf C\ge 1$ such that 
\begin{equation}\label{ee6}\|x-G_m(x)\|\ \le\ \mathbf C\overline{\sigma}^{\mathbf n}_m(x), \forall x\in X, \forall m\in\mathbb{N}, \forall G_m(x),\end{equation}
where 
$$\overline{\sigma}^{\mathbf n}_m(x)\ =\ \inf\left\{\|x-P_{I}(x)\|: I\in \mathcal{I}^{\mathbf n}, |I\cap \supp(x)|\le m\right\}.$$
The least constant $\mathbf C$ is denoted by $\mathbf C'_{\mathbf n, ca}$.
\end{defi}

\begin{defi}\normalfont\label{dnRSLC}
A basis $\mathcal{B}$ is ($\mathbf n$, RSLC) if there exists the least constant $\mathbf C_{\mathbf n, rs}\ge 1$ such that
$$\|x+1_{\varepsilon A}\|\ \le\ \mathbf C_{\mathbf n, rs}\|x+1_{\delta B}\|,$$
for all $x\in X$ with $\|x\|_\infty\le 1$, for all $(A, B)\in\mathbb{S}(\mathbf n)$ with $A < (\supp(x)\sqcup B)|_{\min A}\cap \mathbf n$, and for all signs $\varepsilon, \delta$.
\end{defi}

\begin{prop}\label{pp2}
A basis $\mathcal{B}$ is $\mathbf C_{\mathbf n, rs}$-($\mathbf n$, RSLC) if and only if 
$$\|x\|\ \le\ \mathbf C_{\mathbf n, rs}\|x-P_A(x)+1_{\varepsilon B}\|,$$
for all $x\in X$ with $\|x\|_\infty\le 1$, for all $(A, B)\in\mathbb{S}(\mathbf n)$ with $A < (\supp(x-P_A(x))\sqcup B)|_{\min A}\cap \mathbf n$, and for all signs $\varepsilon$. 
\end{prop}

\begin{thm}\label{mm1}
Let $\mathcal{B}$ be a basis of a Banach space. The following hold
\begin{enumerate}
    \item[i)] If $\mathcal{B}$ is $\mathbf C'_{\mathbf n, ca}$-CAG($\mathbf n$, II), then $\mathcal{B}$ is $\mathbf C'_{\mathbf n, ca}$-suppression quasi-greedy and $\mathbf C'_{\mathbf n, ca}$-($\mathbf n$, RSLC).
    \item[ii)] If $\mathcal{B}$ is $\mathbf C_\ell$-suppression quasi-greedy and $\mathbf C_{\mathbf n, rs}$-($\mathbf n$, RSLC), then the basis $\mathcal{B}$ is $\mathbf C_\ell\mathbf C_{\mathbf n, rs}$-CAG($\mathbf n$, II).
\end{enumerate}
\end{thm}

\begin{proof}
(1) Assume that $\mathcal{B}$ is $\mathbf C'_{\mathbf n, ca}$-CAG(II). Taking $I = \emptyset$ in \eqref{ee6}, we know that $\mathcal{B}$ is $\mathbf C'_{\mathbf n, ca}$-suppression quasi-greedy. We shall prove that $\mathcal{B}$ is $\mathbf C'_{\mathbf n, ca}$-($\mathbf n$, RSLC). Let $x, A, B, \varepsilon, \delta$ be chosen as in Definition \ref{dnRSLC}. Set $y = x + 1_{\varepsilon A} + 1_{\delta B}$. If $A = \emptyset$, then $\|x\|\le \mathbf C'_{\mathbf n, ca}\|x+1_{\delta B}\|$ because $\mathcal{B}$ is $\mathbf C'_{\mathbf n, ca}$-suppression quasi-greedy. Otherwise, let $D = [\min A, \max A]\cap \mathbf n$. 

\begin{claim}
We have $D\cap \supp(y) = A$.
\end{claim}

\begin{proof}
Clearly, $A\subset D\cap \supp(y)$. Suppose, for a contradiction, that there exists $j\in (D\cap \supp(y))\backslash A$. Then $j\in (\supp(x)\cup B)\cap \mathbf n$ and $\min A < j < \max A$, which contradicts that $A < (\supp(x)\sqcup B)|_{\min A}\cap \mathbf n$.
\end{proof}
Hence, $|D\cap \supp(y)| = |A|\le |B|$. We have
$$\|x+ 1_{\varepsilon A}\| \ =\ \|y - P_B(y)\|\ \le\ \mathbf C'_{\mathbf n, ca}\overline{\sigma}^{\mathbf n}_{|B|}(y)\ =\ \mathbf C'_{\mathbf n, ca}\|y-P_D(y)\|\ =\ \mathbf C'_{\mathbf n, ca}\|x + 1_{\delta B}\|.$$
Therefore, $\mathcal{B}$ is $\mathbf C'_{\mathbf n, ca}$-($\mathbf n$, RSLC). 

(2) Assume that $\mathcal{B}$ is $\mathbf C_\ell$-suppression quasi-greedy and $\mathbf C_{\mathbf n, rs}$-($\mathbf n$, RSLC). We show that $\mathcal{B}$ is $\mathbf C_\ell\mathbf C_{rs}$-CAG($\mathbf n$, II). Let $x\in X$, $A\in G(x, m)$ for some $m\in\mathbb{N}$ and $A\subset \supp(x)$. Let $B = \{n_{q+1}, n_{q+2}, \ldots, n_{q+p}\}$ with $|B\cap \supp(x)| \le |A|$. We need to show that
$$\|x-P_A(x)\|\ \le\ \mathbf C_{\mathbf n, rs}\mathbf C_\ell\|x-P_B(x)\|.$$ Set $D = B\cap \supp(x)$, $\varepsilon = (\sgn(e_n^*(x)))_n$. 
Let $E = D\backslash A$ and $F = A\backslash D$. Since $|D|\le |A|$, we have $|E|\le |F|$. Furthermore, $E\subset\mathbf n$. If $E = \emptyset$, then $D\subset A$. We have
$$\|x-P_A(x)\| \ =\ \|x-P_D(x) - P_{A\backslash D}(x)\|\ \le\ \mathbf C_\ell\|x-P_D(x)\|$$
because $A\backslash D$ is a greedy set of $x-P_D(x)$. For the rest of the proof, assume that $E\neq \emptyset$.

\begin{claim}
We have $F|_{\min E}\cap \mathbf n > E$. 
\end{claim}

\begin{proof}
We have
\begin{align*}
    \min E&\ \ge\ \min D\ \ge\ \min B\ \ge\ n_{q+1},\\
    \max E&\ \le\ \max D\ \le \ \max B\ \le\ n_{q+p}.
\end{align*}
Hence, it suffices to show that $F|_{n_{q+1}}\cap \mathbf n > n_{q+p}$. Suppose, for a contradiction, that there exists $\ell\in F\cap \mathbf n$ such that $\ell \ge n_{q+1}$ and $\ell \le n_{q+p}$. Hence, $\ell\in B\cap F$, which implies that $\ell\notin \supp(x)$. However, that $\ell\in F$ implies that $\ell\in A\subset\supp(x)$, a contradiction.
\end{proof}

\begin{claim}
We have $\supp(x-P_A(x)-P_{E}(x))|_{\min E} \cap \mathbf n > E$.
\end{claim}

\begin{proof}
Since $\min E \ge n_{q+1}$ and $\max E\le n_{q+p}$, it suffices to show that 
$$\supp(x-P_{A\cup E}(x))|_{n_{q+1}}\cap \mathbf n\ >\ n_{q+p}.$$
Suppose, for a contradiction, that there exists $n_\ell\in \supp(x-P_{A\cup E}(x))$ such that $n_{q+1}\le n_\ell \le n_{q+p}$. Then $n_\ell\in B\cap \supp(x) = D$. However, $n_\ell\in \supp(x-P_{A\cup E}(x))$ implies that $n_\ell\notin A\cup E = A\cup D$. Hence, $n_\ell\notin D$, a contradiction. 
\end{proof}

Set $\alpha = \min_{n\in F}|e_n^*(x)|$. By the above claims, we can apply Proposition \ref{pp2} and Theorem \ref{bto} to have
\begin{align*}
    \|x-P_A(x)\|&\ \le\ \mathbf C_{\mathbf n, rs}\|x-P_A(x)-P_{E}(x) + \alpha 1_{\varepsilon F}\|\\
    &\ =\ \mathbf C_{\mathbf n, rs}\|T_\alpha(x-P_A(x)-P_E(x)+P_F(x))\|\\
    &\ \le\ \mathbf C_{\mathbf n, rs}\mathbf C_\ell\|x-P_D(x)\|\ =\ \mathbf C_{\mathbf n, rs}\mathbf C_\ell\|x-P_B(x)\|,
\end{align*}
as desired. This completes our proof. 
\end{proof}

\begin{defi}\normalfont\label{dnslc}
A basis $\mathcal{B}$ is ($\mathbf n$, symmetric for largest coefficients) (or ($\mathbf n$, SLC), for short) if there exists the least constant $\mathbf C_{\mathbf n, s}$ such that 
$$\|x+1_{\varepsilon A}\|\ \le\ \mathbf C_{\mathbf n, s}\|x+1_{\delta B}\|,$$
for all $x\in X$ with $\|x\|_\infty \le 1$, for all $(A, B)\in \mathbb{S}(\mathbf n)$ with $A\sqcup B\sqcup \supp(x)$, and for all signs $\varepsilon, \delta$.
\end{defi}

\begin{prop}\label{pp3}Let $\mathcal{B}$ be $\mathbf C_q$-quasi-greedy. Then $\mathcal{B}$ is $\mathbf C_{\mathbf n, s}$-($\mathbf n$, SLC) if and only if $\mathcal{B}$ is $\mathbf C_{\mathbf n, rs}$-($\mathbf n$, RSLC). Moreover, $\mathbf C_{\mathbf n, rs}\le \mathbf C_{\mathbf n, s}$ and 
$$\mathbf C_{\mathbf n, s}\ \le\ 1+(1 + \mathbf C^2_{\mathbf n, rs})\mathbf C_q.$$
\end{prop}

\begin{proof}
Observe that if $(A, B)\in\mathbb{S}(\mathbf n)$ and $A < (B\sqcup \supp(x))|_{\min A}\cap \mathbf n$, then $A\sqcup B\sqcup \supp(x)$. Hence, if $\mathcal{B}$ is $\mathbf C_{\mathbf n, s}$-($\mathbf n$, SLC), then $\mathcal{B}$ is $\mathbf C_{\mathbf n, s}$-($\mathbf n$, RSLC). 

We prove the implication ($\mathbf n$, RSLC) $\Longrightarrow$ ($\mathbf n$, SLC). 
Assume that $\mathcal{B}$ is $\mathbf C_{\mathbf n, rs}$-($\mathbf n$, RSLC). 

\begin{claim}\label{cl1}
If $(A, B)\in \mathbb{S}(\mathbf n)$, then $\|1_{\varepsilon A}\|\le \mathbf C^2_{\mathbf n, rs}\|1_{\delta B}\|$ for all signs $\varepsilon, \delta$.
\end{claim}
\begin{proof}
Choose $D\subset\mathbf n$ such that $A\cup B < D$ and $|D| = |A|$. Due to ($\mathbf n$, RSLC), we have
$$\|1_{\varepsilon A}\|\ \le\ \mathbf C_{\mathbf n, rs}\|1_D\|\mbox{ and }\|1_D\|\ \le\ \mathbf C_{\mathbf n, rs}\|1_{\delta B}\|.$$
Therefore, $\|1_{\varepsilon A}\|\le \mathbf C_{\mathbf n, rs}^2\|1_{\delta B}\|$.
\end{proof}
We are ready to show that $\mathcal{B}$ is ($\mathbf n$, SLC). Let $x, A, B, \varepsilon, \delta$ be chosen as in Definition \ref{dnslc}. We have
\begin{align*}
    \|x+1_{\varepsilon A}\|&\ \le\ \|x + 1_{\delta B}\| + \|1_{\delta B}\| +  \|1_{\varepsilon A}\|\\
    &\ \le\ (1+\mathbf C_q)\|x+1_{\delta B}\| + \mathbf C^2_{\mathbf n, rs}\|1_{\delta B}\| \mbox{ by Claim \ref{cl1}}\\
    &\ \le\ (1+\mathbf C_q + \mathbf C^2_{\mathbf n, rs}\mathbf C_q)\|x+1_{\delta B}\|.
\end{align*}
This shows that $\mathcal{B}$ is ($\mathbf n$, SLC). 
\end{proof}

\begin{cor}
Let $\mathcal{B}$ be a basis of a Banach space. The following are equivalent
\begin{enumerate}
    \item[i)] $\mathcal{B}$ is ($\mathbf n$, almost greedy).
    \item[ii)] $\mathcal{B}$ is CAG($\mathbf n$, I).
    \item[iii)] $\mathcal{B}$ is CAG($\mathbf n$, II).
    \item[iv)] $\mathcal{B}$ is quasi-greedy and ($\mathbf n$, RSLC).
    \item[v)] $\mathcal{B}$ is quasi-greedy and ($\mathbf n$, SLC).
    \item[vi)] $\mathcal{B}$ is quasi-greedy and ($\mathbf n$, democratic). 
    \item[vii)] $\mathcal{B}$ is quasi-greedy and ($\mathbf n$, superdemocratic). 
\end{enumerate}
\end{cor}

\begin{proof}
Theorems \ref{mm2} and \ref{mm5}  give i) $\Longleftrightarrow$ ii) $\Longleftrightarrow$ vi). 
That iii) $\Longleftrightarrow$ iv) $\Longleftrightarrow$ v) is due to Theorem \ref{mm1} and 
Proposition \ref{pp3}. Claim \ref{cl1} shows that ($\mathbf n$, RSLC) $\Longrightarrow$ ($\mathbf n$, democratic), so iv) $\Longrightarrow$ vi). Furthermore, the proof of Proposition \ref{pp3} gives that vi) $\Longrightarrow$ v). Finally, Remark \ref{r1} gives vi) $\Longleftrightarrow$ vii).  
\end{proof}

\subsection{1-CAG($\mathbf n$, I) and 1-CAG($\mathbf n$, II) bases}
First, we characterize $1$-($\mathbf n$, almost greedy) bases and $1$-($\mathbf n$, SLC) bases.
\begin{prop}\label{ppp1}
A basis $\mathcal{B}$ is $1$-($\mathbf n$, almost greedy) if and only if for any $x\in X$ and $j\in\mathbf n$, 
\begin{equation}\label{ee13}\|x-G_1(x)\|\ \le\ \|x-e^*_j(x)e_j\|.\end{equation}
\end{prop}
\begin{proof}
If $\mathcal{B}$ is $1$-($\mathbf n$, almost greedy), then clearly, $\mathcal{B}$ satisfies \eqref{ee13}. We prove the converse. 
Let $A$ be a greedy set of $x$ of order $m$ and $B \subset \mathbf n$ with $|B| = m$. Let $D = A\cap B$ and $y = x - P_D(x)$. We have
$$x-P_A(x)\ =\ y - P_{A\backslash D}(y)\mbox{ and }x - P_B(x) \ =\ y - P_{B\backslash D}(y).$$
Note that $(A\backslash D) \sqcup (B\backslash D)$ and $|A\backslash D| = |B\backslash D|$. For ease of notation, let $E = A\backslash D$ and $F = B\backslash D$. Write $E = \{k_1, \ldots, k_{|E|}\}$, where $|e^*_{k_1}(y)|\ge \cdots\ge|e^*_{k_{|E|}}(y)|$ and write $F = \{\ell_1, \ldots, \ell_{|F|}\}$. Since $E$ is a greedy set of $y$, we have
\begin{eqnarray*}\|y-P_E(y)\|&\ =\ &\|y - P_{E\backslash \{k_{|E|}\}}(y) - e^*_{k_{|E|}}(y -  P_{E\backslash \{k_{|E|}\}}(y))e_{k_{|E|}}\|\\
&\ \stackrel{\eqref{ee13}}{\le}\ &\|y-P_{E\backslash \{k_{|E|}\}}(y) - e^*_{\ell_1}(y-P_{E\backslash \{k_{|E|}\}}(y))e_{\ell_1}\|\\
&\ \le\ & \|y-P_{E\backslash \{k_{|E|}\}}(y) - e^*_{\ell_1}(y)e_{\ell_1}\|\\
&\ \le\ &\cdots \ \le\ \|y  - P_{F}(y)\|. 
\end{eqnarray*}
Substituting $y = x-P_{A\cap B}(x)$ to get 
$$\|x-P_A(x)\|\ \le\ \|x-P_B(x)\|.$$
This completes our proof. 
\end{proof}

\begin{prop}\label{ppp2}
A basis $\mathcal{B}$ is $1$-($\mathbf n$, SLC) if and only if for any $x\in X$ with $\|x\|_\infty\le 1$, we have
\begin{equation}\label{ee10}\|x+se_j\|\ \le\ \|x+te_k\|,\end{equation}
for any $j\in \mathbf n$,  $\{k\}\sqcup \{j\}\sqcup \supp(x)$, and for any scalars $s, t$ such that $|s| = |t| = 1$. 
\end{prop}

\begin{proof}
Clearly, if $\mathcal{B}$ is $1$-($\mathbf n$, SLC), then $\mathcal{B}$ satisfies \eqref{ee10}. Conversely, assume that $\mathcal{B}$ satisfies \eqref{ee10}. We claim that $\mathcal{B}$ satisfies the following condition
\begin{equation}\label{ee11}\|x\|\ \le\ \|x+te_k\|, \end{equation}
for any $k\in\mathbb{N}\backslash \supp(x)$ and for any scalar $t$ with $|t| = 1$. 
Indeed, let $x, t, k$ be chosen as in \eqref{ee11}. Pick $j\in\mathbf n\cap (\mathbb N\backslash (\supp(x)\cup \{k\}))$ (here we use density to assume $\supp(x)$ is finite). We have
$$\|x\|\ \le\ \sup_{|s|=1}\|x+se_j\|\ \stackrel{\eqref{e10}}{\le}\ \|x+te_k\|.$$
Let $x, A, B, \varepsilon, \delta$ be chosen as in Definition \ref{dnslc}. If $A = \emptyset$, then we use \eqref{ee11} inductively to show that $\|x\|\le \|x+1_{\delta B}\|$. Suppose that $|A| \ge 1$. Write $A = \{a_1, a_2, \ldots, a_{|A|}\}$ and $B = \{b_1, b_2, \ldots, b_{|B|}\}$. We have
\begin{align*}
\|x+1_{\varepsilon A}\|\ =\ \|x + 1_{\varepsilon (A\backslash \{a_1\})} + \varepsilon_{a_1}e_{a_1}\|&\ \stackrel{\eqref{ee10}}{\le}\ \|x+ 1_{\varepsilon (A\backslash \{a_1\})}+\delta_{b_1}e_{b_1}\|\\
&\ \stackrel{\eqref{ee10}}{\le}\ \cdots \ \stackrel{\eqref{ee10}}{\le}\ \|x + 1_{\delta \{b_1, \ldots, b_{|A|}\}}\|\\
&\ \stackrel{\eqref{ee11}}{\le}\ \|x+1_{\delta B}\|.
\end{align*}
Hence, $\mathcal{B}$ is $1$-($\mathbf n$, SLC). 
\end{proof}

\begin{cor}
If a basis $\mathcal{B}$ is $1$-($\mathbf n$, SLC), then it is $1$-suppression quasi-greedy. 
\end{cor}
\begin{proof}
Use \eqref{ee11} inductively. 
\end{proof}

\begin{thm}[Generalization of Theorem 2.3 in \cite{AA}]\label{mm50}
A basis $\mathcal{B}$ is $1$-($\mathbf n$, almost greedy) if and only if $\mathcal{B}$ is $1$-($\mathbf n$, SLC). 
\end{thm}

\begin{proof}
Assume that $\mathcal{B}$ is $1$-($\mathbf n$, almost greedy). Let $x, j, k, s, t$ be as in \eqref{ee10} and set $y= x + se_j + te_k$. We have 
$$\|x+se_j\|\ =\ \|y-te_k\|\ \le\ \widetilde{\sigma}^{\mathbf n}_1(y)\ \le\ \|y-se_j\|\ =\ \|x+te_k\|.$$
Hence, $\mathcal{B}$ satisfies \eqref{ee10} and thus, is $1$-($\mathbf n$, SLC). 

Next, assume that $\mathcal{B}$ is $1$-($\mathbf n$, SLC). We follow the proof of \cite[Theorem 2.3]{AA}. Let $x\in X$, $j\in\mathbf n$, and $k\in\mathbb{N}$ such that $G_1(x) = e_k^*(x)e_k$. By Proposition \ref{ppp1}, it suffices to show that $\|x-e_k^*(x)e_k\|\le \|x-e_j^*(x)e_j\|$. Assume that $k\neq j$. Let $y = x - e_j^*(x)e_j - e_k^*(x)e_k$. Since $|e_j^*(x)|\le |e_k^*(x)|$, we know that
\begin{eqnarray*}\|x-e_k^*(x)e_k\| \ =\ \|y + e_j^*(x)e_j\|\ \le\ \sup_{|s| = 1}\|y + se_k^*(x)e_j\|&\ \stackrel{\eqref{ee10}}{\le}\ & \|y + e_k^*(x)e_k\|\\
&\ =\ &\|x-e_j^*(x)e_j\|,
\end{eqnarray*}
as desired.
\end{proof}

We now show that being $1$-CAG($\mathbf n$, I) and being $1$-CAG($\mathbf n$, II) are both equivalent to being $1$-($\mathbf n$, almost greedy).   
\begin{prop}\label{ppp3}
A basis $\mathcal{B}$ is 1-CAG($\mathbf n$, II) if and only if $\mathcal{B}$ is $1$-($\mathbf n$, RSLC). 
\end{prop}
\begin{proof}
According to Theorem \ref{mm1}, we need only to prove that a $1$-($\mathbf n$, RSLC) basis is $1$-suppression quasi-greedy. In Definition \ref{dnRSLC}, set $A = \emptyset$ and let $B$ be a singleton to see that $1$-($\mathbf n$, RSLC) implies \eqref{ee11}, which, by induction, shows that $\mathcal{B}$ is $1$-suppression quasi-greedy. 
\end{proof}

\begin{prop}\label{ppp20}
A basis $\mathcal{B}$ is $1$-($\mathbf n$, SLC) if and only if $\mathcal{B}$ is $1$-($\mathbf n$, RSLC).
\end{prop}
\begin{proof}
By definitions, $1$-($\mathbf n$, SLC) $\Longrightarrow$ $1$-($\mathbf n$, RSLC). Conversely, assume that $\mathcal{B}$ is $1$-($\mathbf n$, RSLC). Take $x, j, k, s, t$ as in \eqref{ee10}. Using the notation in Definition \ref{dnRSLC}, we set 
$A = \{j\}$ and  $B = \{k\}$. Clearly, $A < (\supp(x)\sqcup B)|_{\min A}\cap \mathbf n$. Therefore, 
$$\|x + se_j\|\ \le\ \|x+te_k\|.$$
By Proposition \ref{ppp2}, $\mathcal{B}$ is $1$-($\mathbf n$, SLC). 
\end{proof}

\begin{cor}
Let $\mathcal{B}$ be a basis of a Banach space and $\mathbf n$ be a sequence. The following are equivalent
\begin{enumerate}
    \item [i)] $\mathcal{B}$ is $1$-($\mathbf n$, almost greedy).
    \item [ii)] $\mathcal{B}$ is $1$-CAG($\mathbf n$, I).
    \item [iii)] $\mathcal{B}$ is $1$-CAG($\mathbf n$, II).
    \item [iv)] $\mathcal{B}$ is $1$-($\mathbf n$, SLC).
    \item [v)] $\mathcal{B}$ is $1$-($\mathbf n$, RSLC). 
\end{enumerate}
\end{cor}

\begin{proof}
By Theorem \ref{mm50} and Propositions \ref{ppp3} and \ref{ppp20}, i) $\Longleftrightarrow$ iv) $\Longleftrightarrow$ v) $\Longleftrightarrow$ iii). By definitions, we know that iii) $\Longrightarrow$ ii). We show that ii) $\Longrightarrow$ iv) or equivalently, ii) $\Longrightarrow$ \eqref{ee10}. Choose $x, j, k, s, t$ as in \eqref{ee10}. We have
$$\|x + se_j\|\ =\ \|(x+se_j + te_k) - te_k\|\ \le\ \widecheck{\sigma}^{\mathbf n}_1(x+se_j +te_k)\ \le\ \|x+te_k\|.$$
This completes our proof. 
\end{proof}

\section{On ($\mathbf n$, strong partially greedy) bases with gaps}\label{gaps}

\subsection{The theory of greedy-type properties with gaps}

In 2017, Oikhberg \cite{O} defined and studied the $\mathbf n$-quasi-greedy property, a variant of the quasi-greedy property, as follows: let $\mathbf s = s_1, s_2, \ldots$ be a strictly increasing sequence of positive integers\footnote{In the literature of gaps, the standard notation for a gap sequence is $\mathbf n$; however, we use $\mathbf s$ to not confuse ourselves with our sequence $\mathbf n$ in the definition of ($\mathbf n$, strong partially greedy) bases.}; a basis $\mathcal{B}$ is said to be $\mathbf s$-quasi-greedy if 
\begin{equation}\label{e50}\lim_{i} G_{s_i}(x) \ =\ x, \forall x\in X, \forall (G_{s_i}(x))_i.\end{equation}
\cite[Theorem 2.1]{O} states that for a Markushevich basis, \eqref{e50} is equivalent to the condition: there exists a constant $\mathbf C > 0$ such that 
$$\|G_{s_i}(x)\|\ \le\ \mathbf C\|x\|, \forall x\in X,\forall i\in\mathbb{N}, \forall G_{s_i}(x).$$
Clearly, a quasi-greedy basis is $\mathbf s$-quasi-greedy for any sequence $\mathbf s$. However, the reverse is not true. 
\begin{defi}\normalfont
A strictly increasing sequence $\mathbf s$ is said to have bounded quotient gaps if there exists $\ell\in\mathbb{N}_{\ge 2}$ such that
$$\sup_{k}\frac{s_{k+1}}{s_k}\ \le\ \ell.$$
We then say that $\mathbf s$ has $\ell$-bounded quotient gaps. On the other hand, if
$$\sup_{k}\frac{s_{k+1}}{s_k}\ =\ \infty,$$
we say that $\mathbf s$ has arbitrarily large quotient gaps. 
\end{defi}

\begin{defi}\normalfont
A strictly increasing sequence $\mathbf n$ is said to have bounded additive gaps if there is $\ell\in\mathbb{N}$ such that
$$\max_{k}(n_{k+1}-n_k) \ \le\ \ell.$$
We then say that $\mathbf n$ has $\ell$-bounded additive gaps. On the other hand, if 
$$\sup_{k}(n_{k+1}-n_k) \ =\ \infty,$$
we say that $\mathbf n$ has arbitrary large additive gaps. 
\end{defi}

\begin{thm}\cite[Proposition 3.1]{O}
If $\mathbf s$ has arbitrarily large quotient gaps, there exists a Banach space with an $\mathbf s$-quasi-greedy Schauder basis that is not quasi-greedy.  
\end{thm}

Interestingly, if $\mathbf s$ has bounded quotient gaps, then for a Schauder basis, the $\mathbf s$-quasi-greedy property is equivalent to the quasi-greedy property (see \cite[Theorem 5.2]{BB}). Recently, research on greedy-type properties with gaps has made much progress. For example, Berasategui and Bern\'{a} \cite{BB3} investigated $\mathbf s$-democracy-like properties including 
unconditionality for constant coefficients, UL-property, democracy, symmetry for largest coefficients, to name a few.  
Meanwhile, \cite{BB1} studied more in depth $\mathbf s$-quasi-greedy bases, $\mathbf s$-bidemocracy (another popular and useful notion in the literature), $\mathbf s$-semi-greedy bases (first introduced in \cite{DKK}), and $\mathbf s$-strong partially greedy bases. 

By definition, we have the following implications, none of which can be reversed
$$\mbox{ greedy}\Longrightarrow \mbox{ almost greedy}\Longrightarrow \mbox{ strong partially greedy}\Longrightarrow\mbox{ quasi-greedy}.$$
\begin{defi}\normalfont
A basis $\mathcal{B}$ is $\mathbf s$-greedy if there exists a constant $\mathbf C \ge 1$ such that 
$$\|x-G_m(x)\|\ \le\ \mathbf C\sigma_m(x),\forall x\in X, \forall m\in\mathbf s, \forall G_m(x).$$
The definition of $\mathbf s$-almost greedy bases replaces $\sigma_m(x)$ with $\widetilde{\sigma}_m(x)$.
\end{defi}

By \cite[Remark 1.1]{O}, the greedy (almost greedy, resp.) property is so strong that for any sequence $\mathbf s$, the $\mathbf s$-greedy ($\mathbf s$-almost greedy, resp.) property implies greedy (almost greedy, resp.) property. Berasategui and Lassalle \cite{BL} proved that a Markushevich basis is almost greedy if and only if it is semi-greedy, which gave the conjecture that for any sequence $\mathbf s$, an $\mathbf s$-semi-greedy Markushevich basis is also semi-greedy. Indeed, \cite[Theorem 5.2]{BB1} confirmed the conjecture. However, the strong partially greedy property and the quasi-greedy property are not as strong for the same conclusion to hold. 

\begin{defi}\normalfont
A basis $\mathcal B$ is $\mathbf s$-($\mathbf n$, strong partially greedy) if there exists $\mathbf C\ge 1$ such that
$$\|x-G_m (x)\|\ \le\ \mathbf C\widehat{\sigma}^{\mathbf n}_m(x), \forall x\in X, \forall m\in \mathbf{s}, \forall G_m(x).$$
\end{defi}

For pedagogical purpose, we use the term ``($\mathbb N$, strong partially greedy)" in place of ``strong partially greedy" and ``($\mathbb N$, conservative)" in place of "conservative" (Definitions \ref{d10} and \ref{d11}).
Berasategui and Bern\'{a} \cite{BB1} proved the following proposition. 

\begin{prop}\cite[Proposition 6.9 and Proposition 6.14]{BB1}\label{p21}
\begin{enumerate}
    \item[i)] Let $\mathbf s$ be a sequence with arbitrarily large quotient gaps. There is a Banach space $X$ with a monotone Schauder basis $\mathcal{B}$ that is $\mathbf s$-($\mathbb{N}$, strong partially greedy), but the basis is neither ($\mathbb{N}$, conservative) nor quasi-greedy.
    \item[ii)] Let $\mathbf s$ be a sequence with arbitrarily large additive gaps. There is a Banach space $X$ with a $1$-unconditional basis $\mathcal{B}$ that is $\mathbf s$-($\mathbb{N}$, strong partially greedy), but the basis is not ($\mathbb{N}$, conservative), and thus not strong partially greedy. 
\end{enumerate}
\end{prop}

\begin{rek}\normalfont
Unlike quasi-greedy bases, even when $\mathbf s$ has bounded quotient gaps, an $\mathbf s$-($\mathbb N$, strong partially greedy) Schauder basis is not necessarily ($\mathbb N$, strong partially greedy). However, by \cite[Lemma 6.16]{BB1}, if $\mathbf s$ has bounded additive gaps, an $\mathbf s$-($\mathbb{N}$, strong partially greedy) Markushevich basis is ($\mathbb{N}$, strong partially greedy). 
\end{rek}

Berasategui and Bern\'{a} characterized $\mathbf s$-($\mathbb{N}$, strong partially greedy) Schauder bases when $\mathbf s$ has bounded quotient gaps. Let $\mathbb{T}(\mathbf n, \mathbf s)$ be the collection of all ordered pairs of finite sets $(A, B)$ such that $A\subset\mathbf n, |A|\le |B|, A\le n_s < B\cap \mathbf n$ for some $s\in \mathbf s$. Obviously,
$$\mathbb{T}(\mathbf n, \mathbf s)\ \subset\ \mathbb{T}(\mathbf n), \forall \mathbf n, \mathbf s.$$

\begin{rek}\normalfont\label{r3}
Let $(A, B)\in \mathbb{T}(\mathbf n)$. If $B\cap \mathbf n = \emptyset$, then $(A, B)\in \mathbb{T}(\mathbf n, \mathbf s)$. Indeed, choose $s\in\mathbf s$ such that $A < n_s$. Since $B\cap\mathbf n = \emptyset$, we get $A < n_s < B\cap \mathbf n$. Hence, $(A, B)\in \mathbb{T}(\mathbf n, \mathbf s)$.
\end{rek}

\begin{defi}\normalfont\label{defoc}
A basis $\mathcal{B}$ in a Banach space $X$ is $\mathbf s$-order-($\mathbf n$, superconservative) if there exists $\mathbf C > 0$ such that 
\begin{equation}\label{e60} \|1_{\varepsilon A}\| \ \le\ \mathbf C\|1_{\delta B}\|,\end{equation}
for all $(A, B)\in \mathbb{T}(\mathbf n, \mathbf s)$ and for all signs $\varepsilon$, $\delta$. The smallest constant satisfying \eqref{e60} is denoted by $\Delta^{\mathbf n}_{osc}$, and we say $\mathcal{B}$ is $\Delta^{\mathbf n}_{osc}$-$\mathbf s$-order-($\mathbf n$, superconservative). If \eqref{e60} holds for $\varepsilon \equiv \delta \equiv 1$, we say that $\mathcal{B}$ is $\Delta^{\mathbf n}_{oc}$-$\mathbf s$-order-($\mathbf n$, conservative), where $\Delta^{\mathbf n}_{oc}$ is the smallest constant for the inequality to hold. When $\mathbf n = \mathbb{N}$, we obtain \cite[Definition 6.5]{BB1} in the case of a Schauder basis (see \cite[Remark 6.6]{BB1}).
\end{defi}

\begin{prop}\cite[Proposition 6.13]{BB1}\label{p20}
Let $\mathcal{B}$ be a Schauder basis. If $\mathbf s$ has bounded quotient gaps, then $\mathcal{B}$ is $\mathbf s$-($\mathbb{N}$, strong partially greedy) if and only if $\mathcal{B}$ is quasi-greedy and $\mathbf s$-order-($\mathbb{N}$, superconservative). 
\end{prop}

\subsection{Generalization to $\mathbf s$-($\mathbf n$, strong partially greedy) bases}
We generalize the above results to ($\mathbf n$, strong partially greedy) bases with gaps. 

\subsubsection{$\mathbf n$-Schauder bases}

\begin{defi}\normalfont
A basis $\mathcal{B}$ is said to be $\mathbf n$-Schauder if there exists $\mathbf C>0$ such that 
$$\|P^{\mathbf n}_{m}(x)\|\ \le\ \mathbf C\|x\|, \forall x\in X, \forall m\in \mathbb{N}.$$
The least constant $\mathbf C$ is called the $\mathbf n$-basis constant. 
\end{defi}

\begin{prop}
If the difference set $\Delta_{\mathbf n, \mathbf m}$ is finite, then a basis is $\mathbf n$-Schauder if and only if it is $\mathbf m$-Schauder. 
\end{prop}
\begin{proof}
Assume that $\mathcal{B}$ is $\mathbf C$-$\mathbf n$-Schauder. Since $\Delta_{\mathbf n, \mathbf m}$ is finite, there exists $N\in\mathbb{N}$ such that $\mathbf n\cap [N,\infty)\subset\mathbf m$ and $\mathbf m\cap [N, \infty)\subset\mathbf n$. Pick $M\in\mathbb{N}$ and $x\in X$. Let $A_1 = \{m_1, \ldots, m_M\}\cap [N,\infty)$ and $A_2 = \{m_1, \ldots, m_M\}\backslash A_1$. We have 
$$\|P_{A_2}(x)\|\ \le\ \sup_{|A|< N}\|P_A\|\|x\|\ \le\ (N-1)\sup_{n}(\|e_n\|\|e_n^*\|)\|x\|\ \le\ (N-1)c_2^2\|x\|.$$
Furthermore, since $A_1\in \mathcal{I}^{\mathbf n}$ and $\mathcal{B}$ is $\mathbf n$-Schauder,
$$\|P_{A_1}(x)\|\ \le\ \|P_{\{n_1, \ldots, \max A_1\}}(x)\| + \|P_{\{n_1, \ldots, n_s\}}(x)\|\ \le\ 2\mathbf C\|x\|,$$
where $n_s$ is the number in $\mathbf n$ that is right before $\min A_1$. We obtain
$$\|P_{\{m_1, \ldots, m_M\}}(x)\|\ \le\ \|P_{A_1}(x)\| + \|P_{A_2}(x)\|\ \le\ ((N-1)c_2^2+2\mathbf C)\|x\|.$$
This shows that $\mathcal{B}$ is $\mathbf m$-Schauder.
\end{proof}

\begin{exa}\label{e111}\normalfont Let $\mathbf n = n_1, n_2, \ldots$ be a strictly increasing sequence such that $\mathbb{N}\backslash \mathbf n$ is infinite.
We give an example of a basis $\mathcal{B}$ that is $\mathbf n$-Schauder but is not Schauder. Let $(e_n)_n$ be a Markushevich basis of a space $X$ such that $(e_n)_n$ is not Schauder, i.e., there exists nonzero and normalized $(x_n)_{n}\subset X$ such that 
$\sup_{n}\|S_n(x_n)\|_{X} = \infty$.
Let $\mathbb {Y} = X\oplus c_0$ under the $\ell_1$-norm. Denote the canonical basis of $c_0$ to be $(f_n)_n$. Let $H: \mathbb{N}\backslash \mathbf n\rightarrow \mathbb{N}$ be the increasing bijection. Define the basis $\mathcal{B} = (g_n)_{n}$, where $g_{n_j} = (0, f_{j})$ and $g_{n} = (e_{H(n)}, 0)$ for $n\notin \mathbf n$. 

We check that $\mathcal{B}$ is $\mathbf n$-Schauder. Pick $y:= (x, x')\in \mathbb{Y}$ and $m\in \mathbb{N}$. We have
$$\|P^{\mathbf n}_m(y)\|_{\mathbb{Y}}\ =\ \left\|\sum_{j=1}^m g^*_{n_j}(y)g_{n_j}\right\|_{\mathbb{Y}}\ =\ \left\|\sum_{j=1}^m f_j^*(x')f_j\right\|_{c_0}\ \le\ \|x'\|_{c_0}\ \le\ \|y\|_{\mathbb{Y}}.$$

Now we show that $\mathcal{B}$ is not Schauder. Let $m = H^{-1}(n)$. Define $y_n: =  (x_n, 0)$. We have
$\|S_m(y_n)\|_{\mathbb{Y}} = \left\|S_n(x_n)\right\|_{X}, \mbox{ while }\|y_n\|_{\mathbb{Y}} = 1$.
Hence, $$\sup_m\frac{\|S_m(y_n)\|_{\mathbb{Y}}}{\|y_n\|_{\mathbb{Y}}}\ =\ \infty,$$ so $\mathcal{B}$ is not Schauder. 
\end{exa}

\begin{exa}\label{e112}\normalfont Let $\mathbf n = n_1, n_2, \ldots$ such that $\mathbb{N}\backslash \mathbf n$ is infinite. We give an example of a basis $\mathcal{B}$ that is Schauder but is not $\mathbf n$-Schauder. Since $\mathbb{N}\backslash \mathbf n$ is infinite, we can find a subsequence $\mathbf {n'} = n'_1, n'_2, \ldots$ of $\mathbf n$ such that $n'_j - 1\notin \mathbf n$ for all $j\in\mathbb{N}$.  Let $\mathbf m = m_1, m_2, \ldots$, where $m_{2j-1} = n'_j-1$ and $m_{2j} = n'_j$ for all $j\in\mathbb{N}$. Let $X$ be the completion of $c_{00}$ under the following norm: for $x = (x_1, x_2, \ldots)$, 
$$\|x\|\ =\ \sup_{N}\left|\sum_{j=1}^N x_{m_j}\right| + \max_{n\notin\mathbf m}|x_n|.$$
Let $\mathcal{B}$ be the canonical basis, which is clearly Schauder with basis constant $1$. To see that $\mathcal{B}$ is not $\mathbf n$-Schauder, consider the vector $x_{2n} = \sum_{i=1}^{2n} (-1)^i e_{m_i}$ for $n\in\mathbb{N}$. While $\|x_{2n}\| = 1$, for sufficiently large $N:= N_n$,
$\|P^{\mathbf n}_N(x_{2n})\| = n$. That $\lim_{n\rightarrow \infty}\frac{\|P^{\mathbf n}_N(x_{2n})\|}{\|x_{2n}\|}=\infty$ implies that $\mathcal{B}$ is not $\mathbf n$-Schauder.
\end{exa}

\begin{rek}\normalfont
Examples \ref{e111} and \ref{e112} can be generalized to any two sequences $\mathbf m$ and $\mathbf n$ having infinite difference set $\Delta_{\mathbf m, \mathbf n}$.
\end{rek}

\subsubsection{Characterization of $\mathbf s$-($\mathbf n$, strong partially greedy) bases}

\begin{prop}[Generalization of Proposition \ref{p20}]\label{p30}
Let $\mathcal{B}$ be both a Schauder and $\mathbf n$-Schauder basis. If $\mathbf s$ has bounded quotient gaps, then $\mathcal{B}$ is $\mathbf s$-($\mathbf n$, strong partially greedy) if and only if $\mathcal{B}$ is quasi-greedy and $\mathbf s$-order-($\mathbf n$, superconservative). 
\end{prop}

We give the definition and an useful characterization of $\mathbf s$-order-($\mathbf n$, PSLC), which shall be used in the proof of Proposition \ref{p30}.
\begin{defi}\normalfont\label{defpslc}
A basis $\mathcal{B}$ in a Banach space $X$ is $\mathbf s$-order-($\mathbf n$, PSLC) if there exists $\mathbf C \ge 1$ such that
\begin{equation}\label{e66} \|x+1_{\varepsilon A}\| \ \le\ \mathbf C\|x+1_{\delta B}\|,\end{equation}
for all $(A, B)\in \mathbb{T}(\mathbf n, \mathbf s)$, for all signs $\varepsilon$, $\delta$, and for all $x\in X$ with $\|x\|_\infty\le 1$ and $A < (\supp(x)\sqcup B)\cap \mathbf n$. The smallest constant satisfying \eqref{e66} is denoted by $\Delta^{\mathbf n}_{pl}$.
\end{defi}

\begin{prop}\label{x}
A basis $\mathcal{B}$ is $\Delta^{\mathbf n}_{pl}$-$\mathbf s$-order-($\mathbf n$, PSLC) if and only if 
\begin{equation}\label{e70}\|x\|\ \le\ \Delta^{\mathbf n}_{pl}\|x-P_A(x) + 1_{\varepsilon B}\|,\end{equation}
for all $(A, B)\in \mathbb{T}(\mathbf n, \mathbf s)$, for all signs $\varepsilon$, and for all $x\in X$ with $\|x\|_\infty\le 1$ and $A < (\supp(x-P_A(x))\sqcup B)\cap \mathbf n$.
\end{prop}

\begin{proof}
Suppose that $\mathcal{B}$ satisfies \eqref{e70}. Choose $x, A, B, \varepsilon, \delta$ as in Definition \ref{defpslc}. Let $y = x+ 1_{\varepsilon A}$. We have
$$\|x+1_{\varepsilon A}\|\ =\ \|y\| \ \stackrel{\eqref{e70}}{\le}\ \Delta^{\mathbf n}_{pl}\|y - P_A(y) + 1_{\delta B}\|\ =\ \Delta^{\mathbf n}_{pl}\|x+1_{\delta B}\|.$$

Conversely, suppose that $\mathcal{B}$ is $\Delta^{\mathbf n}_{pl}$-$\mathbf s$-order-($\mathbf n$, PSLC). Choose $x, A, B, \varepsilon$ as in \eqref{e70}. We have
\begin{align*}
    \|x\|\ =\ \left\|x-P_A(x) + \sum_{n\in A}e_n^*(x)e_n\right\|&\ \le\ \sup_{(\delta)}\|x-P_A(x)+1_{\delta A}\|\\
    &\ \le\ \Delta^{\mathbf n}_{pl}\|x-P_A(x) + 1_{\varepsilon B}\|.
\end{align*}
This completes our proof. 
\end{proof}

\begin{lem}\label{ll1}
If $\mathcal{B}$ be $\mathbf C_q$-quasi-greedy and $\Delta^{\mathbf n}_{osc}$-$\mathbf s$-order-($\mathbf n$, superconservative), then $\mathcal{B}$ is $\Delta^{\mathbf n}_{pl}$-$\mathbf s$-order-($\mathbf n$, PSLC) with $\Delta^{\mathbf n}_{pl}\le 1+\mathbf C_q + \Delta^{\mathbf n}_{osc}\mathbf C_q$. 
\end{lem}

\begin{proof}
Let $x, A, B, \varepsilon, \delta$ be chosen as in Definition \ref{defpslc}. We have
$$\|x\|\ \le\ \|x+1_{\delta B}\| + \|1_{\delta B}\|\ \le\ (1+\mathbf C_q)\|x+1_{\delta B}\|.$$
Furthermore, 
$$\|1_{\varepsilon A}\|\ \le\ \Delta^{\mathbf n}_{osc}\|1_{\delta B}\|\ \le\ \Delta^{\mathbf n}_{osc}\mathbf C_q\|x+1_{\delta B}\|.$$
By the triangle inequality, we get
$$\|x+1_{\varepsilon A}\|\ \le\ \|x\| + \|1_{\varepsilon A}\|\ \le\ (1+\mathbf C_q + \Delta^{\mathbf n}_{osc}\mathbf C_q)\|x+1_{\delta B}\|.$$
This completes our proof. 
\end{proof}

\begin{proof}[Proof of Proposition \ref{p30}]
Assume that $\mathcal{B}$ is $\mathbf C$-$\mathbf s$-($\mathbf n$, strong partially greedy). We have
\begin{equation}\label{e4}\|x-G_{n}(x)\|\ \le\ \mathbf C\|x\|, \forall x\in X, \forall n\in \mathbf s, \forall G_n(x).\end{equation}
By \cite[Theorem 5.2]{BB}, $\mathcal{B}$ is quasi-greedy. 
Let us show that $\mathcal{B}$ is $\mathbf s$-order-($\mathbf n$, superconservative). Fix $A, B, \varepsilon, \delta$ as in Definition \ref{defoc}. Assume that $A \le n_s < B\cap \mathbf n$ for some $s\in\mathbf s$. Let 
$$x\ : =\ 1_{\varepsilon A} + 1_D + 1_{\delta B},$$
where $D:= \{n_1, \ldots, n_s\}\backslash A$. Since $D$ and $B$ are disjoint,
$$|D\cup B| \ =\ |D| + |B|\ \ge\ |D| + |A| \ =\ s.$$
Let $E\subset D\cup B$ such that $|E| = s$. By above, assume that $\mathcal{B}$ is $\mathbf C_\ell$-suppression quasi-greedy. 
By the $\mathbf s$-($\mathbf n$, strong partially greedy) property, we have 
$$\|1_{\varepsilon A}\|\ =\ \|x-P_{D\cup B}(x)\|\ \le\  \mathbf C_\ell\|x-P_E(x)\|  \ \le\ \mathbf C_\ell\mathbf C\|x-P^{\mathbf n}_s(x)\|\ =\ \mathbf C_\ell\mathbf C\|1_{\delta B}\|.$$
Therefore, $\mathcal{B}$ is $\mathbf s$-order-($\mathbf n$, superconservative).

Next, assume that $\mathcal{B}$ is $\mathbf C_\ell$-suppression quasi-greedy and $\mathbf s$-order-($\mathbf n$, superconservative). By Lemma \ref{ll1}, $\mathcal{B}$ is $\Delta^{\mathbf n}_{pl}$-$\mathbf s$-order-($\mathbf n$, PSLC) for some constant $\Delta^{\mathbf n}_{pl}\ge 1$. Let $x\in X$, $s\in\mathbf s$, $A\in G(x, s)$. Set $E:= \{n_1, \ldots, n_s\}\backslash A$, $F:= A\backslash \{n_1, \ldots, n_s\}$, and $\alpha = \min_{n\in A}|e_n^*(x)|$. We verify that $x-P_A(x)-P_E(x)$, $E$, and $F$ satisfy the conditions in Proposition \ref{x}: note that $(E, F)\in \mathbb{T}(\mathbf n, \mathbf s)$, $E < (\supp(x-P_A(x)-P_E(x))\sqcup F)\cap \mathbf n$. By Proposition \ref{x} and Theorem \ref{bto}, we obtain
\begin{align*}
    \|x-P_A(x)\|&\ \le\ \Delta^{\mathbf n}_{pl}\left\|x - P_A(x) - P_E(x) + \alpha\sum_{n\in F}\sgn(e_n^*(x))e_n\right\|\\
    &\ \le\ \Delta^{\mathbf n}_{pl}\left\|T_\alpha\left(x - P_A(x) - P_E(x) + P_F(x)\right)\right\|\\
    &\ \le\ \Delta_{\lambda, pl}\mathbf C_{\ell}\|x-P^{\mathbf n}_s(x)\|\\
    &\ \le\ \Delta_{\lambda, pl}\mathbf C_{\ell}\mathbf K_\mathbf{n} \widehat{\sigma}^{\mathbf n}_s(x),
\end{align*}
where $\mathbf K_\mathbf{n}$ is the $\mathbf n$-basis constant. 
Therefore, $\mathcal{B}$ is $\mathbf s$-($\mathbf n$, strong partially greedy).
\end{proof}

We now give an example that shows the equivalence in Proposition \ref{p30} fails when $\mathbf s$ has arbitrarily large quotient gaps. The example also shows that 
$\mathbf s$-($\mathbf n$, strong partially greedy) bases are not necessarily ($\mathbf n$, strong partially greedy) when
$\mathbf s$ has arbitrarily large quotient gaps. Our example is a modification of the example in \cite[Proposition 6.9]{BB1}. 

\begin{prop}[Generalization of Proposition \ref{p21} item i)]
Let $\mathbf s$ be a sequence with arbitrarily large quotient gaps. There exists a Banach space $\mathbb X$ with a Schauder and $\mathbf n$-Schauder basis $\mathcal{B}$ such that $\mathcal{B}$ is $\mathbf s$-($\mathbf n$, strong partially greedy) but is neither ($\mathbf n$, conservative) nor quasi-greedy. 
\end{prop}

\begin{proof} 
Define $$\mathcal{S}\ :=\ \left\{S\subset\mathbb{N}: |S|\in \mathbf s \mbox{ and }n_{|S|} < S\cap \mathbf n\right\}.$$
Since $\mathbf s$ has arbitrarily large quotient gaps, we can find a subsequence $(s_{k_j})_j$ such that for all $j$, 
\begin{equation}\label{ee1}
    s_{k_j+1}\ >\ 3(j+1)s_{k_j}.
\end{equation}
Let $X$ be the completion of $c_{00}$ with respect to the following norm: for $x = (x_1, x_2, \ldots)$, 
$$\|x\|\ :=\ \max\left\{\|(x_i)_i\|_{\infty}, \sup_{S\in\mathcal{S}}\sum_{i\in S}|x_i|, \sup_{j\in\mathbb{N}}\sup_{1\le \ell \le js_{k_j}}\left|\sum_{i=1}^\ell x_{n_{s_{k_j}+i}}\right|\right\}.$$
Let $\mathcal{B}$ be the canonical basis of $X$. Clearly, $\mathcal{B}$ is both Schauder and $\mathbf n$-Schauder. Let $B_0 = \emptyset$ and $B_m = \{n_1, \ldots, n_k\}$ for $m\ge 1$.

i) $\mathcal{B}$ is $\mathbf s$-($\mathbf n$, strong partially greedy): Pick $x\in X$, $s\in\mathbf s$, $0\le m\le s$, and $A\in G(x, s)$. We shall show that $$\|x-P_A(x)\|\ \le\ 2\|x-P_m^{\mathbf n}(x)\|.$$

Case 1: $\|x-P_A(x)\| = \|x-P_A(x)\|_{\infty}$. Then $$\|x-P_A(x)\|\ \le\ \|x-P_m^{\mathbf n}(x)\|_{\infty}\ \le\ \|x-P_m^{\mathbf n}(x)\|.$$

Case 2: $\|x-P_A(x)\| = \sup_{S\in\mathcal{S}}\sum_{i\in S}|e_i^*(x-P_A(x))|$. Fix $S\in\mathcal{S}$. First, assume that $|S\cap (B_m\backslash A)| \le |S\cap (A\backslash B_m)|$. We have
\begin{align*}
    &\sum_{i\in S}|e_i^*(x-P_A(x))|\\
    &\ =\ \sum_{i\in S\backslash (A\cup B_m)}|e_i^*(x-P_A(x))| + \sum_{S\cap (B_m\backslash A)}|e_i^*(x)|\\
    &\ \le\ \sum_{i\in S\backslash (A\cup B_m)}|e_i^*(x-P_m^{\mathbf n}(x))| + |S\cap (B_m\backslash A)|\max_{i\in S\cap (B_m\backslash A)}|e_i^*(x)|\\
    &\ \le\  \sum_{i\in S\backslash (A\cup B_m)}|e_i^*(x-P_m^{\mathbf n}(x))| + |S\cap (A\backslash B_m)|\min_{i\in S\cap (A\backslash B_m)}|e_i^*(x-P_m^{\mathbf n}(x))|\\
    &\ \le\ \sum_{i\in S}|e_i^*(x-P_m^{\mathbf n}(x))|\ \le\ \|x-P_m^{\mathbf n}(x)\|.
\end{align*}
Next, assume that $|S\cap (B_m\backslash A)| > |S\cap (A\backslash B_m)|$. We write
\begin{align*}
    |A| &\ =\ |A\backslash (B_m\cup S)| + |S\cap (A\backslash B_m)| + |A\cap B_m\cap S| + |(A\cap B_m)\backslash S|\\
    |B_m| &\ =\ |B_m\backslash (A\cup S)| + |S\cap (B_m\backslash A)| + |A\cap B_m\cap S| + |(A\cap B_m)\backslash S|.
\end{align*}
Since $|A| = s\ge m = |B_m|$, we know that 
$$|A\backslash (B_m\cup S)| + |S\cap (A\backslash B_m)|\ \ge\ |S\cap (B_m\backslash A)|\ >\ |S\cap (A\backslash B_m)|.$$
Therefore, there exists $D\subset A\backslash (B_m\cup S)$ such that 
$$|S\cap (A\backslash B_m)| + |D| \ =\ |S\cap (B_m\backslash A)|.$$
Pick a set $E > A\cup B_m\cup S$, $E\cap \mathbf n > n_{|S|}$, and $|E| = |S\cap (A\backslash B_m)|$. Form 
$$S'\ :=\ (S\backslash (S\cap (B_m\backslash A)))\cup D\cup E\ =\ (S\backslash (B_m\backslash A))\cup D\cup E.$$
Then $|S'| = |S\backslash (B_m\backslash A)| + |D| + |E| = |S|$. We check that $S'\cap \mathbf n > n_{|S|}$. By how we define $S'$, we need only to check that $D\cap \mathbf n > n_{|S|}$. Since $|S\cap (B_m\backslash A)| > 0$, $S$ contains $n_j$ for some $1\le j\le m$. So, $n_{|S|} < n_m < D\cap \mathbf n$ as $D\cap B_m = \emptyset$. We have
\begin{align*}
&\sum_{i\in S}|e_i^*(x-P_A(x))|\\
&\ =\ \sum_{i\in S\backslash (A\cup B_m)}|e_i^*(x-P_A(x))| + \sum_{i\in S\cap (B_m\backslash A)}|e_i^*(x)|\\
&\ \le\ \sum_{i\in S\backslash (A\cup B_m)}|e_i^*(x-P^{\mathbf n}_{m}(x))| + |S\cap (B_m\backslash A)|\max_{i\in S\cap (B_m\backslash A)}|e_i^*(x)|\\
&\ \le\ \sum_{i\in S\backslash (A\cup B_m)}|e_i^*(x-P^{\mathbf n}_{m}(x))| + |D\cup(S\cap (A\backslash B_m))|\min_{i\in D\cup (S\cap (A\backslash B_m))}|e_i^*(x)|\\
&\ \le\ \sum_{i\in S\backslash (A\cup B_m)}|e_i^*(x-P^{\mathbf n}_{m}(x))| + \sum_{i\in D\cup E\cup(S\cap(A\backslash B_m))}|e_i^*(x)|\\
&\ =\ \sum_{i\in S'\backslash (A\cup D\cup E)}|e_i^*(x-P^{\mathbf n}_{m}(x))| + \sum_{i\in D\cup E\cup (S\cap A)}|e_i^*(x-P^{\mathbf n}_{m}(x))|\\
&\ =\ \sum_{i\in S'}|e_i^*(x-P^{\mathbf n}_{m}(x))|\ \le\ \|x-P^{\mathbf n}_{m}(x)\|.
\end{align*}
Since $S\in\mathcal{S}$ is arbitrary, we conclude that $\|x-P_A(x)\|\le \|x-P^{\mathbf n}_m(x)\|$ in this case. 

Case 3: $$\|x-P_A(x)\| \ =\ \sup_{j\in\mathbb{N}}\sup_{1\le \ell \le js_{k_j}}\left|\sum_{i=1}^\ell e^*_{n_{s_{k_j}+i}}(x-P_A(x))\right|.$$
Fix $j, \ell\in\mathbb{N}$ such that $1\le \ell \le js_{k_j}$. First, assume that $s = |A|\le s_{k_j}$. Choose $S\subset \mathbb{N}$ such that $|S| = s_{k_j}$, $n_{s_{k_j}} < S\cap \mathbf n$, and 
$$A\cap \{n_{s_{k_j}+1}, \ldots, n_{s_{k_j} + \ell}\}\ \subset\ S.$$
Then $S\in\mathcal{S}$ and 
$$B_m \ <\ n_m + 1 \ \le \ n_{|A|} + 1\ \le\ n_{s_{k_j}}+1\ \le\ n_{s_{k_j} + 1}.$$
Therefore, we obtain 
\begin{align*}
    \left|\sum_{i=1}^\ell e^*_{n_{s_{k_j}+i}}(x-P_A(x))\right|&\ \le\ \left|\sum_{i=1}^\ell e^*_{n_{s_{k_j}+i}}(x)\right| + \sum_{i=1}^\ell\left|e^*_{n_{s_{k_j}+i}}(P_A(x))\right|\\
    &\ \le\ \left|\sum_{i=1}^\ell e^*_{n_{s_{k_j}+i}}(x-P_m^{\mathbf n}(x))\right| + \sum_{i\in S}\left|e^*_{i}(x-P_m^{\mathbf n}(x))\right|\\
    &\ \le\ 2\|x-P_m^{\mathbf n}(x)\|.
\end{align*}
Next, assume that $|A| > s_{k_j}$. Since $|A|\in\mathbf s$, $|A|\ge s_{k_j+1}$. Let 
$$D\ =\ \{n_{s_{k_j}+1}, n_{s_{k_j}+2}, \ldots, n_{s_{k_j}+\ell}\}\mbox{ and }E \ =\ \{n_1, \ldots, n_{|A|}\}.$$
By \eqref{ee1}, $D\subset E$. If $D\subset A$, then 
$$\left|\sum_{i=1}^\ell e^*_{n_{s_{k_j}+i}}(x-P_A(x))\right|\ =\ 0\ \le\ \|x-P_m^{\mathbf n}(x)\|.$$
Suppose that $D\not\subset A$. Since $|D\backslash A|\le |E\backslash A| = |A\backslash E|$, we can choose $S_1\subset A\backslash E$ such that $|S_1| = |D\backslash A|$. Choose $S_2 > \max\{n_{|A|}, \max A\}$ with $|S_2| = |A| - |S_1|$. Form $S = S_1\cup S_2$. We have $|S| = |A|$, $n_{|A|} < S\cap \mathbf n$, and $|D\backslash A| = |A\cap S|$. Since $S\in\mathcal{S}$, we get
\begin{align*}
    \left|\sum_{i=1}^\ell e^*_{n_{s_{k_j}+i}}(x-P_A(x))\right|\ =\ \left|\sum_{i\in D\backslash A}e_i^*(x)\right|&\ \le\ |D\backslash A|\max_{i\in D\backslash A}|e_i^*(x)|\\
    &\ \le\ |A\cap S|\min_{i\in A\cap S}|e_i^*(x)|\\
    &\ \le\ \sum_{i\in S}|e_i^*(x-P_m^{\mathbf n}(x))|\\
    &\ \le\ \|x-P_m^{\mathbf n}(x)\|.
\end{align*}

From these cases, we know that $\|x-P_A(x)\|\le 2\|x-P_m^{\mathbf n}(x)\|$; therefore, $\mathcal{B}$ is $\mathbf s$-($\mathbf n$, strong partially greedy). 

ii) $\mathcal{B}$ is not ($\mathbf n$, conservative): Let 
$$D_j\ = \ \left\{n_{s_{k_j}+1}, \ldots, n_{(j+1)s_{k_j}}\right\}\mbox{ and }E_j\ =\ \left\{n_{(j+1)s_{k_j}+1}, \ldots, n_{2(j+1)s_{k_j}}\right\}.$$
By the definition of $\|\cdot\|$, $\|1_{D_j}\| = js_{k_j}$. We shall show that $\|1_{E_j}\|\ \le\ s_{k_j}$. If $j'\le j$, then 
$$n_{(j'+1)s_{k_{j'}}}  \ <\ E_j,$$
while if $j' > j$, then by \eqref{ee1}, 
$$n_{s_{k_{j'}}+1} \ >\ E_j.$$
Therefore, $$\sup_{j'\in\mathbb{N}}\sup_{1\le \ell \le j's_{k_{j'}}}\left|\sum_{i=1}^\ell e^*_{n_{s_{k_{j'}}+i}}(1_{E_j})\right| \ =\ 0.$$
Pick $S\in\mathcal{S}$. If $|S|\ge s_{k_j+1}$, then 
$$E_j \ <\ n_{s_{k_j+1}}\ \le\ n_{|S|} \ <\ S\cap \mathbf n.$$
So, $\sum_{i\in S}|e_i^*(1_{E_j})| = 0$. Since $|S|\in\mathbf s$, we can suppose that $|S|\le s_{k_j}$. Then
$$\|1_{E_j}\|\ =\ \max\left\{\|1_{E_j}\|_\infty, \sup_{\substack{S\in \mathcal{S}\\ |S|\le s_{k_j}}}\sum_{i\in S}|e_i^*(1_{E_j})|\right\}\ \le\ s_{k_j}.$$
We have $$\|1_{D_j}\|/\|1_{E_j}\|\rightarrow\infty\mbox{ as }j\rightarrow\infty.$$
Therefore, $\mathcal{B}$ is not ($\mathbf n$, conservative). 

iii) $\mathcal{B}$ is not quasi-greedy. Let $\varepsilon = (1, -1, 1, -1, \ldots)$ and consider 
$\|1_{\varepsilon B_{{s_{k_j+1}}}}\|$. Clearly,
$$\sup_{j'\in\mathbb{N}}\sup_{1\le \ell \le j's_{k_{j'}}}\left|\sum_{i=1}^\ell e^*_{n_{s_{k_{j'}}+i}}(1_{\varepsilon B_{{s_{k_j+1}}}})\right| \ =\ 1.$$
Pick $S\in\mathcal{S}$. If $|S| \ge s_{k_j+1}$, then $S\cap \mathbf n > n_{s_{k_j+1}} \ge B_{s_{k_j+1}}$ and so, 
$$\sum_{i\in S}|e_i^*(1_{\varepsilon B_{{s_{k_j+1}}}})| \ =\ 0.$$
If $|S|\le s_{k_j}$, then 
$$\sum_{i\in S}|e_i^*(1_{\varepsilon B_{{s_{k_j+1}}}})| \ \le\ s_{k_j}.$$
We have shown that 
$$\|1_{\varepsilon B_{{s_{k_j+1}}}}\|\ \le\ s_{k_j}.$$
On the other hand, it is easy to check that 
$$\left\|1_{B_{{s_{k_j+1}}}}\right\| \ \ge\ js_{k_j}.$$
Therefore, $\mathcal{B}$ does not have the UL property. By Remark \ref{r1}, $\mathcal{B}$ is not quasi-greedy. 
\end{proof}

The proof of Proposition \ref{p30} uses the Schauder condition to prove the quasi-greedy property. We can drop the Schauder condition and require that $\mathbf s$ has bounded additive gaps instead. 

\begin{prop}\label{p31}
Let $\mathcal{B}$ be an $\mathbf n$-Schauder Markushevic basis. If $\mathbf s$ has bounded additive gaps, then $\mathcal{B}$ is $\mathbf s$-($\mathbf n$, strong partially greedy) if and only if $\mathcal{B}$ is quasi-greedy and $\mathbf s$-order-($\mathbf n$, superconservative). 
\end{prop}

\begin{proof}
The proof is similar to that of Proposition \ref{p30} but uses \cite[Proposition 4.1]{O} instead of \cite[Theorem 5.2]{BB} to prove the quasi-greedy property. 
\end{proof}

\subsubsection{When an $\mathbf s$-($\mathbf n$, strong partially greedy) basis is ($\mathbf n$, strong partially greedy)}

\begin{prop}[Generalization of Lemma 6.16 in \cite{BB1}]\label{p32}
If $\mathbf s$ has bounded additive gaps, an $\mathbf s$-($\mathbf n$, strong partially greedy) Markushevich basis is ($\mathbf{n}$, strong partially greedy). 
\end{prop}

\begin{proof}
Assume that $\mathcal{B}$ is $\mathbf C$-$\mathbf s$-($\mathbf n$, strong partially greedy). We have
\begin{equation}\label{e444}\|x-G_{n}(x)\|\ \le\ \mathbf C\|x\|, \forall x\in X, \forall n\in \mathbf s, \forall G_n(x).\end{equation}
By \cite[Proposition 4.1]{O}, $\mathcal{B}$ is quasi-greedy. Using the same argument as in the proof of Proposition \ref{p30}, we know that $\mathcal{B}$ is $\Delta^{\mathbf n}_{osc}$-$\mathbf s$-order-($\mathbf n$, superconservative) for some $\Delta^{\mathbf n}_{osc} > 0$. Let us show that $\mathcal{B}$ is ($\mathbf n$, superconservative). Pick $(A, B)\in \mathbb{T}(\mathbf n)$ and signs $\varepsilon, \delta$. Let $s = \min \mathbf s$. We proceed by case analysis.

Case 1: $B\cap \mathbf n = \emptyset$. By Remark \ref{r3}, $(A, B)\in \mathbb{T}(\mathbf n, \mathbf s)$, and we get $$\|1_{\varepsilon A}\|\ \le\ \Delta^{\mathbf n}_{osc}\|1_{\delta B}\|.$$

Case 2: $B\cap \mathbf n\neq \emptyset$ and $\min (B\cap \mathbf n) \le n_s$. Then $A < n_s$. Since $A\subset\mathbf n$, $A = \emptyset$; hence, $\|1_{\varepsilon A}\| = 0 \le \|1_{\delta B}\|$.

Case 3: $B\cap \mathbf n\neq \emptyset$ and $B\cap \mathbf n > n_s$. Let $N:= \max\{k\in \mathbf s: n_k < B\cap \mathbf n\}$ and $A' = \{a\in A: a\le n_N\}$. Observe that $(A', B)\in \mathbb{T}(\mathbf n, \mathbf s)$. Since $\mathcal{B}$ is $\Delta^{\mathbf n}_{osc}$-$\mathbf s$-order-($\mathbf n$, superconservative),
$$\|1_{\varepsilon A'}\|\ \le\ \Delta^{\mathbf n}_{osc}\|1_{\delta B}\|.$$
Consider $A'' = A\backslash A'$. Let $M$ be the next number after $N$ in $\mathbf s$. We have
$$n_N\ <\ A''\ <\ \min(B\cap \mathbf n)\ \le\ n_{M}\mbox{ and }A''\ \subset\ \mathbf n.$$
By hypothesis, $\mathbf s$ has $\ell$-bounded additive gaps for some $\ell\ge 1$. Hence, $|A''|\le \ell-1$, which gives
$$\|1_{\varepsilon A''}\|\ \le\ (\ell-1)\sup_{n}\|e_n\|\ \le\ (\ell-1)\sup_{n}\|e_n\|\sup_{n}\|e_n^*\|\|1_{\delta B}\|\ \le\ (\ell-1) c_2^2\|1_{\delta B}\|.$$
We obtain that 
$$\|1_{\varepsilon A}\|\ \le\ \|1_{\varepsilon A'}\| + \|1_{\varepsilon A''}\|\ \le\ (\Delta^{\mathbf n}_{osc}+(\ell-1) c_2^2)\|1_{\delta B}\|.$$

From these cases, we know that $\mathcal{B}$ is ($\mathbf n$, superconservative). By Theorem \ref{m1}, $\mathcal{B}$ is ($\mathbf n$, strong partially greedy).
\end{proof}

When $\mathbf s$ has arbitrarily large additive gaps, we give an example of a Schauder basis that is $\mathbf s$-($\mathbf n$, strong partially greedy) but is not ($\mathbf{n}$, strong partially greedy). 

\begin{prop}[Generalization of Proposition \ref{p21} item ii)]
Suppose $\mathbf s$ has arbitrarily large additive gaps. There is a Banach space $X$ with a $1$-unconditional basis that is $\mathbf s$-($\mathbf n$, strong partially greedy) but is not ($\mathbf{n}$, conservative). 
\end{prop}

\begin{proof}
We modify the example in \cite[Proposition 6.14]{BB1}. Choose a subsequence $(s_{k_j})_j$ such that 
\begin{equation}\label{ee2}s_{k_{j}+1} - s_{k_{j}}\ >\ 3\cdot 10^{j}\end{equation}
and a decreasing sequence of positive numbers $(p_k)_k$ so that
\begin{equation}\label{ee3}\lim_{k\rightarrow\infty} p_k\ =\ 1\mbox{ and }
\sup_{j} j\left(\frac{1}{p_{k_{j}+1}} - \frac{1}{p_{k_{j}}}\right)\ =\ \infty.\end{equation}
A possible choice is 
$$p_{k_{{2^j}}}\ =\ 1+\frac{1}{j}\mbox{ and } p_{k_{{2^j}}+1}\ =\ 1+\frac{1}{j+1}$$
and then choose other $p_k$'s such that $(p_k)_k$ is decreasing.  
Define 
$$\mathcal{S}_k\ :=\ \{S\subset\mathbf n: |S| = 10^k \mbox{ and }n_{s_{k}} < S\}\mbox{ and }T_j\ :=\ \{n_{s_{k_j}+1}, \ldots, n_{s_{k_j}+10^{j}}\}.$$
Let $X$ be the completion of $c_{00}$ with respect to the following norm:
$$
    \|(x_i)_i\| \ =\ \max\left\{\|(x_i)_i\|_\infty, \|(x_i)_i\|_1, \sup_{j}\left(\sum_{i\in T_j}|x_i|^{p_{k_{j}+1}}\right)^{\frac{1}{p_{k_{j}+1}}}\right\},
$$
where $$\|(x_i)_i\|_1\ :=\ \sup_{k}\left(\sum_{S\in\mathcal{S}_k}\left(\sum_{i\in S}|x_i|^{p_k}\right)^{\frac{1}{p_k}} + \left(\sum_{i\notin \mathbf n}|x_i|^{p_k}\right)^{\frac{1}{p_k}}\right).$$
Let $\mathcal{B}$ be the canonical basis, which is $1$-unconditional and normalized. 

i) $\mathcal{B}$ is $\mathbf s$-($\mathbf n$, strong partially greedy): By the proof of Theorem \ref{p30}, it suffices to show that $\mathcal{B}$ is $\mathbf s$-order-($\mathbf n$, conservative). Let $(A, B)\in \mathbb{T}(\mathbf n, \mathbf s)$ with $A \le n_s< B\cap \mathbf n$. Pick $j\in \mathbb{N}$ such that $T_j\cap A \neq\emptyset$. Since $n_s \ge A$, we get $B\cap \mathbf n\ >\ n_s\ge n_{s_{k_{j}+1}}$. If $|B\cap \mathbf n|\ge |A\cap T_j|$, choose $S\in\mathcal{S}_{k_j+1}$ such that $|S\cap B|\ge |A\cap T_j|$. We obtain
$$\left(\sum_{i\in T_j}|e^*_i(1_A)|^{p_{k_{j}+1}}\right)^{\frac{1}{p_{k_{j}+1}}}\ \le\ \left(\sum_{i\in S}|e^*_i(1_B)|^{p_{k_j+1}}\right)^{\frac{1}{p_{k_j + 1}}}\ \le\ \|1_B\|.$$
If $|B\cap \mathbf n| < |A\cap T_j|$, choose $S\in \mathcal{S}_{k_j+1}$ such that $S\cap B = B\cap \mathbf n$. We have
\begin{align*}
    &\left(\sum_{i\in T_j}|e_i^*(1_A)|^{p_{k_j}+1}\right)^{\frac{1}{p_{k_j+1}}}\\
    &\ \le\ \left(\sum_{i\in B}|e_i^*(1_B)|^{p_{k_j+1}}\right)^{\frac{1}{p_{k_j+1}}}\\
    &\ \le\ \left(\sum_{i\in S\cap B}|e_i^*(1_B)|^{p_{k_j+1}}\right)^{\frac{1}{p_{k_j+1}}} + \left(\sum_{i\in B\backslash S}|e_i^*(1_B)|^{p_{k_j+1}}\right)^{\frac{1}{p_{k_j+1}}}\ \le\ \|1_B\|.
\end{align*}
Next, pick $k\in\mathbb{N}$ and $S\in\mathcal{S}_k$ such that $S\cap A \neq \emptyset$. If $|B\cap \mathbf n|\ge |S\cap A|$, choose $S'\in \mathcal{S}_k$ such that $|S'\cap B| = |S\cap A|$. We obtain
$$\left(\sum_{i\in S}|e_i^*(1_A)|^{p_k}\right)^{\frac{1}{p_k}}\ \le\ \left(\sum_{i\in S'}|e_i^*(1_B)|^{p_k}\right)^{\frac{1}{p_k}}\ \le\ \|1_B\|.$$
If $|B\cap \mathbf n| < |S\cap A|$, we use the same argument as above to obtain the same conclusion. We have shown that $\|1_A\|\le \|1_B\|$; hence, $\mathcal{B}$ is $\mathbf s$-order-($\mathbf n$, conservative). 

ii) $\mathcal{B}$ is not ($\mathbf n$, conservative): We have
$$\|1_{T_j}\|\ \ge\ \left(\sum_{i\in T_j}\left|e_i^*(1_{T_j})\right|^{p_{k_j+1}}\right)^{\frac{1}{p_{k_j+1}}}\ =\ 10^{\frac{j}{p_{k_j+1}}}.$$

Let 
$$D_j\ :=\ \{n_{s_{k_j}+10^j + 1}, n_{s_{k_j} + 10^j + 2},\ldots, n_{s_{k_j}+10^j + 10^j}\}.$$
By \eqref{ee2}, we know that $T_i\cap D_j = \emptyset$ for all $i, j\in\mathbb{N}$. Pick $k\in\mathbb{N}$ and $S\in\mathcal{S}_k$. If $k > k_j+1$, then $S > n_{s_{k_j+1}}$ and so, $S\cap D_j = \emptyset$. Suppose that $k\le k_{j}$. We have
$$\sup_{k\le k_{j}}\sum_{S\in\mathcal{S}_k}\left(\sum_{i\in S}\left|e^*_i\left(1_{D_j}\right)\right|^{p_k}\right)^{\frac{1}{p_k}}\ \le\ 10^{\frac{j}{p_{k_{j}}}}.$$
Observe that $(T_j, D_j)\in \mathbb{T}(\mathbf n)$; however, by \eqref{ee3}, 
$$\frac{\|1_{T_j}\|}{\|1_{D_j}\|}\ =\ 10^{j\left(\frac{1}{p_{k_{j}+1}}-\frac{1}{p_{k_{j}}}\right)}\ \rightarrow\ \infty.$$
We conclude that $\mathcal{B}$ is not ($\mathbf n$, conservative). By Theorem \ref{m1}, $\mathcal{B}$ is not ($\mathbf n$, strong partially greedy). This completes our proof. 
\end{proof}

\section{Larger greedy sums for ($\mathbf n$, strong partially greedy) bases}\label{largersum}
Define $\iota: \mathbf n\rightarrow \mathbb{N}$ as $\iota(n_k) = k$.

\subsection{($\lambda$, $\mathbf n$, strong partially greedy) bases and their characterizations}
Recall the condition in the definition of greedy bases: there exists $\mathbf C\ge 1$ such that 
\begin{equation}\label{ee14}\|x-G_m(x)\|\ \le\ \mathbf C\sigma_m(x), \forall x\in X, \forall m\in\mathbb{N}, \forall G_m(x).\end{equation}
Let $\lambda > 1$. \cite[Theorem 3.3]{DKKT} shows that if we replace $G_m(x)$ in \eqref{ee14} by a larger greedy sum $G_{\lceil\lambda m\rceil}(x)$, then we have a condition that is equivalent to the almost greedy property. Particularly, $\mathcal{B}$ is almost greedy if and only if there exists $\mathbf C \ge 1$ such that 
\begin{equation}\label{ee15}
    \|x-G_{\lceil \lambda m\rceil}(x)\|\ \le\ \mathbf C\sigma_m(x), \forall x\in X, \forall m\in \mathbb{N}, \forall G_m(x).
\end{equation}
Interestingly, enlarging the greedy sum size from $m$ to $\lceil\lambda m\rceil$ moves us from the greedy property to the almost greedy property. Motivated by the idea, the author in \cite{C1} showed that while enlarging the greedy sum size in \eqref{e7} still gives us almost greedy bases, enlarging the greedy sum in \eqref{e13} gives us strictly weaker bases (see \cite[Theorem 1.5]{C1}). In the same manner, we define

\begin{defi}\normalfont
A basis $\mathcal{B}$ in a Banach space is said to be ($\lambda$, $\mathbf n$, strong partially greedy) if there exists $\mathbf C\ge 1$ such that 
\begin{equation}\label{ee30}\|x-G_{\lceil \lambda m\rceil}(x)\|\ \le\ \mathbf C\widehat{\sigma}_m^{\mathbf n}(x), \forall x\in X,\forall m\in\mathbb{N},\forall G_m(x).\end{equation}
The smallest constant $\mathbf {C}$ for \eqref{ee30} to hold is denoted by $\mathbf C_{\lambda, \mathbf n, sp}$.  \end{defi}

We shall characterize ($\lambda$, $\mathbf n$, strong partially greedy) bases by the quasi-greedy property and ($\lambda$, $\mathbf n$, PSLC). 

\begin{defi}\normalfont\label{dnlPSLC}
A basis $\mathcal{B}$ is ($\lambda$, $\mathbf n$, PSLC) if there exists a constant $\mathbf C\ge 1$ such that 
$$\|x+1_{\varepsilon A}\|\ \le\ \mathbf C\|x+1_{\delta B}\|,$$
for all $x\in X$ with $\|x\|_\infty\le 1$, for all $(A, B)\in \mathbb{S}(\mathbf n)$ with $(\lambda - 1) \iota(\max A) + |A|\le |B|$ and $A < (\supp(x)\sqcup B)\cap \mathbf n$, and for all signs $\varepsilon, \delta$. The least constant $\mathbf C$ is denoted by $\Delta_{\lambda, \mathbf n, pl}$.
\end{defi}

\begin{defi}\normalfont
A basis $\mathcal{B}$ is ($\lambda$, $\mathbf n$, superconservative) if there exists a constant $\mathbf C > 0$ such that
$$\|1_{\varepsilon A}\|\ \le\ \mathbf C\|1_{\delta B}\|,$$
for all $(A, B)\in\mathbb{T}(\mathbf n)$ with $(\lambda - 1) \iota(\max A) + |A| \le |B|$ and for all signs $\varepsilon, \delta$. The least constant $\mathbf C$ is denoted by $\mathbf C_{\lambda, \mathbf n, sc}$. When $\varepsilon \equiv \delta \equiv 1$, we say that $\mathcal{B}$ is ($\lambda$, $\mathbf n$, conservative). The corresponding constant is denoted by $\mathbf C_{\lambda, \mathbf n, c}$.
\end{defi}

\begin{thm}\label{mm10}
Let $\mathcal{B}$ be a basis of a Banach space $X$. The following are equivalent
\begin{enumerate}
    \item[i)] $\mathcal B$ is ($\lambda$, $\mathbf n$, strong partially greedy).
    \item[ii)] $\mathcal{B}$ is quasi-greedy and ($\lambda$, $\mathbf n$, PSLC).
    \item[iii)] $\mathcal{B}$ is quasi-greedy and ($\lambda$, $\mathbf n$, superconservative).
    \item[iv)] $\mathcal{B}$ is quasi-greedy and ($\lambda$, $\mathbf n$, conservative).
\end{enumerate}
\end{thm}

Before proving Theorem \ref{mm10}, we need the following lemmas, the first of which reformulate ($\lambda$, $\mathbf n$, PSLC) into a more useful form. The proof is similar to that of Lemma \ref{l2}.
\begin{lem}\label{ll30}
A basis $\mathcal{B}$ is $\Delta_{\lambda, \mathbf n, pl}$-($\lambda$, $\mathbf n$, PSLC) if and only if 
$$\|x\|\ \le\ \Delta_{\lambda, \mathbf n, pl}\|x-P_A(x) + 1_{\varepsilon B}\|,$$
for all $x\in X$ with $\|x\|_\infty\le 1$, for all $(A, B)\in\mathbb{S}(\mathbf n)$ with $(\lambda - 1)\iota(\max A) + |A|\le |B|$ and $A < (\supp(x-P_A(x))\sqcup B)\cap \mathbf n$, and for all signs $\varepsilon$.
\end{lem}

\begin{lem}\label{ll51}
If a basis $\mathcal{B}$ is quasi-greedy and ($\lambda$, $\mathbf n$, superconservative), then $\mathcal{B}$ is ($\lambda$, $\mathbf n$, PSLC).
\end{lem}
\begin{proof}
Similar to the proof of Lemma \ref{l10}.
\end{proof}

\begin{lem}\label{ll50}
Let $\mathcal{B}$ be a $\mathbf C_\ell$-suppression quasi-greedy basis and $\lambda > 1$. The following hold.
\begin{enumerate}
    \item[i)] If $\mathcal{B}$ is $\mathbf C_{\lambda, \mathbf n, sp}$-($\lambda$, $\mathbf n$, strong partially greedy), then $\mathcal{B}$ is $\mathbf C_\ell\mathbf C_{\lambda, \mathbf n, sp}$-($\lambda$, $\mathbf n$, PSLC).
    \item[ii)] If $\mathcal{B}$ is $\Delta_{\lambda, \mathbf n, pl}$-($\lambda$, $\mathbf n$, PSLC), then $\mathcal{B}$ is $\mathbf C_\ell \Delta_{\lambda, \mathbf n, pl}$-($\lambda$, $\mathbf n$, strong partially greedy).
\end{enumerate}
\end{lem}

\begin{proof}
i) Choose $x, A, B, \varepsilon, \delta$ as in Definition \eqref{dnlPSLC}. Let $n_s = \max A$ and set $D = \{n_1, \ldots, n_s\}\backslash A$. Form $y = x + 1_{\varepsilon A} + 1_{D} + 1_{\delta B}$. Observe that 
$$|D\cup B| \ =\ |D| + |B| \ \ge\ (s - |A|) + (\lambda - 1)s + |A|\ =\ \lambda s.$$
Choose $E\subset D\cup B$ such that $|E| = \lceil \lambda s\rceil$ to have
\begin{align*}\|x+1_{\varepsilon A}\| \ =\ \|y - P_{D\cup B}(y)\|\ \le\ \mathbf C_\ell\|y - P_E(y)\|&\ \le\ \mathbf C_\ell\mathbf C_{\lambda, \mathbf n, sp} \widehat{\sigma}_s^{\mathbf n}(y)\\
&\ \le\  \mathbf C_\ell\mathbf C_{\lambda, \mathbf n, sp}\|y-P^{\mathbf n}_{s}(y)\|\\
&\ =\ \mathbf C_\ell\mathbf C_{\lambda, \mathbf n, sp}\|x+1_{\delta B}\|.
\end{align*}

ii) Let $x\in X$, $m\in\mathbb{N}$, $k\le m$, and $A\in G(x, \lceil \lambda m\rceil)$. We need to show that 
$$\|x-P_A(x)\|\ \le\ \mathbf C_\ell \Delta_{\lambda, \mathbf n, pl}\|x-P^{\mathbf n}_{k}(x)\|.$$
Set $E = \{n_1, \ldots, n_k\}\backslash A$, $F = A\backslash \{n_1, \ldots, n_k\}$, and $\alpha = \min_{n\in A}|e_n^*(x)|$. We verify that $E, F, x- P_A(x)$ satisfy the condition of Lemma \ref{ll30}: clearly, $(E, F)\in \mathbb{S}(\mathbf n)$ and 
\begin{align*}
    (\lambda-1)\iota(E) + |E|&\ \le\ (\lambda - 1)m + (m - |A\cap \{n_1, \ldots, n_k\}|)\\
    &\ \le\ \lambda m - |A\cap \{n_1, \ldots, n_k\}|\ \le\ |A| - |A\cap \{n_1, \ldots, n_k\}| \ =\ |F|.
\end{align*}
It is easy to check that  
$$E \ <\ (\supp(x-P_A(x)-P_E(x))\sqcup F)\cap \mathbf n.$$
Let $\varepsilon = (\sgn(e_n^*(x))_n$. Using Lemma \ref{ll30} and Theorem \ref{bto}, we obtain
\begin{align*}
\|x-P_A(x)\|&\ \le\ \Delta_{\lambda, \mathbf n, pl}\|x-P_A(x)-P_{E}(x) + \alpha 1_{\varepsilon F}\|\\
&\ \le\ \Delta_{\lambda, \mathbf n, pl}\|T_{\alpha}(x-P_A(x)-P_E(x) + P_F(x))\|\\
&\ \le\ \Delta_{\lambda, \mathbf n, pl}\mathbf C_\ell\|x-P^{\mathbf n}_{k}(x)\|.
\end{align*}
This completes our proof. 
\end{proof}

\begin{lem}\label{eee22}
If $\mathcal{B}$ is ($\lambda$, $\mathbf n$, strong partially greedy), then $\mathcal{B}$ is quasi-greedy. 
\end{lem}

\begin{proof}
Assume that $\mathcal{B}$ is $\mathbf C_{\lambda, \mathbf n, sp}$-($\lambda$, $\mathbf n$, strong partially greedy). For each $m\in\mathbb{N}$, we have
$$\|x-G_{\lceil\lambda m\rceil}(x)\|\ \le\ \mathbf C_{\lambda, \mathbf n, sp}\widehat{\sigma}^{\mathbf n}_m(x)\ \le\ \mathbf C_{\lambda, \mathbf n, sp}\|x\|.$$
By \cite[Lemma 2.3]{C1}, we know that $\mathcal{B}$ is quasi-greedy. 
\end{proof}

\begin{proof}[Proof of Theorem \ref{mm10}]
Lemmas \ref{ll50}  and \ref{eee22} show that i) $\Longleftrightarrow$ ii). By definitions, ($\lambda$, $\mathbf n$, PSLC) $\Longrightarrow$ ($\lambda$, $\mathbf n$, superconservative), so ii) $\Longrightarrow$ iii). Lemma \ref{ll51} gives iii) $\Longrightarrow$ ii). Finally, iii) $\Longleftrightarrow$ iv) is due to Remark \ref{r1}. 
\end{proof}

By the characterizations in Theorem \ref{mm10}, we know that if $\lambda_1 > \lambda_2 > 1$, then a ($\lambda_2$, $\mathbf n$, strong partially greedy) basis is ($\lambda_1$, $\mathbf n$, strong partially greedy). We now provide a sufficient condition for the converse to hold. 

\begin{defi}\normalfont Let $N\in\mathbb{N}$. 
A basis $\mathcal{B}$ is said to be right-skewed if there exists $\mathbf C > 0$ such that for every finite set $A\subset\mathbb{N}$, we can find a set of positive integers $B$ satisfying the condition: $B> A$, $|B| = |A|$, and 
$\|1_{B}\| \le \mathbf C\|1_A\|$.
\end{defi}

\begin{prop} Fix $1 < \lambda_1 < \lambda_2$.
Let $\mathcal{B}$ be a ($\lambda_2$, $\mathbf n$, conservative) basis that is right-skewed. Then $\mathcal{B}$ is ($\lambda_1$, $\mathbf n$, conservative). 
\end{prop}

\begin{proof}
Set $s:= \left\lceil\frac{\lambda_2-\lambda_1}{\lambda_1 -1 }\right\rceil$.
Assume that $\mathcal{B}$ is $\Delta_{\lambda_2, \mathbf n, c}$-($\lambda_2$, $\mathbf n$, conservative). 
Choose $(A, B)\in\mathbb{T}(\mathbf n)$ with $|A| + (\lambda_1 - 1)\iota(\max A)\le |B|$. 
Choose disjoint subsets $(D_j)_{j=1}^{s}$ satisfying $|D_j| = |B|$, $B < D_1 < D_2 < \cdots < D_s$, and 
$$\|1_{D_j}\|\ \le\ \mathbf C\|1_{D_{j-1}}\|, \forall 2\le j\le s\mbox{ and }\|1_{D_1}\|\ \le\ \mathbf C\|1_B\|.$$
Set $D = \cup_{j=1}^s D_j$.
Then $(B\cup D)\cap \mathbf n > A$ and 
$$|B\cup D| \ =\ |B| + |D| \ \ge\ |A| + (\lambda_1-1)\iota(\max A) + \frac{\lambda_2-\lambda_1}{\lambda_1 -1 }|B|\ \ge\ |A| + (\lambda_2-1)\iota(\max A).$$
Since $\mathcal{B}$ is ($\lambda_2$, $\mathbf n$, conservative), we know that 
\begin{align*}\|1_A\|\ \le\ \Delta_{\lambda_2, \mathbf n, c}\|1_{B\cup D}\|&\ \le\ \Delta_{\lambda_2, \mathbf n, c}(\|1_B\| + \|1_D\|)\\
&\ \le\ \Delta_{\lambda_2, \mathbf n, c}\left(\|1_B\| + \sum_{j=1}^s \|1_{D_j}\|\right)\\
&\ \le\ \Delta_{\lambda_2, \mathbf n, c}(1+\mathbf C + \mathbf C^2 + \cdots + \mathbf C^s)\|1_B\|\\
&\ \le\ \begin{cases}\Delta_{\lambda_2, \mathbf n, c}\frac{\mathbf C^{s+1}-1}{\mathbf C - 1}\|1_B\|&\mbox{ if } \mathbf C \neq 1,\\ \Delta_{\lambda_2, \mathbf n, c}(s + 1)\|1_B\| &\mbox{ if }\mathbf C = 1.\end{cases}
\end{align*}
Hence, $\mathcal{B}$ is ($\lambda_1$, $\mathbf n$, conservative). 
\end{proof}

\begin{cor}\label{cc20}Let $1 < \lambda_1 < \lambda_2$. 
A right-skewed basis $\mathcal{B}$ is ($\lambda_1$, $\mathbf n$, strong partially greedy) if and only if $\mathcal{B}$ is ($\lambda_2$, $\mathbf n$, strong partially greedy).
\end{cor}

\subsection{($\lambda$, $\mathbf n$, strong partially greedy) but not ($\mathbf n$, strong partially greedy) bases}
\begin{thm}
There exists a Banach space $X$ with an $1$-unconditional basis $\mathcal{B}$ that is ($\lambda$, $\mathbf n$, strong partially greedy) for all $\lambda > 1$ but is not ($\mathbf n$, strong partially greedy).
\end{thm}

\begin{proof}
Fix $\lambda > 1$. Choose a subsequence $\mathbf n' = (n_{k_j})$ of $\mathbf n$ such that 
\begin{equation}\label{ee60}
    \ln(\lceil (\lambda - 1)k_j\rceil) \ \ge\ 2\sqrt{j}.
\end{equation}
For each finite set $F\subset\mathbb N$, define the weight sequences $(w^{F}_n)_{n}$ as
$$w^F_n\ =\ \begin{cases}\frac{1}{\sqrt{n}} &\mbox{ if }F\subset\mathbf n',\\ \frac{1}{n} &\mbox{ if }F\not\subset \mathbf n'.\end{cases}$$
Let $X$ be the completion of $c_{00}$ under the following norm: for $x = (x_1, x_2, \ldots)$, 
$$\|x\|\ =\ \sup\left\{\sum_{i\in F}w^{F}_{\pi(i)}|x_i|: \mbox{ finite }F\subset\mathbb{N}\mbox{ and } \pi: F\rightarrow \{1, \ldots, |F|\} \mbox{ is a bijection}\right\}.$$
Let $\mathcal{B}$ be the canonical basis, which is normalized and $1$-unconditional. 

i) $\mathcal{B}$ is not ($\mathbf n$, conservative) and thus, not ($\mathbf n$, strong partially greedy): choose $N\in\mathbb{N}$, $A_N = \{n_{k_1}, n_{k_2}, \ldots, n_{k_N}\}$ and $B_N\subset \{i\in \mathbf n: i\notin \mathbf n'\}$ with $|B_N| = N$ and $B_N > n_{k_N}$. Clearly, $(A_N, B_N)\in \mathbb{T}(\mathbf n)$; however, 
$$\|1_{A_N}\| \ =\ \sum_{i=1}^N\frac{1}{\sqrt{i}}\mbox{ and }\|1_{B_N}\|\ =\ \sum_{i=1}^N\frac{1}{i}.$$
Since $\|1_{A_N}\|/\|1_{B_N}\|\rightarrow\infty$ as $N\rightarrow\infty$, our basis $\mathcal{B}$ is not ($\mathbf n$, conservative).

ii) $\mathcal{B}$ is ($\lambda$, $\mathbf n$, conservative) and thus, ($\lambda$, $\mathbf n$, strong partially greedy): Let $(A, B)\in \mathbb{T}(\mathbf n)$ with $|A| + (\lambda - 1)\iota(\max A)\le |B|$. Pick a finite set $F\subset\mathbb{N}$ and a bijection $\pi: F\rightarrow \{1, \ldots, |F|\}$. If $F\not\subset \mathbf n'$, then 
$$\sum_{i\in F}w^{F}_{\pi(i)}|e_i^*(1_A)|\ \le\ \sum_{i = 1}^{|A|}\frac{1}{i}\ \le\ \sum_{i=1}^{|B|}\frac{1}{i} \ \le\ \|1_B\|.$$
If $F\subset\mathbf n'$, then write $F = \{n_{k_{j_1}}, n_{k_{j_2}}, \ldots, n_{k_{j_s}}\}$. Without loss of generality, assume that 
$F\subset A$ (because for $i\in F\backslash A$, $|e_i^*(1_A)| = 0$).
Then \begin{align*}&|B|\ \ge\ (\lambda -1)\iota(\max A)\ \ge\ (\lambda-1)\iota(n_{k_{j_s}}) \ =\ (\lambda-1)k_{j_s}, \mbox{ and}\\&\sum_{i\in F}w^{F}_{\pi(i)}|e_i^*(1_A)|\ \le\ \sum_{i=1}^{s}\frac{1}{\sqrt{i}}\ \le\ 2\sqrt{s}.\end{align*}
We have
\begin{align*}\|1_B\|\ \ge\ \sum_{i=1}^{|B|}\frac{1}{i}\ \ge\ \sum_{i=1}^{\lceil (\lambda - 1)k_s\rceil}\frac{1}{i}&\ =\ \int_{1}^{\lceil (\lambda - 1)k_s\rceil+1}\frac{dx}{x}\\
&\ \ge\ \ln(\lceil (\lambda - 1)k_s\rceil)\ \stackrel{\eqref{ee60}}{\ge}\ 2\sqrt{s}\ \ge\ \sum_{i\in F}w^{F}_{\pi(i)}|e_i^*(1_A)|.\end{align*}
Letting $F$ and $\pi$ vary, we get $\|1_A\|\ \le\ \|1_B\|$. We have shown that $\mathcal{B}$ is ($\lambda$, $\mathbf n$, strong partially greedy) for some $\lambda > 1$. To show that $\mathcal{B}$ is ($\lambda$, $\mathbf n$, strong partially greedy) for all $\lambda > 1$, we verify that $\mathcal{B}$ is right-skewed and resort to Lemma \ref{cc20}.

iii) $\mathcal{B}$ is right-skewed: pick a finite set $A\subset\mathbb{N}$. Choose $B\subset \{i\in \mathbf n: i\notin\mathbf n'\}$ such that $B> A$ and $|B| = |A|$. Then 
$$\|1_B\|\ = \ \sum_{i=1}^{|B|}\frac{1}{i} \ =\ \sum_{i=1}^{|A|}\frac{1}{i} \ \le\ \|1_A\|.$$
This completes our proof.
\end{proof}

\section{Weighted ($\mathbf n$, strong partially greedy) bases}\label{weight}
\subsection{The theory of weighted greedy-type bases}
As a variant of greedy-type bases, researchers have studied the sequentially weighted version, where
an arbitrary weight sequence $\zeta = (s_n)_{n=1}^\infty\in (0,\infty)^{\mathbb{N}}$ is used to ``weigh" each subset of $\mathbb{N}$. For $A\subset\mathbb{N}$, the $\zeta$-measure of $A$ is $\zeta(A) = \sum_{i\in A}s_i$. Let $\eta > 0$ and define
$$\sigma^{\zeta}_{\eta}(x)\ :=\ \inf\left\{\left\|x-\sum_{n\in A}a_ne_n\right\|: |A| < \infty, \zeta(A)\le \eta, a_n\in\mathbb{F}\right\}$$
and 
$$\widetilde{\sigma}^{\zeta}_{\eta}(x)\ :=\ \inf\left\{\left\|x-P_A(x)\right\|: |A| < \infty, \zeta(A)\le \eta\right\}.$$
\begin{defi}\normalfont
A basis is $\zeta$-greedy if there exists a constant $\mathbf C \ge 1$ such that
$$\|x-P_A(x)\|\ \le\ \mathbf C\sigma^{\zeta}_{\zeta(A)}(x), \forall x\in X, \forall A\in G(x, m)\mbox{ for some }m\in\mathbb{N}.$$

A basis is $\zeta$-almost greedy if there exists a constant $\mathbf C \ge 1$ such that
$$\|x-P_A(x)\|\ \le\ \mathbf C\widetilde{\sigma}^{\zeta}_{\zeta(A)}(x), \forall x\in X, \forall A\in G(x, m)\mbox{ for some }m\in\mathbb{N}.$$
\end{defi}

\begin{defi}\normalfont
A basis is $\zeta$-democratic if there exists a constant $\mathbf C\ge 1$ such that 
$$\|1_A\|\ \le\ C\|1_B\|,$$
for all finite sets $A, B$ with $\zeta(A)\le \zeta(B)$.
\end{defi}

Among other results, Dilworth et al. \cite{DKTW} characterized $\zeta$-almost greedy bases as being quasi-greedy and $\zeta$-democratic (\cite[Theorem 2.6]{DKTW}). Later, Bern\'{a} et al. \cite{BDKOW} introduced the so-called $\zeta$-Property (A) and showed that a basis $\mathcal{B}$ is $\zeta$-greedy ($\zeta$-almost greedy, resp.) if and only if it is unconditional (quasi-greedy, resp.) and has the $\zeta$-Property (A). Other results related to sequential weights include characterizations of weighted weak semi-greedy bases and weighted semi-greedy bases \cite{BL2, B3, DKTW} and characterizations of weighted partially greedy and weighted reverse partially greedy bases \cite{K}. Recently, the author of the present paper \cite{C2} generalized sequential weights to an arbitrary weight $\omega$ on sets and gave an example of a basis that is set-weighted greedy but is not sequence-weighted greedy for any sequence. The goal of this section is to study set-weighted ($\mathbf n$, strong partially greedy) bases. 

\begin{defi}\label{d81}\normalfont
Let $\mathcal{P}(\mathbb N)$ be the power set of $\mathbb{N}$. A weight on set is a nonnegative function $\omega: \mathcal{P}(\mathbb{N})\rightarrow [0,\infty]$ such that 
\begin{enumerate}
    \item[i)] $\omega(\emptyset) = 0$, 
    \item[ii)] $\omega(A) \in (0,\infty]$ for each nonempty set $A\subset\mathbb{N}$.
\end{enumerate}
\end{defi}

\subsection{Weighted ($\mathbf n$, strong partially greedy) bases and characterizations}

For each $m\ge 1$, we define $B_m := \{n_1, \ldots, n_m\}$ and $B_0 = \emptyset$. Also let 
$\mathbb T^{\omega}(\mathbf n)$ be the collection of all ordered pairs of finite sets $(A, B)$ with $A\subset\mathbf n$, $\omega(A)\le \omega(B)$, and $A < B\cap \mathbf n$.

\begin{defi}\normalfont
A basis is said to be $\omega$-($\mathbf n$, strong partially greedy) if there exists $\mathbf C\ge 1$ such that for all $x\in X$, $m\in\mathbb{N}$, and $A\in G(x, m)$, we have
\begin{equation}\label{e500}\|x-P_A(x)\|\ \le\ \mathbf C\widehat{\sigma}^{\mathbf n,\omega}_{A}(x),\end{equation}
where 
$$\widehat{\sigma}^{\mathbf n, \omega}_A(x)\ :=\ \inf\{\|x-P_{B_m}(x)\|:\omega(B_m\backslash A)\le \omega(A\backslash B_m)\}.$$
The smallest $\mathbf C$ satisfying \eqref{e500} is denoted by $\mathbf C^{\omega}_{\mathbf n, sp}$.
\end{defi}

\begin{defi}\normalfont\label{d21}
A basis is $\omega$-($\mathbf n$, PSLC) if there exists a constant $\mathbf C\ge 1$ such that
\begin{equation}\label{e51}\|x+ 1_{\varepsilon A}\|\ \le\ \mathbf C\|x + 1_{\delta B}\|,\end{equation}
for all $x\in X$ with $\|x\|_\infty \le 1$, for all $(A, B)\in \mathbb{T}^{\omega}(\mathbf n)$ with $A < (B\sqcup \supp(x))\cap \mathbf n$, and for all signs $\varepsilon, \delta$. The smallest $\mathbf C$ for which \eqref{e51} holds is denoted by $\Delta^{\omega}_{\mathbf n, pl}$. 
\end{defi}

We have an useful reformulation of $\omega$-($\mathbf n$, PSLC), whose proof is similar to the proof of Lemma \ref{l2}.

\begin{lem}\label{l30}
A basis $\mathcal{B}$ is $\Delta^{\omega}_{\mathbf n, pl}$-$\omega$-($\mathbf n$, PSLC) if and only if 
$$\|x\|\ \le\ \Delta^{\omega}_{\mathbf n, pl}\|x-P_A(x) + 1_{\varepsilon B}\|,$$
for all $x\in X$ with $\|x\|_\infty\le 1$, for all signs $\varepsilon$, and for all $(A, B)\in \mathbb{T}^{\omega}(\mathbf n)$ with $A < (\supp(x-P_A(x))\sqcup B)\cap \mathbf n$.
\end{lem}

\begin{thm}\label{m10}
Let $\mathcal{B}$ be a basis and $\omega$ be a weight on subsets of $\mathbb{N}$.
\begin{enumerate}
    \item [i)] If $\mathcal{B}$ is $\mathbf C^{\omega}_{\mathbf n, sp}$-$\omega$-($\mathbf n$, strong partially greedy), then $\mathcal{B}$ is $\mathbf C^{\omega}_{\mathbf n, sp}$-suppression quasi-greedy and $\mathbf C^{\omega}_{\mathbf n, sp}$-$\omega$-($\mathbf n$, PSLC).
    \item [ii)] If $\mathcal{B}$ is $\mathbf C_\ell$-suppression quasi-greedy and $\Delta^{\omega}_{\mathbf n,pl}$-$\omega$-($\mathbf n$, PSLC), then $\mathcal{B}$ is $\mathbf C_\ell \Delta^{\omega}_{\mathbf n,pl}$-$\omega$-($\mathbf n$, strong partially greedy). 
\end{enumerate}
\end{thm}

\begin{proof}
(i) Let $\mathcal{B}$ be $\mathbf C^{\omega}_{\mathbf n, sp}$-$\omega$-($\mathbf n$, strong partially greedy). Let $x\in X$, $m\in\mathbb{N}$, and $\Lambda \in G(x, m)$.  We have
$$\|x-P_{\Lambda}(x)\|\ \le\ \mathbf C^\omega_{\mathbf n,sp}\widehat{\sigma}^{\mathbf n,\omega}_{\Lambda}(x)\ \le\ \mathbf C^{\omega}_{\mathbf n, sp}\|x\|,$$
which shows that $\mathcal{B}$ is $\mathbf C^\omega_{\mathbf n, sp}$-suppression quasi-greedy. To show that $\mathcal{B}$ is $\mathbf C_{\mathbf n, sp}^\omega$-$\omega$-($\mathbf n$, PSLC), we choose $x, A, B, \varepsilon, \delta$ as in Definition \ref{d21}. Let $\max A = n_s$ and $y := 1_{\varepsilon A} + 1_D + x + 1_{\delta B}$, where $D = \{n_1, n_2, \ldots, n_s\}\backslash A$.
Since $B\cup D$ is a greedy set of $y$ and $w(A)\le w(B)$, we have
\begin{align*}\|x+1_{\varepsilon A}\|\ =\ \|y-P_{B\cup D}(y)\|&\ \le\ \mathbf C^\omega_{\mathbf n, sp}\widehat{\sigma}^{\mathbf n, \omega}_{B\cup D}(y)\\
&\ \le\ \mathbf C^\omega_{\mathbf n, sp}\|y - P_{A\cup D}(y)\|\ =\ \mathbf C^\omega_{\mathbf n, sp}\|x+1_{\delta B}\|.\end{align*}

(ii) Assume that $\mathcal{B}$ is $\mathbf C_\ell$-suppression quasi-greedy and is $\Delta^{\omega}_{\mathbf n,pl}$-$\omega$-($\mathbf n$, PSLC). Let $x\in X$, $m\in\mathbb N$, and $\Lambda\in G(x, m)$. Fix $B_k$ with $\omega(B_k\backslash \Lambda)\le \omega(\Lambda\backslash B_k)$. We need to show that 
$$\|x-P_{\Lambda}(x)\|\ \le\ \mathbf C_\ell \Delta^{\omega}_{\mathbf n,pl}\|x-P_{B_k}(x)\|.$$
Set $E = B_k\backslash \Lambda$, $F = \Lambda\backslash B_k$, and $\alpha = \min_{n\in \Lambda}|e_n^*(x)|$. 
By Lemma \ref{l30} and Theorem \ref{bto}, we obtain
\begin{align*}
    \|x-P_{\Lambda}(x)\|&\ \le\ \Delta^{\omega}_{\mathbf n,pl}\left\|x-P_{\Lambda}(x)-P_E(x) + \alpha\sum_{n\in F}\sgn(e_n^*(x))e_n\right\|\\
    &\ =\ \Delta^{\omega}_{\mathbf n,pl}\left\|T_\alpha\left(x-P_{\Lambda}(x)-P_E(x) + P_F(x)\right)\right\|\\
    &\ \le\ \mathbf C_\ell \Delta^{\omega}_{\mathbf n,pl}\|x-P_{B_k}(x)\|,
\end{align*}
as desired.
\end{proof}

\begin{defi}\normalfont
A basis $\mathcal{B}$ is $\omega$-($\mathbf n$, superconservative) if there exists a constant $\mathbf C > 0$ such that 
\begin{equation}\label{e80}\|1_{\varepsilon A}\|\ \le\ \mathbf C\|1_{\delta B}\|,\end{equation}
for all $(A, B)\in \mathbb T^{\omega}(\mathbf n)$ and for all signs $\varepsilon, \delta$. The smallest $\mathbf C$ for \eqref{e80} to hold is denoted by $\Delta^{\omega}_{\mathbf n, sc}$. If \eqref{e80} holds for $\varepsilon \equiv \delta \equiv 1$, then $\mathcal{B}$ is said to be $\omega$-($\mathbf n$, conservative), and the smallest constant in this case is denoted by $\Delta^{\omega}_{\mathbf n, c}$. 
\end{defi}

Similar to Lemma \ref{l10}, we have

\begin{lem}\label{l80}
If $\mathcal{B}$ is $\mathbf C_q$-quasi-greedy and $\Delta^{\omega}_{\mathbf n, sc}$-$\omega$-($\mathbf n$, superconservative), then $\mathcal{B}$ is $\Delta^{\omega}_{\mathbf n, pl}$-$\omega$-($\mathbf n$, PSLC) with $\Delta^{\omega}_{\mathbf n, pl}\le 1+\mathbf C_q + \Delta^{\omega}_{\mathbf n, sc}\mathbf C_q$. 
\end{lem}

\begin{thm}\label{m20}
Let $\mathcal{B}$ be a basis and $\omega$ be a weight on subsets of $\mathbb{N}$. The following are equivalent:
\begin{enumerate}
    \item [i)] $\mathcal{B}$ is $\omega$-($\mathbf n$, strong partially greedy).
    \item [ii)] $\mathcal{B}$ is quasi-greedy and $\omega$-($\mathbf n$, PSLC).
    \item [iii)] $\mathcal{B}$ is quasi-greedy and $\omega$-($\mathbf n$, superconservative).
    \item [iv)] $\mathcal{B}$ is quasi-greedy and $\omega$-($\mathbf n$, conservative). 
\end{enumerate}
\end{thm}

\begin{proof}
That i) $\Longleftrightarrow$ ii) is due to Theorem \ref{m10}. By definition and Lemma \ref{l80}, ii) $\Longleftrightarrow$ iii). Finally, iii) $\Longleftrightarrow$ iv) is due to Remark \ref{r1}.
\end{proof}

\begin{cor}\label{cc1}
A basis $\mathcal{B}$ is quasi-greedy if and only if $\mathcal{B}$ is $\omega$-($\mathbf n$, strong partially greedy) for some $\omega$.
\end{cor}

\begin{proof}
If $\mathcal{B}$ is $\omega$-($\mathbf n$, strong partially greedy) for some $\omega$, then $\mathcal{B}$ is quasi-greedy by Theorem \ref{m20}. Conversely, suppose that $\mathcal{B}$ is quasi-greedy. Define the weight $\omega$ on a set $A\subset \mathbb{N}$ 
$$\omega(A)\ =\ \begin{cases}\|1_A\| &\mbox{ if }A\mbox{ is finite,}\\ \infty &\mbox{ if }A\mbox{ is infinite}.\end{cases}$$
By Theorem \ref{m20}, we need to verify that $\mathcal{B}$ is $\omega$-($\mathbf n$, conservative). This is clearly true as for any two finite sets $A$ and $B$, we have $\|1_A\|\le \|1_B\|$ whenever $\omega(A)\le \omega(B)$ by the definition of the weight $\omega$. 
\end{proof}

\subsection{Properties of $\zeta$-($\mathbf n$, conservative) bases}
We consider sequence-weighted ($\mathbf n$, conservative) bases, which are a special case of the set-weighted version. In particular, let $\zeta = (s_n)_n\in (0,\infty)^{\mathbb{N}}$, then for each subset $A\subset\mathbb{N}$, 
$\zeta(A) := \sum_{n\in A} s_n$; that is, the weight of each set is determined by the sequence $\zeta$. Hence, we do not have the freedom of assigning weights to sets as in Definition \ref{d81}.

\begin{prop}\label{p12}
Let $\mathcal{B}$ be a basis of a Banach space $X$ and $\zeta = (s_n)_n\in (0, \infty)^{\mathbb N}$.
\begin{enumerate}
    \item[i)] Let $\mathcal B$ be $\mathbf C$-$\zeta$-($\mathbf n$, superconservative). If a finite set $A\subset\mathbf n$ has $$\zeta(A)\ \le\ \limsup_{n\rightarrow\infty}s_n,$$ then $\sup_{\varepsilon}\|1_{\varepsilon A}\|\le 2\mathbf Cc_2$.
    \item[ii)] If $\sup_n s_n = \infty$ and $\mathcal B$ is $\zeta$-($\mathbf n$, superconservative), then $(e_n)_{n\in\mathbf n}$ is equivalent to the canonical basis of $c_0$.
    \item[iii)] Let $\mathbf n\neq \mathbb{N}$. If $\inf_{n\in\mathbf n} s_n = 0$ and $\mathcal B$ is $\mathbf C$-$\zeta$-($\mathbf n$, superconservative), then $\mathcal B$ has a subsequence equivalent to the canonical basis of $c_0$.
\end{enumerate}
\end{prop}

\begin{proof}
i) Let $A\subset\mathbf n$ be a finite set with $\zeta(A)\le \limsup_{n\rightarrow\infty} s_n$. Choose $n_2 > n_1 > A$ such that $\zeta(A) < s_{n_1} + s_{n_2} = \zeta(\{n_1, n_2\})$. Since $(A, \{n_1, n_2\})\in \mathbb{T}^{\zeta}(\mathbf n)$, we obtain
$$\|1_{\varepsilon A}\|\ \le\ \mathbf C\|e_{n_1} + e_{n_2}\|\ \le\ 2\mathbf Cc_2.$$

ii) Use item i). 

iii) Pick $N\in \mathbb{N}\backslash \mathbf{n}$. Let $(n'_i)_i\subset\mathbf n$ be a sequence such that $\sum_{i} s_{n'_i} < s_N$. For any finite set $A\subset (n'_i)_i$, we have $(A, \{N\})\in \mathbb T^{\zeta}(\mathbf n)$. Hence, $\sup_{\varepsilon}\|1_{\varepsilon A}\|\le \mathbf C\|e_N\|\ \le \ \mathbf Cc_2$. Therefore, $(e_{n'_i})_i$ is equivalent to the canonical basis of $c_0$.
\end{proof}

\begin{prop}\label{p100}
Suppose that $0 < \inf_{n} s_n\le \sup_n s_n < \infty$. Then $\mathcal{B}$ is $\zeta$-($\mathbf n$, superconservative) if and only if $\mathcal{B}$ is ($\mathbf n$, superconservative). 
\end{prop}

\begin{proof}
Without loss of generality, let $0 < \alpha:= \inf_{n} s_n \le \sup_n s_n = 1$.
Assume that $\mathcal{B}$ is $\Delta_{\mathbf n, sc}$-($\mathbf n$, superconservative). Let $(A, B)\in \mathbb T^{\zeta}(\mathbf n)$ and choose signs $\varepsilon, \delta$. We have $\alpha |A|\le \zeta(A)\le \zeta(B)\le |B|$. If $|A| < 2\lceil 1/\alpha\rceil$, then 
$$\|1_{\varepsilon A}\|\ <\ 2\lceil 1/\alpha\rceil\sup_{n}\|e_n\|\ \le\ 2\lceil 1/\alpha\rceil\sup_{n}\|e_n\|\sup_{m}\|e_m^*\|\|1_{\delta B}\|\ \le\ 2\lceil 1/\alpha\rceil c_2^2 \|1_{\delta B}\|.$$
We consider the case $|A| \ge 2\lceil 1/\alpha\rceil$. Then $|B|\ge \alpha |A|\ge 2$. Partition $A$ into $N = \lceil 2/\alpha\rceil$ sets $A_1, \ldots, A_N$ such that each set $A_i$ has 
$$|A_i|\ \le\ \frac{|A|}{N} + 1\ \le\ \frac{|B|}{\alpha N} + 1\ \le\ \frac{|B|}{2} + 1\ \le\ |B|.$$
Hence, $(A_i, B)\in \mathbb T(\mathbf n)$ and
$\|1_{\varepsilon A_i}\| \le \Delta_{\mathbf n, sc}\|1_{\delta B}\|$. Therefore,
$$\|1_{\varepsilon A}\|\ \le\ \sum_{i=1}^{N}\|1_{\varepsilon A_i}\|\ \le\ \left\lceil \frac{2}{\alpha}\right\rceil\Delta_{\mathbf n, sc}\|1_{\delta B}\|.$$
We have shown that $\mathcal{B}$ is $\zeta$-($\mathbf n$, superconservative).

Assume that $\mathcal{B}$ is $\Delta^{\zeta}_{\mathbf n, sc}$-$\zeta$-($\mathbf n$, superconservative). Let $(A, B)\in \mathbb T(\mathbf n)$ and choose signs $\varepsilon, \delta$. If $\zeta(A)\le \zeta(B)$, then $(A, B)\in \mathbb T^{\zeta}(\mathbf n)$ and so, $\|1_{\varepsilon A}\|\le \Delta^{\zeta}_{\mathbf n, sc}\|1_{\delta B}\|$. Suppose that $\zeta(A) > \zeta(B)$. 

Case 1: $\zeta(B) > 2$. 
Partition $A$ into $N$
sets $A_1, A_2, \ldots, A_N$ such that $A_1 < A_2 < \cdots < A_N$,  
$\zeta(A_i)  \le \zeta(B)  < \zeta(A_i) + 1\mbox{ for }1\le i\le N-1,\mbox{ and }\zeta(A_N) \le \zeta(B)$. Observe that for $1\le i\le N-1$, 
$$\zeta(A_i)\ >\ \zeta(B) - 1\ >\ \frac{\zeta(B)}{2}.$$
Therefore,  
$$N-1\ \le\ \frac{\zeta(A)}{\zeta(B)/2}\ \le\ \frac{2|A|}{\alpha |B|}\ \le\ \frac{2}{\alpha}.$$
Since $\mathcal{B}$ is $\Delta^{\zeta}_{\mathbf n, sc}$-$\zeta$-($\mathbf n$, superconservative) and $(A_i, B)\in \mathbb{T}^{\zeta}(\mathbf n)$, we obtain
$$
    \|1_{\varepsilon A}\|\ \le\ \sum_{i=1}^N\|1_{\varepsilon A_i}\|\ \le\ \sum_{i=1}^N\Delta^{\zeta}_{\mathbf n, sc}\|1_{\delta B}\|\ \le\ \left(1+\frac{2}{\alpha}\right)\Delta^{\zeta}_{\mathbf n, sc}\|1_{\delta B}\|.
$$

Case 2: $\zeta(B) \le 2$. We have
$$\alpha |A|\ \le\ \alpha|B|\ \le\ \zeta(B)\ \le \ 2.$$
Hence, $|A|\le 2/\alpha$, which gives
$$\|1_{\varepsilon A}\|\ \le\ \frac{2}{\alpha}\sup_{n}\|e_n\|\ \le\ \frac{2}{\alpha}\sup_{n}\|e_n\|\sup_m\|e_m^*\|\|1_{\delta B}\|\ \le\ \frac{2c_2^2}{\alpha}\|1_{\delta B}\|.$$
This shows that $\mathcal{B}$ is  ($\mathbf n$, superconservative).
\end{proof}

\begin{cor}\label{c10}Let $\mathbf n\neq \mathbb{N}$ and $\zeta = (s_n)_n$ be such that either $\inf_{n} s_n  > 0$ or $\inf_{n\in\mathbf n}s_n = 0$.
If a basis $\mathcal{B} = (e_n)_n$ is $\zeta$-($\mathbf n$, strong partially greedy) for the weight sequence $\zeta = (s_n)_n$, then either $\mathcal{B}$ is ($\mathbf n$, strong partially greedy) or $\mathcal{B}$ has a subsequence that is equivalent to the canonical basis of $c_0$.
\end{cor}

\begin{proof}
Let $\mathcal{B}$ be $\zeta$-($\mathbf n$, strong partially greedy) for some weight sequence $\zeta = (s_n)_n$.
By Theorem \ref{m20}, we know that $\mathcal{B}$ is quasi-greedy and $\zeta$-($\mathbf n$, superconservative). 

If $0 < \inf_{n} s_n \le \sup_n s_n <\infty$, then Proposition \ref{p100} states that $\mathcal{B}$ is ($\mathbf n$, superconservative). By Theorem \ref{m1}, $\mathcal{B}$ is ($\mathbf n$, strong partially greedy). 

If $\sup_n s_n = \infty$, then Proposition \ref{p12} item ii) gives that $(e_n)_{n\in\mathbf n}$ is equivalent to the canonical basis of $c_0$.

If $\inf_{n\in\mathbf n} s_n = 0$, Proposition \ref{p12} item iii) gives that $(e_n)_{n\in\mathbf n}$ has a subsequence that is equivalent to the canonical basis of $c_0$.
\end{proof}

\begin{cor}\label{c11}Let $\mathbf n\neq \mathbb{N}$ and $\Delta_{\mathbb N, \mathbf n}$ be finite.
If a basis $\mathcal{B} = (e_n)_n$ is $\zeta$-($\mathbf n$, strong partially greedy) for some weight sequence $\zeta = (s_n)_n$, then either $\mathcal{B}$ is ($\mathbf n$, strong partially greedy) or $\mathcal{B}$ has a subsequence that is equivalent to the canonical basis of $c_0$.
\end{cor}

\begin{proof}
If $\inf s_n > 0$, then the conclusion follows from Corollary \ref{c10}. If $\inf_n s_n = 0$, then the finite $\Delta_{\mathbb N, \mathbf n}$ implies that $\inf_{n\in\mathbf n} s_n = 0$. Again the conclusion follows from Corollay \ref{c10}.
\end{proof}

We shall use Corollary \ref{c10} to prove the next theorem.
\begin{thm}\label{m14}
Let $\mathbf n$ be such that $\mathbb{N}\backslash \mathbf n$ is infinite. There exists a $1$-unconditional Schauder basis $\mathcal B$ that is $\omega$-($\mathbf n$, strong partially greedy) for some weight on sets $\omega$, but $\mathcal B$ is not $\zeta$-($\mathbf n$, strong partially greedy) for any weight sequence $\zeta$ that satisfies either $\inf_{n} s_n  > 0$ or $\inf_{n\in\mathbf n}s_n = 0$.
\end{thm}

\begin{proof}
We slightly modify the example in the proof of Theorem \ref{m8}. Let $K = \mathbb N\backslash \mathbf n = k_1 < k_2 < \cdots$.
Define a function $\phi: \mathbb{N}\rightarrow \mathbb{N}$ as 
$$\phi(n) \ =\ \begin{cases}1&\mbox{ if } n \le k_1,\\ j &\mbox{ if } k_{j}+1\le  n\le k_{j+1} \mbox{ for some }j\ge 1.\end{cases}$$
Let $X$ be the completion of $c_{00}$ under the following norm: for $x = (x_1, x_2, \ldots)$,
$$\|x\|\ :=\ \sup_{F}\sum_{m\in F}|x_m| + \|x\|_{\ell_2},$$
where $F\subset \mathbb N$ and $\sqrt{\phi(\min F)}\ge |F|$. Let $\mathcal{B}$ be the canonical basis. Using the same argument as in Case 2 of Theorem \ref{m8}, we know that $\mathcal{B}$ is not ($\mathbf n$, conservative) and thus, not ($\mathbf n$, strong partially greedy).

Define a weight $\omega$ on each set $A\subset\mathbb{N}$ as follows:
$$\omega(A)\ =\ \begin{cases}\|1_A\| &\mbox{ if }A\mbox{ is finite,}\\ \infty &\mbox{ if }A\mbox{ is infinite}.\end{cases}$$
Since $\|1_A\| = \omega(A)$ for each finite $A\subset\mathbb{N}$, $\mathcal{B}$ is $\omega$-($\mathbf n$, conservative). Hence, we know that $\mathcal{B}$ is $\omega$-($\mathbf n$, strong partially greedy) by Theorem \ref{m20}.

Let us show that $\mathcal{B}$ is not $\zeta$-($\mathbf n$, strong partially greedy) for any weight sequence $\zeta$ that satisfies either $\inf_{n} s_n  > 0$ or $\inf_{n\in\mathbf n}s_n = 0$. By Corollary \ref{c10}, we need only to show that $\mathcal{B}$ has no subsequence equivalent to the canonical basis of $c_0$. Let $(f_n)_n$ be a subsequence of $\mathcal{B}$ and $N\in\mathbb{N}$, we have $\|\sum_{n=1}^N f_n\|\ge \sqrt{N}\rightarrow\infty$
as $N\rightarrow\infty$. Hence, $(f_n)_n$ is not equivalent to the canonical basis of $c_0$, a contradiction. 
\end{proof}

\begin{rek}\normalfont\label{r22}
The weight $\omega$ defined in the proof of Theorem \ref{m14} has some nice properties
\begin{enumerate}
    \item[i)] $\omega(A) < \infty$ if $A$ is finite,
    \item[ii)] $\omega(A)\rightarrow 0$ as $\sum_{n\in A}\omega(\{n\})\rightarrow 0$, and $\omega(A)\rightarrow \infty$ as $\sum_{n\in A}\omega(\{n\})\rightarrow\infty$,
    \item[iii)] If $A, B\subset\mathbb{N}$ are finite sets such that $\emptyset \neq A\subsetneq B$, then $\omega(B) - \omega(A)\ge \eta > 0$, where $\eta$ depends only on $|A|$ and $|B|$. In particular, 
    $$\omega(A) + (\sqrt{|B|}-\sqrt{|A|}) \ \le\ \omega(B).$$
\end{enumerate}
\end{rek}

\begin{cor}\label{cc2}
\begin{enumerate}
    \item[i)]  There exists a $1$-unconditional basis $\mathcal{B}$ that is $\omega$-($\mathbb N$, strong partially greedy) for some weight on sets $\omega$, but $\mathcal{B}$ is not $\zeta$-($\mathbb N$, strong partially greedy) for any weight sequence $\zeta$ with $\inf s_n > 0$.
    \item[ii)] Let $\mathbf n\neq \mathbb{N}$ and $\Delta_{\mathbb N, \mathbf n}$ be finite. There exists a $1$-unconditional basis $\mathcal{B}$ that is $\omega$-($\mathbf n$, strong partially greedy) for some weight on sets $\omega$, but $\mathcal{B}$ is not $\zeta$-($\mathbf n$, strong partially greedy) for any weight sequence $\zeta$.
    \item[iii)] Let $\mathbf n$ be such that $\Delta_{\mathbb N, \mathbf n}$ is infinite. There exists a $1$-unconditional Schauder basis $\mathcal B$ that is $\omega$-($\mathbf n$, strong partially greedy) for some weight on sets $\omega$, but $\mathcal B$ is not $\zeta$-($\mathbf n$, strong partially greedy) for any weight sequence $\zeta$ that satisfies either $\inf_{n} s_n  > 0$ or $\inf_{n\in\mathbf n}s_n = 0$.
\end{enumerate}
\end{cor}

\begin{rek}\normalfont
Corollary \ref{cc2} item i) is sharp in the sense that we cannot drop the requirement $\inf s_n > 0$. Indeed, Khurana \cite{K} characterized $\zeta$-($\mathbb N$, strong partially greedy) bases by quasi-greediness and the so-called $\varsigma$-left-Property (A). By \cite[Remark 3.3]{K}, any basis trivially satisfies $\zeta$-left-Property (A) with $\zeta = (s_n)_{n=1}^\infty = (2^{-n})_{n=1}^\infty$. Hence, if we have an $\omega$-($\mathbb N$, strong partially greedy) basis, it is quasi-greedy by Theorem \ref{m20} and has $\zeta$-left-Property (A) for $s_n = 2^{-n}$. Therefore, the basis is automatically $\zeta$-($\mathbb N$, strong partially greedy). 

As we shall show later, we can remove the condition ``either $\inf_{n} s_n  > 0$ or $\inf_{n\in\mathbf n}s_n = 0$" in item iii) and bring items ii) and iii) into one theorem. 
\end{rek}

\begin{proof}[Proof of Corollary \ref{cc2}]
Item i) is due to \cite[Theorem 1.21]{C2}, while iii) is due to Theorem \ref{m14}. We prove item ii). Suppose that $\mathbf n\neq \mathbb{N}$ and $\mathbb{N}\backslash \mathbf n$ is finite. We use the example of a basis $\mathcal{B}$ in \cite[Section 4]{C2}. The basis $\mathcal{B}$ is $\omega$-($\mathbb N$, strong partially greedy) for some weight on sets $\omega$, but $\mathcal{B}$ is neither ($\mathbb N$, strong partially greedy) nor has a subsequence equivalent to the canonical basis of $c_0$. 

We claim that $\mathcal{B}$ is not $\zeta$-($\mathbf n$, strong partially greedy) for any weight sequence $\zeta$. Suppose otherwise. By Corollary \ref{c11}, $\mathcal{B}$ is either ($\mathbf n$, strong partially greedy) or $\mathcal{B}$ has a subsequence that is equivalent to the canonical basis of $c_0$. Since the difference set $\Delta_{\mathbb{N}, \mathbf n}$ is finite, by Theorem \ref{equiv}, $\mathcal{B}$ is either ($\mathbb N$, strong partially greedy) or $\mathcal{B}$ has a subsequence that is equivalent to the canonical basis of $c_0$. However, neither of these holds. 
\end{proof}

\begin{thm}
Let $\mathbf n\neq \mathbb{N}$. There exists a $1$-unconditional Schauder basis $\mathcal B$ that is $\omega$-($\mathbf n$, strong partially greedy) for some weight on sets $\omega$, but $\mathcal B$ is not $\zeta$-($\mathbf n$, strong partially greedy) for any weight sequence $\zeta$.
\end{thm}

\begin{proof}
Let $s_n = \frac{1}{\sqrt{n}}$ and $t_{n} = \frac{1}{n}$ for $n\ge 1$. Let $X$ be the completion of $c_{00}$ under the following norm: for $x = (x_1, x_2, \ldots)$, 
$$\|x\|\ =\ \sup_{\pi, \pi'}\left(\sum_{i=1}^\infty s_{\pi(i)}|x_{n_{2i-1}}| + \sum_{j \in  \mathbb{N}\backslash (n_{2i-1})_i} t_{\pi'(j)}|x_{j}|\right),$$
where $\pi:\mathbb{N}\rightarrow\mathbb{N}$ and $\pi': \mathbb{N}\backslash (n_{2i-1})_i\rightarrow\mathbb{N}$ are bijections. Consider the canonical basis $\mathcal{B}$, which is normalized and $1$-unconditional. Define a weight $\omega$ on each set $A\subset\mathbb{N}$ as follows:
$$\omega(A)\ =\ \begin{cases}\|1_A\| &\mbox{ if }A\mbox{ is finite,}\\ \infty &\mbox{ if }A\mbox{ is infinite}.\end{cases}$$
Since $\|1_A\| = \omega(A)$ for each finite $A\subset\mathbb{N}$, $\mathcal{B}$ is $\omega$-($\mathbf n$, conservative). Hence, we know that $\mathcal{B}$ is $\omega$-($\mathbf n$, strong partially greedy) by Theorem \ref{m20}. Besides, $\omega$ satisfies all properties listed in Remark \ref{r22}. We show that $\mathcal{B}$ is not $\zeta$-($\mathbf n$, strong partially greedy) for any weight sequence $\zeta = (s_n)_n$. Suppose otherwise. 

If either $\sup_{n} s_n = \infty$ or $\inf_{n\in\mathbf n} s_n = 0$, then Proposition \ref{p12} items ii) and iii) give that $\mathcal{B}$ has a subsequence that is equivalent to the canonical basis of $c_0$, a contradiction.

For the rest of the proof, assume that $0 < \alpha := \inf_{n\in\mathbf n} s_n \le  \sup_{n} s_n =: \beta < \infty$ and set $p = \lceil \beta/\alpha\rceil$. For each $N\in\mathbb{N}$, let $A_N = \{n_{1}, n_{3}, \ldots, n_{2N-1}\}$ and $B_N = \{n_{2N}, n_{2N+2}, \ldots, n_{2N + 2(pN-1)}\}$. Note that $A_N\subset\mathbf n$, $A_N < B_N\cap \mathbf n$, and 
$$\zeta(A_N) \ \le\ \beta N\ \le\ p N\alpha\ =\ |B|\alpha \ \le\ \zeta(B_N).$$
Hence, $(A_N, B_N)\in \mathbb{T}^{\zeta}(\mathbf n)$. However, 
$$\|1_{A_N}\|\ =\ \sum_{i=1}^N \frac{1}{\sqrt{i}}\ \ge\ 2\sqrt{N+1}-2,\mbox{ while }\|1_{B_N}\|\ =\ \sum_{i=1}^{pN} \frac{1}{i}\ \le\ \ln(pN) + 1.$$
We have $\|1_{A_N}\|/\|1_{B_N}\|\rightarrow\infty$ as $N\rightarrow\infty$. Therefore, $\mathcal{B}$ is not $\zeta$-($\mathbf n$, conservative) and thus, is not $\zeta$-($\mathbf n$, strong partially greedy). 
\end{proof}

\end{document}